\documentclass[11pt]{article}

\usepackage{graphicx}
\usepackage{subcaption}
\usepackage{longtable}
\usepackage{wrapfig}
\usepackage{rotating}
\usepackage[normalem]{ulem}
\usepackage{amsmath}
\usepackage{amssymb}
\usepackage{mathtools}
\usepackage{capt-of}
\usepackage{hyperref}
\usepackage{fullpage}
\usepackage{xcolor}
\usepackage{booktabs} 
\usepackage{bm}
\usepackage{amsthm}
\usepackage{tocloft}
\usepackage{comment}

\usepackage[backend=biber,url=false,eprint=false,maxbibnames=99]{biblatex}
\DeclareMathOperator{\real}{Re}
\DeclareMathOperator{\imag}{Im}
\DeclareMathOperator{\sign}{sign}

\newcommand{\conj}[1]{\overline{#1}}
\newcommand{\normv}{\mathrm{v}}
\newcommand{\Space}{\mathcal{X}}
\newtheorem{definition}{Definition}[section]

\newtheorem{remark}[definition]{Remark}
\newtheorem{theorem}[definition]{Theorem}

\newtheorem{proposition}[definition]{Proposition}
\newtheorem{lemma}[definition]{Lemma}

\newcommand{\be}{\begin{equation}}
\newcommand{\ee}{\end{equation}}

\renewcommand{\div}{\,{\rm div}\,}

\renewcommand{\Re}{\operatorname{Re}}

\newcommand{\RR}{\mathbb R}

\newcommand{\codenumber}[1]{#1}
\newcommand{\resultnumber}[1]{#1}

\definecolor{linkcol}{HTML}{1A4F9C}
\hypersetup{
  bookmarksopen=true,
  bookmarksdepth=subsubsection,
  colorlinks=true,
  allcolors=linkcol
}
\title{Non-uniqueness for the Complex Ginzburg--Landau Equation}
\author{Joel Dahne, Vladimír Šverák}
\date{}

\graphicspath{{figures/}}

\begin{document}

\maketitle

\begin{abstract}
  We show that singularities of the three-dimensional cubic complex
  Ginzburg--Landau equation developing from smooth initial data can
  induce non-uniqueness of solutions after the time of blowup. The
  solutions we study start as backward self-similar solutions, and the
  source of the non-uniqueness is an unstable eigenvalue of the
  linearization around the forward self-similar profile generated by
  the singularity. We establish the existence of this eigenvalue, and
  hence the non-uniqueness, by a computer-assisted proof for specific
  values of the parameters in the equation. The work is in part
  motivated by the question of what happens after a singularity of
  the Navier--Stokes equation, with which the complex Ginzburg--Landau
  equation shares several important properties, such as an energy
  inequality and the scaling symmetry. This question has become more
  pressing with the recently announced proof of finite-time blowup for
  the forced Navier--Stokes equation.
\end{abstract}

\tableofcontents

\section{Introduction}
\label{sec:introduction}

We consider the Cauchy problem for the complex Ginzburg--Landau (CGL)
equation for functions
\(u\colon \RR^{3}\times (0,\infty)\to \mathbb{C}\)
\begin{equation}\label{eq:CGL}
  \begin{split}
    i \frac{\partial u}{\partial t} + (1 - i\epsilon)\Delta u + (1 + i\delta)|u|^{2}u &= 0 \qquad\text{in }\mathbb{R}^{3} \times (0, \infty),\\
    u|_{t=0} &= u_{0},
  \end{split}
\end{equation}
where the initial condition \(u_{0}\) belongs to a suitable class of
functions to be specified later. The CGL equation is a basic model in
nonlinear physics, with applications ranging from nonlinear optics to
superconductivity and superfluidity, see the
review~\cite{Aranson2002}. We assume \(\epsilon>0\) and \(\delta\ge0\)
are given parameters. With these sign conditions, solutions
of~\eqref{eq:CGL} satisfy an energy inequality, see below. The Cauchy
problem in this setting has been studied
in~\cite{Doering1994, Ginibre1996, Ginibre1997}. Self-similar
singularities of~\eqref{eq:CGL} were found numerically
in~\cite{Plech2001} and studied using asymptotic and perturbative
methods in~\cite{Budd2005, Rottschafer2008, Rottschafer2012}. Their
existence was proved in~\cite{Dahne2024} for specific parameter
values.

In this paper we are interested in what happens \emph{after} a
singularity has formed. Our main result is that singularities
of~\eqref{eq:CGL} can lead to non-uniqueness: there are solutions that
are smooth up to the time of blowup, but that can be continued past
the singularity in more than one way.

\begin{theorem}[Informal version of Theorem~\ref{thm:main}]\label{thm-informal}
  Let \(\delta = 0\) and let \(\epsilon\) be in a small neighborhood
  of \(\codenumber{0.1681}\). There exist two distinct (locally
  suitable weak, as defined in
  Section~\ref{sec:non-uniqueness-from-unstable-eigenvalue}) solutions
  \(u_{1}\) and \(u_{2}\) of~\eqref{eq:CGL} in
  \(\RR^{3} \times [-T, T]\), for some \(T > 0\), that are smooth away
  from the point \((0, 0)\), where they blow up, and that share the
  same initial data
  \begin{equation*}
    u_{1}(\cdot, -T) = u_{2}(\cdot, -T) \in C^{\infty}(\mathbb{R}^{3}) \cap L^{3,\infty}(\mathbb{R}^{3}).
  \end{equation*}
  They coincide for \(t \leq 0\) but differ for \(t > 0\).
\end{theorem}

The initial data in Theorem~\ref{thm-informal} is smooth, but it
decays only like \(|x|^{-1}\) at infinity, so the solutions do not
have finite energy. We expect that this restriction can be removed by
a suitable truncation, see step (iii) below. The solutions are
associated with backward and forward self-similar solutions, and the
non-uniqueness is a result of an unstable eigenvalue of the
linearization around the forward self-similar solution. This is
explained in more detail in
Sections~\ref{sec:self-similar-singularities}
and~\ref{sec:non-uniqueness-transition-time}. The existence of this
eigenvalue is established through a computer-assisted proof, see
Section~\ref{sec:proof-strategy}.

An important part of our motivation comes from considering the CGL
equation as a model for certain aspects of the PDE theory for the
Navier--Stokes equation. The Cauchy problem for the Navier--Stokes
equation in \(\RR^{3}\) is
\begin{equation}\label{V-I-2}
  \begin{aligned}
    u_{t}+u\cdot\nabla u+\nabla p-\nu\Delta u &=0\\
    \div u &=0
           &\smash{\raisebox{.5\baselineskip}
             {\(\text{in }\RR^{3}\times(0,\infty),\)}}\,\\
    u|_{t=0}&=u_{0}
           &\text{in }\RR^{3},\phantom{mm}
  \end{aligned}
\end{equation}
and one of the important open problems is whether one can find a class
of fields \(u\) in which the Cauchy problem has a unique solution. If
that were not the case, the Navier--Stokes equation would have to be
considered an incomplete model, as it would call into question one of
the main points of such models, namely predicting future states from
the current state.

For both equations~\eqref{eq:CGL} and~\eqref{V-I-2} there is a theory
of Leray solutions, based on the energy identity: for \(t_{1}<t_{2}\) and
\(\delta=0\) we have
\begin{equation*}
  \int_{\RR^{3}} \frac{1}{2}|u(x,t_{2})|^{2}\,dx+\int_{t_{1}}^{t_{2}}\int_{\RR^{3}} \epsilon|\nabla u(x,t)|^{2}\,dx\,dt=\int_{\RR^{3}}\frac{1}{2}|u(x,t_{1})|^{2}\,dx
\end{equation*}
for the CGL solutions\footnote{In the case \(\delta>0\) we get an
  additional favorable term \(\delta|u|^{4}\) in the identity.} and
\begin{equation*}
  \int_{\RR^{3}} \frac{1}{2}|u(x,t_{2})|^{2}\,dx+\int_{t_{1}}^{t_{2}}\int_{\RR^{3}} \nu|\nabla u(x,t)|^{2}\,dx\,dt=\int_{\RR^{3}}\frac{1}{2}|u(x,t_{1})|^{2}\,dx
\end{equation*}
for the Navier--Stokes solutions. Based on these identities, and the
scaling symmetry \(u(x,t)\mapsto\lambda u(\lambda x,\lambda^{2}t)\),
which is also shared by the two equations, one can develop the theory
of \emph{suitable weak solutions} and their partial regularity,
following Leray~\cite{Leray1934} and later developments
in~\cite{Doering1994, Scheffer1977, Caffarelli1982, Yan1999}. However,
even for the best classes of suitable weak solutions, uniqueness is
unknown.

In reasonable classes, solutions are unique as long as they remain
smooth. For sufficiently regular initial data, a loss of uniqueness
can therefore only occur after a singularity has formed. For rough
initial data, non-uniqueness for the Navier--Stokes equation is by now
established in several settings: for weak solutions that do not
satisfy the energy inequality~\cite{Buckmaster2019}, for Leray
solutions of the equation with a suitable non-smooth
forcing~\cite{Albritton2022}, and for Leray solutions of the unforced
equation~\cite{Hou2025}, see also~\cite{Coiculescu2025}. The results
in~\cite{Albritton2022, Hou2025} build on the program proposed
in~\cite{Jia2015}, based on unstable forward self-similar solutions,
for which numerical evidence was given in~\cite{Guillod2023,
  Ionescu2026}. These results do not, however, address whether
singularities forming from smooth initial data can lead to
non-uniqueness.

Until recently, it was not known whether solutions of the
Navier--Stokes equation can develop singularities from smooth data at
all. A recently announced proof~\cite{OpenAI2026} shows that this is
the case, at least for the forced equation: there are solutions with
zero initial data and a smooth, compactly supported force that blow up
in finite time while the energy stays bounded. This means that the
question of what happens after a singularity is no longer
hypothetical. Leray solutions with the same data and force exist
globally in time and, by weak--strong uniqueness, coincide with the
smooth solution up to the time of blowup; whether they are unique
after that time is a natural question. Our result shows that for the
CGL equation a singularity forming from smooth initial data can indeed
be followed by non-uniqueness. We note, however, that the singularity
in~\cite{OpenAI2026} is of a rather different nature from the
self-similar singularities considered here, and the mechanism of
non-uniqueness studied in this paper does not directly apply to it.

The long-term goal of our work is to show that for the CGL equation
one can have smooth, compactly supported initial data \(u_{0}\) for
which singularities develop that lead to non-uniqueness. Our approach
is based on the following steps:

\begin{itemize}
\item[(i)] Establish the existence of self-similar
  singularities.\footnote{Strictly speaking, the singularities are
    self-similar only in a suitable generalized sense, in that they
    are invariant not under the pure scaling symmetries but under a
    one-parameter group of symmetries that engages both scaling and
    phase rotations.}
\item[(ii)] Show that taking the blowup profile of some of these
  singularities as an initial datum \(u_{0}\) leads to non-uniqueness.
\item[(iii)] Show that the whole situation can be preserved for a
  suitable truncation to compact support of the initial datum leading
  to the self-similar blowup.
\end{itemize}

Step (i) was carried out numerically in~\cite{Plech2001} and
rigorously in~\cite{Dahne2024}. In this paper we establish step (ii).
Our main task is to show that at least some of the blowup profiles of
the known singularities lead to situations in which the general
philosophy of~\cite{Jia2015} for the Navier--Stokes equation applies
(with some modifications) to CGL when we take that blowup profile as
the initial datum. There are a number of technicalities, discussed in
more detail below, that make this task non-trivial, but they all turn
out to be manageable. Step (iii) is left for future work, but we
expect it to be within reach of existing methods, for example
techniques similar to those in~\cite{Jia2015}. In view
of~\cite{OpenAI2026}, a forced version of step (iii), with smooth,
compactly supported initial data and forcing, is also of interest, and
it should be easier still, see Remark~\ref{rem:forced}.

We will see that the nature of the non-uniqueness is such that it
seems difficult to imagine any good selection principle that would
single out ``the right solution''. Rather, it seems that we are
dealing with a situation where the model is incomplete: to determine
what happens when uniqueness is lost, we must look into the ``next
layer'' of the underlying physics; the ``effective theory'' at the
level of CGL is by itself insufficient. As the CGL equation is not
really a fundamental model, that may not be too surprising.

For the Navier--Stokes equation, it is unclear whether the situation
could be similar. One can still hope that even if in some situations
singularities leading to non-uniqueness appear, such scenarios will be
unstable and the equation will still describe a well-defined evolution
in generic situations. By contrast, the non-uniqueness scenario for
CGL may be stable, although that point needs further investigation.

\subsection{Self-similar singularities and their continuation after blowup}
\label{sec:self-similar-singularities}

The simplest singularities of equation~\eqref{eq:CGL} would be
solutions defined on \(\RR^{3}\times(-\infty,0)\) that are invariant
under the scaling
\begin{equation*}
  u(x,t)\mapsto\lambda u(\lambda x,\lambda^{2} t),\qquad \lambda>0.
\end{equation*}
Such solutions likely do not exist, but if we combine the scaling
symmetry with the simple phase rotation symmetry
\begin{equation*}
  u(x,t)\mapsto e^{i\theta} u(x,t),\qquad \theta\in\RR,
\end{equation*}
and seek solutions that are invariant under the symmetry
\begin{equation}\label{scaling2}
  u(x,t)\mapsto\lambda^{1+i\varkappa}u(\lambda x,\lambda^{2}t),\qquad\lambda>0
\end{equation}
(with these transformations still forming a one-parameter subgroup of
the symmetry group of the equation), we can obtain the desired
objects. On \(\RR^{3}\times(-\infty,0)\), these are the
\emph{backward} self-similar solutions of the form
\begin{equation}\label{eq:CGL-backward}
  u(x, t) = \frac{1}{(-2\kappa t)^{\frac{1}{2}\left(1 + i \frac{\omega}{\kappa}\right)}}
  Q\left(\frac{|x|}{(-2\kappa t)^{\frac{1}{2}}}\right),
\end{equation}
where \(Q\colon [0,\infty)\to \mathbb{C}\) is a smooth ``profile''
with \(Q(\xi)\sim \xi^{-1 - i\frac{\omega}{\kappa}}\) as
\(\xi\to\infty\), and \(\omega\) and \(\kappa\) are positive parameters of
physical dimension \([\textup{time}]^{-1}\). (For our purposes one can
take \(\omega = 1\) without loss of generality.) Such backward
self-similar solutions blow up at the origin at \(t = 0\), and we have
\begin{equation*}
  \lim_{t\to 0_{-}} u(x,t)= p_{Q}|x|^{-1 - i\frac{\omega}{\kappa}},
\end{equation*}
for some constant \(p_{Q}\). One can try to take this limiting profile
as the initial condition for the evolution at positive times
\(t > 0\). It is natural to look for the continuation again within the
class of solutions invariant under the scaling
symmetry~\eqref{scaling2}, with \(\varkappa=\frac{\omega}{\kappa}\),
which on \(\RR^{3}\times(0,\infty)\) are the \emph{forward}
self-similar solutions of the form
\begin{equation}\label{eq:CGL-forward}
  u(x, t) = \frac{1}{(2\kappa t)^{\frac{1}{2}\left(1 + i \frac{\omega}{\kappa}\right)}}
  \hat{Q}\left(\frac{|x|}{(2\kappa t)^{\frac{1}{2}}}\right),
\end{equation}
for a suitable smooth profile
\(\hat{Q}\colon[0,\infty)\to\mathbb{C}\). For the forward solution to
continue the backward one, it has to satisfy the matching condition
\begin{equation*}
  \lim_{t\to 0_{+}} u(x,t)= p_{Q}|x|^{-1 - i\frac{\omega}{\kappa}}.
\end{equation*}
Note that the partial regularity theory implies that the solutions of
CGL that are radially symmetric (i.e., depend only on \(|x|\)) are
smooth away from the origin. In particular, the solutions we seek here
will be smooth away from the origin, including at the ``gluing time''
\(t = 0\). At the same time, to allow non-uniqueness, they cannot be
analytic in time at \(t=0\).\footnote{These requirements are, of
  course, compatible for PDEs of parabolic type, and in our case they
  are manifested by the profile \(Q\) having an asymptotic expansion
  at infinity that is not convergent.}

Our first task is then to construct an entire solution
\(u\colon\RR^{3}\times(-\infty,\infty)\to\mathbb{C}\) of the CGL
equation that is smooth away from the origin \((x,t)=(0,0)\) such that
\begin{equation}\label{entire}
  u(x,t)=
  \begin{dcases}
    \frac{1}{(-2\kappa t)^{\frac{1}{2}\left(1+i\frac{\omega}{\kappa}\right)}}
    Q\left(\frac{|x|}{(-2\kappa t)^{\frac{1}{2}}}\right)
    & t<0,
    \\
    p_{Q}|x|^{-1-i\frac{\omega}{\kappa}}
    & t=0,
    \\
    \frac{1}{(2\kappa t)^{\frac{1}{2}\left(1+i\frac{\omega}{\kappa}\right)}}
    \hat{Q}\left(\frac{|x|}{(2\kappa t)^{\frac{1}{2}}}\right)
    & t>0.
  \end{dcases}
\end{equation}
This solution is invariant under the scaling~\eqref{scaling2} with
\(\varkappa=\frac{\omega}{\kappa}\). The profiles \(Q\) and
\(\hat{Q}\) have asymptotic expansions
\begin{equation}\label{bcinfty}
  \begin{split}
    Q(\xi) & =\xi^{-1-i\frac{\omega}{\kappa}}\left(p_{Q}+a_{2}\xi^{-2}+a_{4}\xi^{-4}+\cdots\right),
    \\
    \hat{Q}(\xi) & =\xi^{-1-i\frac{\omega}{\kappa}}\left(p_{Q}-a_{2}\xi^{-2}+a_{4}\xi^{-4}-\cdots\right),
  \end{split}
\end{equation}
for some constants \(a_{2}, a_{4}, \ldots \in \mathbb{C}\).
Restricted to \(\RR^{3} \times [-T, T]\), the solution \(u\)
in~\eqref{entire} is the solution \(u_{1}\) in
Theorem~\ref{thm-informal}.

\subsection{Non-uniqueness at the ``transition time'' \(t=0\)}
\label{sec:non-uniqueness-transition-time}

Introducing the (forward) self-similarity variables
\begin{equation*}
  \xi=\frac{|x|}{\sqrt{2\kappa t}},\qquad \tau=\log(2\kappa t)
\end{equation*}
and seeking \(u(x,t)\) for \(t>0\) as
\begin{equation*}
  u(x, t) = \frac{1}{(2\kappa t)^{\frac{1}{2}\left(1 + i \frac{\omega}{\kappa}\right)}}
  \hat{U}(\xi, \tau),
\end{equation*}
we obtain
\begin{equation}\label{eq:hatU}
  2\kappa\partial_{\tau} \hat{U}
  = (\epsilon + i)\Delta\hat{U}
  + \kappa\xi\partial_{\xi}\hat{U}
  + \left(\kappa + i \omega\right)\hat{U}
  + (i - \delta)|\hat{U}|^{2}\hat{U},
\end{equation}
with a stationary solution given by \(\hat{U}=\hat{Q}\). The potential
for non-uniqueness can arise either from the possibility of this
equation having multiple stationary solutions (with the correct
``boundary condition'' at infinity related to~\eqref{bcinfty}), or
from an unstable manifold of the steady state~\(\hat{Q}\), along which
the solutions of~\eqref{eq:hatU} can ``separate'' from \(\hat{Q}\) as
\(\tau\) runs from \(-\infty\) to finite values. In the former case we
would have two distinct continuations of \(u(x,t)\) from \(t<0\) into
entire solutions invariant under the scaling~\eqref{scaling2}, whereas the
latter case can include more complicated dynamics. Writing
\begin{equation*}
  \hat{U}(\xi, \tau) = \hat{Q}(\xi) + \phi(\xi, \tau),
\end{equation*}
we obtain from~\eqref{eq:hatU} an equation for \(\phi\) of the form
\begin{equation}\label{eq:perturbation-equation}
  \partial_{\tau}\phi=L_{\hat{Q}}\phi+\mathcal{N}(\phi),
\end{equation}
where \(L_{\hat{Q}}\) is the linearization of~\eqref{eq:hatU} at
\(\hat{Q}\), and \(\mathcal{N}(\phi)\) is the higher-order (in
\(\phi\)) remainder, which also depends on \(\hat{Q}\), although we do
not indicate this dependence explicitly. We note that \(L_{\hat{Q}}\)
is only \(\RR\)-linear, not \(\mathbb{C}\)-linear. We expect the
dynamics on the unstable manifold of \(\hat{Q}\) that leads to the
appearance of the second solution in the original \(t\) variable to be
dominated by the linear equation
\begin{equation*}
  \partial_{\tau}\phi=L_{\hat{Q}}\phi,
\end{equation*}
and hence it is natural to focus on the spectral properties of
\(L_{\hat{Q}}\). Since \(L_{\hat{Q}}\) is not \(\mathbb{C}\)-linear,
this requires some care. Assuming the natural space for \(\hat{U}\) is
a suitable space of complex-valued functions \(X\), the natural space
for the spectral theory for \(L_{\hat{Q}}\) is the space
\(X_{\RR}\otimes_{\RR}\mathbb{C}\cong X\oplus \bar{X}\), the
complexification of the real space obtained from \(X\) by restricting
the scalars to \(\RR\). If we can find an eigenvalue
\(\lambda\in \mathbb{C}\) of \(L_{\hat{Q}}\) (or its natural extension
to \(X\oplus\bar{X}\)) with \(\Re\lambda>0\) and an eigenfunction
\(\phi\in X\oplus \bar{X}\) such that \(L_{\hat{Q}}\phi=\lambda\phi\),
we expect that the unstable manifold of \(\hat{Q}\) will contain
solutions
\begin{equation}\label{V-I-8}
  \hat{U}(\xi,\tau)=\hat{Q}(\xi) + \Re\left(ce^{\lambda \tau}\phi\right)+\text{h.o.t.},
\end{equation}
where \(c\in \mathbb{C}\), h.o.t.\ denotes higher-order terms and
\(\Re\) denotes the real part with respect to the conjugation
\(x\otimes z\mapsto x\otimes\bar{z}\) of the complexification. The
real part takes values in the real form \(X_{\RR}\), which we identify
with \(X\). Going back to the original variables, we expect to have
other solutions \(v\) defined for \(t\in(-\infty,T)\) for some \(T>0\)
(and parametrized by \(c\in \mathbb{C}\)), coinciding with \(u\) given
by~\eqref{entire} for \(t<0\), with
\begin{equation}\label{secondsol}
  v(x,t)=
  \begin{dcases}
    \frac{1}{(-2\kappa t)^{\frac{1}{2}\left(1+i\frac{\omega}{\kappa}\right)}}
    Q\left(\frac{|x|}{(-2\kappa t)^{\frac{1}{2}}}\right)
    & t<0,
    \\
    p_{Q}|x|^{-1-i\frac{\omega}{\kappa}}
    & t=0,
    \\
    \frac{1}{(2\kappa t)^{\frac{1}{2}\left(1+i\frac{\omega}{\kappa}\right)}}\left(
    \hat{Q}\left(\frac{|x|}{(2\kappa t)^{\frac{1}{2}}}\right)
    + \Re\left(ce^{\lambda\tau}\phi(\xi)\right)+\text{h.o.t.}
    \right)
    & 0<t<T.
  \end{dcases}
\end{equation}
For a suitable choice of \(c \neq 0\), the restriction of \(v\)
in~\eqref{secondsol} to \(\RR^{3} \times [-T, T]\) is the solution
\(u_{2}\) in Theorem~\ref{thm-informal}.

\begin{remark}\label{rem:forced}
  For the forced version of step (iii), one could proceed as follows.
  Let \(\eta(x)\) be a smooth, radial cutoff function that is \(1\)
  in the unit ball and \(0\) outside the ball of radius \(2\), and
  consider \(u_{\rho}(x,t)=\eta(\rho x)u(x,t)\) and
  \(v_{\rho}(x,t)=\eta(\rho x)v(x,t)\). These satisfy~\eqref{eq:CGL}
  with the same smooth, compactly supported initial data and with
  smooth, compactly supported right-hand sides \(f_{\rho}\) and
  \(g_{\rho}\), respectively, such that \(f_{\rho}-g_{\rho}\) vanishes for
  \(t\leq 0\). We conjecture that, with minor adjustments, one can
  also prove that \(f_{\rho}-g_{\rho}\) is exponentially small in
  \(1/\rho\) as \(\rho\to 0\), uniformly for \(t\in(0,T)\), due to
  the exponential decay of the eigenfunction \(\phi\). One can then
  expect to be able to perturb one of the solutions so that it
  produces the same right-hand side as the other, while preserving
  the non-uniqueness.
\end{remark}

\subsection{Dependence on the parameters}
\label{sec:parameter-dependence}
The non-uniqueness mechanism in
Section~\ref{sec:non-uniqueness-transition-time} requires
\(L_{\hat{Q}}\) to have an unstable eigenvalue, and whether this holds
depends on the parameters. For fixed \(\epsilon\) (with \(\delta = 0\)
and \(\omega = 1\)), one finds numerically that backward self-similar
solutions of the form~\eqref{eq:CGL-backward} exist only for a
discrete set of \(\kappa\) values. In the \((\epsilon, \kappa)\)-plane
these form curves, or \emph{branches}, which are studied in more
detail in~\cite{Plech2001, Dahne2024}. A numerical approximation
of the first such branch is given in Figure~\ref{fig:branch}.

While one expects that every backward self-similar solution can be
continued by a forward one, giving an entire solution of the
form~\eqref{entire}, one does not expect the associated linear
operator \(L_{\hat{Q}}\) to always have an unstable eigenvalue.
Figure~\ref{fig:branch} shows the results of a conventional numerical
computation of the eigenvalue with the largest real part along the branch.
It indicates that there is an unstable eigenvalue on the first part of
the branch, which starts at \(\epsilon = 0\), but that all eigenvalues
become stable before the branch turns around in \(\epsilon\). We have
rigorously verified the existence of an unstable eigenvalue for a
small segment of the branch around the point marked by the black dot
in Figure~\ref{fig:branch}, but have not attempted a proof that would
cover larger segments.

The value \(\epsilon = \codenumber{0.1681}\) used in
Theorem~\ref{thm-informal} is chosen to be close to the stability
threshold. Further along the branch, the mechanism of non-uniqueness
we use is no longer present, and it is conceivable that one has
uniqueness. Closer to the beginning of the branch, we still expect
non-uniqueness, but the rigorous verification becomes computationally
more subtle due to slower decay at infinity of the eigenfunctions.

\begin{figure}
  \centering
  \begin{subfigure}[t]{0.45\textwidth}
    \includegraphics[width=\textwidth]{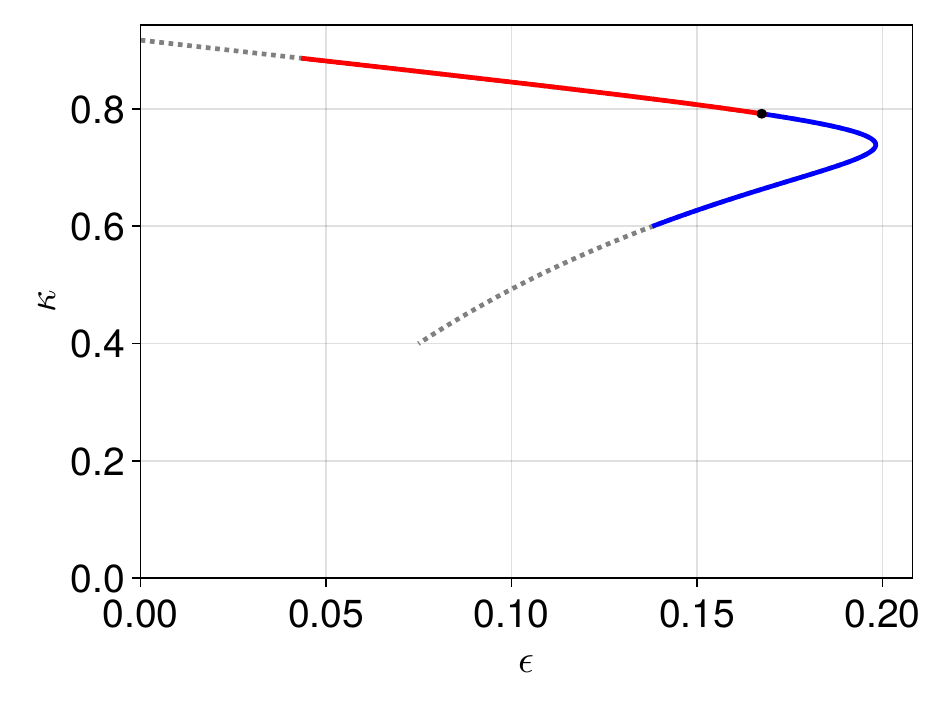}
  \end{subfigure}
  \hspace{0.05\textwidth}
  \begin{subfigure}[t]{0.45\textwidth}
    \includegraphics[width=\textwidth]{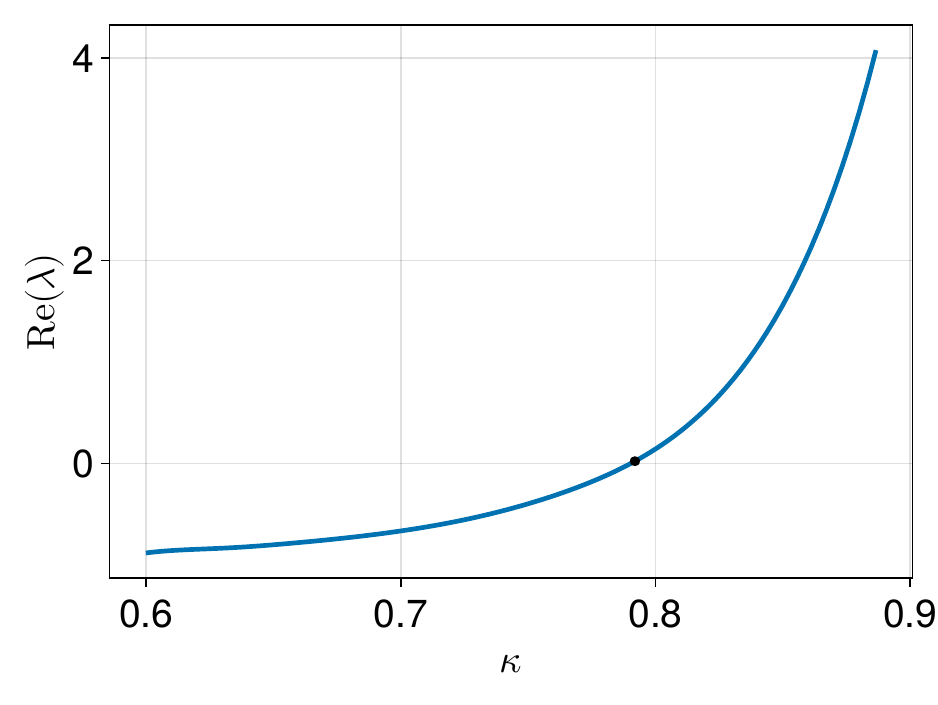}
  \end{subfigure}
  \caption{Numerical results for the first branch of backward
    self-similar solutions studied in~\cite{Plech2001, Dahne2024}. The
    left figure shows an approximation of the branch itself, in the
    \((\epsilon, \kappa)\)-plane. Along the solid part of the branch
    we have computed the eigenvalue of \(L_{\hat{Q}}\) with the
    largest real part. Red indicates that this real part is positive
    and blue that it is negative, whereas along the dashed parts the
    eigenvalue was not computed. The right figure shows the real part
    of this eigenvalue as a function of \(\kappa\). In both figures,
    the black dot marks the point on the branch for which
    Theorem~\ref{thm:unstable-eigenvalue} gives a rigorous enclosure
    of the unstable eigenvalue. The branch was computed using the same
    methodology as in~\cite{Dahne2024} and the eigenvalue using the
    finite-difference approximation from
    Appendix~\ref{sec:finite-differences-approximations}.}
  \label{fig:branch}
\end{figure}

\subsection{Proof strategy and organization of the paper}
\label{sec:proof-strategy}

The proof of Theorem~\ref{thm:main} is split into two main parts:
proving that an unstable eigenvalue of the linear operator
\(L_{\hat{Q}}\) implies the non-uniqueness, and proving the existence
of such an eigenvalue.

The first part is handled in
Section~\ref{sec:non-uniqueness-from-unstable-eigenvalue} and is based
on well-established spectral and semigroup theory. Our strategy
follows the approach of Jia and Šverák~\cite{Jia2015}, which was
subsequently also used in~\cite{Albritton2022, Hou2025}.

From Section~\ref{sec:backward-forward-linearized-solutions} onward,
the paper is devoted to the significantly more challenging second
part, proving the existence of an unstable eigenvalue in the first
place. For this, we follow the strategy developed in~\cite{Dahne2024},
where the existence of backward self-similar solutions to the CGL
equation was first proved. It is based on a shooting method inspired
by~\cite{Plech2001}. The backward and forward
self-similar solutions, as well as the eigenfunctions of
\(L_{\hat{Q}}\), all satisfy ODEs posed on the interval
\([0, \infty)\), with appropriate boundary conditions at zero and
infinity. The idea is to construct two families of solutions, one
satisfying the boundary conditions at zero and the other satisfying
the boundary conditions at infinity. We then need to show that, with
the right choice of parameters, two solutions, one from each family,
match at an intermediate point, forming a global solution satisfying
both boundary conditions. This matching condition is verified through
a computer-assisted proof.

The computer-assisted parts of the proof are based on rigorous
numerics: numerical computations in which every source of error,
including rounding and discretization, is accounted for, in our case
by means of interval arithmetic. We refer to~\cite{Tucker2011} for an
introduction to such methods and to~\cite{GomezSerrano2019} for a
survey of their applications to PDEs. Such methods have recently
played a central role in several results on singularity formation,
for example for the 3D Euler equations~\cite{Chen2022blowup3deuler,
  Chen2023blowup3dEulerII}, compressible fluids~\cite{Buckmaster2022}
and the defocusing nonlinear wave equation~\cite{Cadiot2026}, as well
as for the existence of an unstable eigenvalue for the Navier--Stokes
equation~\cite{Hou2025}.

More details about the shooting method, as well as the precise
equations under consideration, are given in
Section~\ref{sec:backward-forward-linearized-solutions}. The
computer-assisted proof itself is presented in
Section~\ref{sec:existence-unstable-eigenvalue}. It is split into
three parts: (i) Theorem~\ref{thm:backward-existence} proving the
existence of a backward self-similar solution \(Q\), (ii)
Theorem~\ref{thm:forward-existence} proving the existence of a forward
self-similar solution with matching asymptotic behavior and (iii)
Theorem~\ref{thm:unstable-eigenvalue} proving the existence of an
unstable eigenvalue for the associated linear operator
\(L_{\hat{Q}}\).

The proofs rely on estimates that are developed in the remaining
sections of the paper. Section~\ref{sec:forward-solution-infinity}
gives estimates for the solution to the forward self-similar equation
satisfying the boundary conditions at infinity, and
Section~\ref{sec:forward-solution-zero} treats the solution satisfying
the boundary conditions at zero. In the same spirit,
Section~\ref{sec:linearized-solution-infinity} handles solutions to
the linearized equation satisfying the boundary conditions at
infinity, and Section~\ref{sec:linearized-solution-zero} handles those
satisfying the boundary conditions at zero. These sections build upon
the work in~\cite{Dahne2024} and generally follow the same outline.
While some of the estimates from~\cite{Dahne2024} can be reused, in
particular those related to confluent hypergeometric functions, most
of them require changes to take into account the different behavior in
the equations. To simplify comparison with~\cite{Dahne2024}, in these
sections we opt to work with general parameters for the strength of
the nonlinearity, \(\sigma\), and the spatial dimension, \(d\), for
which the CGL equation takes the form
\begin{equation*}
  i \frac{\partial u}{\partial t} + (1 - i\epsilon)\Delta u + (1 + i\delta)|u|^{2\sigma}u = 0 \qquad\text{in }\mathbb{R}^{d} \times (0, \infty).
\end{equation*}

Finally, in Section~\ref{sec:implementation-details} we discuss some
of the implementation details for the computer-assisted proof, in
Appendix~\ref{sec:finite-differences-approximations} we present the
finite-difference method used to compute initial approximations of the
eigenvalue and in Appendix~\ref{sec:improvements-backward} we
establish bounds related to the asymptotic behavior of the backward
self-similar profile.

\begin{remark}
  In this paper, tight intervals are generally written in
  subsuperscript notation (e.g.,
  \([1.147721, 1.147734] = 1.1477_{21}^{34}\)), while wide intervals
  are written in standard endpoint notation (e.g.,
  \([-2.4, 10.2]\)).
\end{remark}

\section{Non-uniqueness from an unstable eigenvalue}
\label{sec:non-uniqueness-from-unstable-eigenvalue}

In this section we prove the non-uniqueness assuming the existence of
an unstable eigenvalue. The arguments largely follow the approach
established by Jia and Šverák~\cite{Jia2015} and make use of the
theory of semigroups for linear evolution equations, for which we
refer to the book by Engel and Nagel~\cite{Engel2000}. We also refer
to the work by Hou, Wang and Yang~\cite{Hou2025} as well as that by
Ionescu, Jia and Palasek~\cite{Ionescu2026}.

The idea is to show that, after rescaling, the semigroup generated by
\(L_{\hat{Q}}\) is hyperbolic. This allows us to decompose the space
into stable and unstable subspaces. We then write the equation in
Duhamel form, integrating each component in the direction in which the
semigroup contracts. This yields a Lyapunov--Perron operator which is
shown to be a contraction on an exponentially weighted space. Its
fixed point yields the ancient solution that gives rise to the
non-uniqueness.

To apply semigroup theory, it is easier to work with functions on
\(\mathbb{R}^{3}\) rather than with the radial component directly. For
this section we therefore work with the self-similar
variables
\begin{equation*}
  \bm{\xi} = \frac{x}{\sqrt{2\kappa t}},\quad \tau = \log(2\kappa t)
\end{equation*}
and perform the analysis in the space \(H^{2}(\mathbb{R}^{3})\).
In this case we get that the linear operator \(L_{\hat{Q}}\) and the
nonlinearity \(\mathcal{N}\) in
Equation~\eqref{eq:perturbation-equation} take the form
\begin{equation}\label{eq:L_Q-hat-R3}
  L_{\hat{Q}}\phi = \frac{1}{2\kappa}\left[
  (\epsilon I + J)\Delta\phi
  + \kappa(\bm{\xi} \cdot \nabla\phi)
  + (\kappa I + \omega J)\phi
  + (-\delta I + J)((\hat{Q}\cdot\hat{Q})I + 2\hat{Q}\hat{Q}^{T})\phi
  \right]
\end{equation}
and
\begin{equation}\label{eq:N-R3}
  \mathcal{N}(\phi) = \frac{1}{2\kappa}(-\delta I + J)
  \left(2(\hat{Q} \cdot \phi)\phi + (\phi\cdot\phi)(\hat{Q} + \phi)\right).
\end{equation}
Here we identify \(\mathbb{C}\) with \(\mathbb{R}^{2}\), so that
\(\hat{Q}\) and \(\phi\) are \(\mathbb{R}^{2}\)-valued, and
\(J = \begin{pmatrix} 0 & -1 \\ 1 & 0 \end{pmatrix}\) is the matrix
corresponding to multiplication by \(i\). The only assumption we will
need for the profile \(\hat{Q}\) is that it lies in
\(W^{2,\infty}(\mathbb{R}^{3})\), so that multiplication by
\((\hat{Q}\cdot\hat{Q})I + 2\hat{Q}\hat{Q}^{T}\) acts as a bounded
linear operator on \(H^{2}(\mathbb{R}^{3})\). In particular, this is
the case if
\begin{equation}\label{eq:Q-hat-decay-assumption}
  \partial^{\alpha}\hat{Q} \in L^{\infty}(\mathbb{R}^{3})
  \quad\text{and}\quad
  \partial^{\alpha}\hat{Q}(\bm{\xi}) \to 0 \text{ as } |\bm{\xi}| \to \infty,
\end{equation}
for every multi-index \(\alpha\) with \(|\alpha| \leq 2\), which is
indeed satisfied for the forward self-similar solution under
consideration.

The first step is to control the essential spectrum of
\(L_{\hat{Q}}\). For this, it is natural to view \(L_{\hat{Q}}\) as a
relatively compact perturbation of the operator
\begin{equation*}
  L_{0}\phi = \frac{1}{2\kappa}\left[
    (\epsilon I + J)\Delta\phi
    + \kappa(\bm{\xi} \cdot \nabla\phi)
    + (\kappa I + \omega J)\phi
  \right].
\end{equation*}
As one can check directly, the family
\begin{equation}\label{eq:L_0-semigroup}
  (S(\tau)\phi)(\bm{\xi})
  = e^{\frac{\tau}{2\kappa}(\kappa I + \omega J)}
  \left(e^{\frac{e^{\tau} - 1}{2\kappa}(\epsilon I + J)\Delta}\phi\right)\!\left(e^{\frac{\tau}{2}}\bm{\xi}\right),
  \qquad \tau \geq 0,
\end{equation}
defines a \(C_{0}\)-semigroup on \(H^{2}(\mathbb{R}^{3})\) that leaves
the Schwartz class \(\mathcal{S}(\mathbb{R}^{3})\) invariant and whose
generator acts as \(L_{0}\) on \(\mathcal{S}(\mathbb{R}^{3})\). We
denote this generator again by \(L_{0}\) and write \(D(L_{0})\) for
its domain. Since \(\mathcal{S}(\mathbb{R}^{3})\) is a dense subspace
of \(H^{2}(\mathbb{R}^{3})\) that is contained in \(D(L_{0})\) and
invariant under the semigroup, it is a core for \(L_{0}\)
by~\cite[Chapter~II, Proposition~1.7]{Engel2000}. The following lemma
gives us control over the spectrum of \(L_{0}\) as well as its
resolvent.

\begin{lemma}\label{lemma:L_0-spectrum}
  The spectrum of \(L_{0}\) is contained in the set
  \(\left\{\lambda \in \mathbb{C} \colon \real \lambda \leq
    -1/4\right\}\). Moreover, for \(\real \lambda > -1/4\), the
  resolvent satisfies
  \begin{equation}\label{eq:L_0-resolvent-bound}
    \|(\lambda - L_{0})^{-1}\|_{\mathcal{L}(H^{2}(\mathbb{R}^{3}))} \lesssim (\real \lambda + 1/4)^{-1},
  \end{equation}
  and
  \begin{equation}\label{eq:L_0-resolvent-smoothing}
    \|(\lambda - L_{0})^{-1}\|_{\mathcal{L}(H^{2}(\mathbb{R}^{3}), H^{3}(\mathbb{R}^{3}))}
    \lesssim (\real \lambda + 1/4)^{-1/2} + (\real \lambda + 1/4)^{-1}.
  \end{equation}
\end{lemma}

\begin{proof}
  Let \(\langle \cdot, \cdot\rangle\) denote the inner product on
  \(L^{2}(\mathbb{R}^{3})\). We first record the dissipativity estimate
  in \(L^{2}(\mathbb{R}^{3})\), for \(\phi\) in the Schwartz class
  \(\mathcal{S}(\mathbb{R}^{3})\). For the diffusion term we have
  \begin{equation*}
    \begin{split}
      \real \int_{\mathbb{R}^{3}} (\epsilon I + J)\Delta\phi \cdot \conj{\phi}\,d\bm{\xi}
      &= -\real \int_{\mathbb{R}^{3}} (\epsilon I + J)\nabla\phi \cdot \conj{\nabla\phi}\,d\bm{\xi}\\
      &= -\epsilon \int_{\mathbb{R}^{3}} \nabla\phi \cdot \conj{\nabla\phi}\,d\bm{\xi}
      = -\epsilon\|\nabla\phi\|_{L^{2}}^{2},
    \end{split}
  \end{equation*}
  where we have used that \(J \nabla \phi \cdot \conj{\nabla\phi}\) is
  purely imaginary, since \(J\) is skew-symmetric. For the drift term
  we get
  \begin{equation*}
    \real \int_{\mathbb{R}^{3}} (\bm{\xi} \cdot \nabla\phi) \cdot \conj{\phi}\,d\bm{\xi}
    = \frac{1}{2}\int_{\mathbb{R}^{3}} \bm{\xi} \cdot \nabla |\phi|^{2} \,d\bm{\xi}
    = \frac{1}{2}\int_{\mathbb{R}^{3}} \nabla \cdot (\bm{\xi}|\phi|^{2}) \,d\bm{\xi}
    - \frac{3}{2}\int_{\mathbb{R}^{3}} |\phi|^{2} \,d\bm{\xi}
    = -\frac{3}{2}\|\phi\|_{L^{2}}^{2},
  \end{equation*}
  where the first integral in the second-to-last expression vanishes by
  the divergence theorem and the rapid decay of \(\phi\). For the last
  term, we get
  \begin{equation*}
    \real \int_{\mathbb{R}^{3}} (\kappa I + \omega J)\phi \cdot \conj{\phi}\,d\bm{\xi}
    = \kappa\|\phi\|_{L^{2}}^{2}.
  \end{equation*}
  Combining these, we get
  \begin{equation}\label{eq:L_0-dissipativity-L2}
    \real \langle L_{0}\phi, \phi \rangle
    = \frac{1}{2\kappa}\left[-\epsilon\|\nabla\phi\|_{L^{2}}^{2}
    - \frac{3\kappa}{2}\|\phi\|_{L^{2}}^{2}
    + \kappa\|\phi\|_{L^{2}}^{2}\right]
    = -\frac{\epsilon}{2\kappa}\|\nabla\phi\|_{L^{2}}^{2} - \frac{1}{4}\|\phi\|_{L^{2}}^{2}.
  \end{equation}

  To pass the dissipativity estimates from \(L^{2}(\mathbb{R}^{3})\)
  to \(H^{2}(\mathbb{R}^{3})\), we use that \(L_{0}\) almost commutes
  with derivatives. Both \(\Delta\) and the zeroth-order term commute
  with \(\partial^{\alpha}\), whereas
  \(\partial^{\alpha}(\bm{\xi} \cdot \nabla\phi) = \bm{\xi} \cdot
  \nabla\partial^{\alpha}\phi + |\alpha|\partial^{\alpha}\phi\), so
  that
  \begin{equation}\label{eq:L_0-commutator}
    \partial^{\alpha}L_{0}\phi = L_{0}\partial^{\alpha}\phi + \frac{|\alpha|}{2}\partial^{\alpha}\phi.
  \end{equation}
  Each derivative thus costs \(+\frac{1}{2}\), reflecting that the
  drift \(\bm{\xi} \cdot \nabla\) is expanding.
  Combining~\eqref{eq:L_0-commutator} with~\eqref{eq:L_0-dissipativity-L2}
  applied to \(\partial^{\alpha}\phi\), the standard inner product on
  \(H^{2}(\mathbb{R}^{3})\) gives
  \begin{equation*}
    \real \sum_{|\alpha| \leq 2} \langle \partial^{\alpha}L_{0}\phi, \partial^{\alpha}\phi\rangle
    = \sum_{|\alpha| \leq 2}\left[
      -\frac{\epsilon}{2\kappa}\|\nabla\partial^{\alpha}\phi\|_{L^{2}}^{2}
      + \left(\frac{|\alpha|}{2} - \frac{1}{4}\right)\|\partial^{\alpha}\phi\|_{L^{2}}^{2}
    \right],
  \end{equation*}
  in which the terms with \(|\alpha| = 2\) carry the positive
  coefficient \(\frac{3}{4}\), and hence we cannot extract a negative
  sign. The standard inner product is therefore not suitable, and we
  instead use the weighted one
  \begin{equation*}
    \langle u, v\rangle_{*} = \sum_{|\alpha| \leq 2} c_{|\alpha|}\langle \partial^{\alpha}u, \partial^{\alpha}v\rangle,
    \qquad c_{0} = 1,\quad c_{1} = \frac{\epsilon}{\kappa},\quad c_{2} = \frac{\epsilon^{2}}{4\kappa^{2}},
  \end{equation*}
  whose norm \(\|\cdot\|_{*}\) is equivalent to
  \(\|\cdot\|_{H^{2}(\mathbb{R}^{3})}\). Letting
  \(N_{j} = \sum_{|\alpha| = j}\|\partial^{\alpha}\phi\|_{L^{2}}^{2}\)
  and
  \(M_{j} = \sum_{|\alpha| =
    j}\|\nabla\partial^{\alpha}\phi\|_{L^{2}}^{2}\), we obtain
  \begin{equation*}
    \real \langle L_{0}\phi, \phi\rangle_{*}
    = \sum_{j = 0}^{2}c_{j}\left(-\frac{\epsilon}{2\kappa} M_{j} + \left(\frac{j}{2} - \frac{1}{4}\right)N_{j}\right).
  \end{equation*}
  Since all terms in \(N_{j + 1}\) are also included in \(M_{j}\), we
  have \(N_{j + 1} \leq M_{j}\), giving us
  \begin{equation}\label{eq:L_0-dissipativity-H2}
    \begin{split}
      \real \langle L_{0}\phi, \phi\rangle_{*}
      &\leq \sum_{j = 0}^{2}c_{j}\left(-\frac{\epsilon}{2\kappa} N_{j + 1} + \left(\frac{j}{2} - \frac{1}{4}\right)N_{j}\right)\\
      &= -\frac{1}{4}N_{0}
        + \left(\frac{c_{1}}{4} - \frac{\epsilon}{2\kappa} c_{0}\right)N_{1}
        + \left(\frac{3c_{2}}{4} - \frac{\epsilon}{2\kappa} c_{1}\right)N_{2}
        - \frac{\epsilon}{2\kappa} c_{2} N_{3}\\
      &= -\frac{1}{4}N_{0} - \frac{\epsilon}{4\kappa}N_{1} - \frac{5\epsilon^{2}}{16\kappa^{2}}N_{2} - \frac{\epsilon^{3}}{8\kappa^{3}}N_{3}
        \leq -\frac{1}{4}\|\phi\|_{*}^{2} - \frac{\epsilon^{3}}{8\kappa^{3}} N_{3}.
    \end{split}
  \end{equation}

  Now let \(\phi \in \mathcal{S}(\mathbb{R}^{3})\). Since
  \(\mathcal{S}(\mathbb{R}^{3})\) is invariant under the semigroup
  \(S(\tau)\), the estimate~\eqref{eq:L_0-dissipativity-H2} gives us
  \begin{equation*}
    \frac{d}{d\tau}\|S(\tau)\phi\|_{*}^{2}
    = 2\real\langle L_{0}S(\tau)\phi, S(\tau)\phi \rangle_{*}
    \leq -\frac{1}{2}\|S(\tau)\phi\|_{*}^{2}.
  \end{equation*}
  By Grönwall's inequality, we hence have
  \(\|S(\tau)\phi\|_{*} \leq e^{-\frac{\tau}{4}}\|\phi\|_{*}\). Since
  \(\mathcal{S}(\mathbb{R}^{3})\) is dense in
  \(H^{2}(\mathbb{R}^{3})\) and \(S(\tau)\) is a bounded operator on
  \(H^{2}(\mathbb{R}^{3})\), this bound extends to every
  \(\phi \in H^{2}(\mathbb{R}^{3})\). Since \(\|\cdot\|_{*}\) and
  \(\|\cdot\|_{H^{2}(\mathbb{R}^{3})}\) are equivalent, this gives
  \(\|S(\tau)\|_{\mathcal{L}(H^{2}(\mathbb{R}^{3}))} \lesssim
  e^{-\frac{\tau}{4}}\). The first two claims now follow
  from~\cite[Chapter~II, Theorem~1.10]{Engel2000}.

  It remains to prove~\eqref{eq:L_0-resolvent-smoothing}. Fix
  \(\lambda\) with \(\real \lambda > -\frac{1}{4}\) and let
  \(\phi \in \mathcal{S}(\mathbb{R}^{3})\).
  Then~\eqref{eq:L_0-dissipativity-H2} and Cauchy--Schwarz give
  \begin{equation}\label{eq:L_0-smoothing-basic}
    (\real \lambda + 1/4)\|\phi\|_{*}^{2} + \frac{\epsilon^{3}}{8\kappa^{3}} N_{3}
    \leq \real \lambda \|\phi\|_{*}^{2} - \real\langle L_{0}\phi, \phi\rangle_{*}
    = \real\langle (\lambda - L_{0})\phi, \phi\rangle_{*}
    \leq \|(\lambda - L_{0})\phi\|_{*}\|\phi\|_{*}.
  \end{equation}
  Discarding the second term on the left yields
  \(\|\phi\|_{*} \leq (\real \lambda + 1/4)^{-1}\|(\lambda -
  L_{0})\phi\|_{*}\), and inserting this back
  into~\eqref{eq:L_0-smoothing-basic} gives
  \(\frac{\epsilon^{3}}{8\kappa^{3}} N_{3} \leq (\real \lambda +
  1/4)^{-1}\|(\lambda - L_{0})\phi\|_{*}^{2}\). Since
  \(\|\phi\|_{H^{3}(\mathbb{R}^{3})}^{2} \eqsim
  \|\phi\|_{H^{2}(\mathbb{R}^{3})}^{2} + N_{3}\) and \(\|\cdot\|_{*}\)
  is equivalent to \(\|\cdot\|_{H^{2}(\mathbb{R}^{3})}\), we conclude
  \begin{equation}\label{eq:L_0-smoothing-schwartz}
    \|\phi\|_{H^{3}(\mathbb{R}^{3})}
    \lesssim \left((\real \lambda + 1/4)^{-1/2} + (\real \lambda + 1/4)^{-1}\right)
    \|(\lambda - L_{0})\phi\|_{H^{2}(\mathbb{R}^{3})}
  \end{equation}
  for all \(\phi \in \mathcal{S}(\mathbb{R}^{3})\).

  Finally, we extend~\eqref{eq:L_0-smoothing-schwartz} to \(D(L_{0})\).
  Let \(\phi \in D(L_{0})\) and, using that
  \(\mathcal{S}(\mathbb{R}^{3})\) is a core for \(L_{0}\), choose
  \(\phi_{n} \in \mathcal{S}(\mathbb{R}^{3})\) with
  \(\phi_{n} \to \phi\) and \(L_{0}\phi_{n} \to L_{0}\phi\) in
  \(H^{2}(\mathbb{R}^{3})\). Then \((\lambda - L_{0})\phi_{n}\) is
  Cauchy in \(H^{2}(\mathbb{R}^{3})\),
  so~\eqref{eq:L_0-smoothing-schwartz} applied to
  \(\phi_{n} - \phi_{m}\) shows that \((\phi_{n})_{n}\) is Cauchy in
  \(H^{3}(\mathbb{R}^{3})\). Its limit must be \(\phi\), whence
  \(\phi \in H^{3}(\mathbb{R}^{3})\)
  and~\eqref{eq:L_0-smoothing-schwartz} passes to the limit. In
  particular, \(D(L_{0}) \subseteq H^{3}(\mathbb{R}^{3})\)
  and~\eqref{eq:L_0-resolvent-smoothing} follows.
\end{proof}

Everything we need about the spectrum of \(L_{\hat{Q}}\) follows from
the fact that it differs from \(L_{0}\) by a perturbation \(K\) that
is well behaved relative to \(L_{0}\). We collect the required
properties of \(K\) in the following lemma.

\begin{lemma}\label{lemma:K-resolvent}
  Let \(K\) denote the matrix-valued function
  \begin{equation*}
    K = \frac{1}{2\kappa}(-\delta I + J)((\hat{Q}\cdot\hat{Q})I + 2\hat{Q}\hat{Q}^{T}),
  \end{equation*}
  identified with the operator \(\phi \mapsto K\phi\) of
  multiplication by it, so that \(L_{\hat{Q}} = L_{0} + K\). Then
  \(K\) is bounded on \(H^{2}(\mathbb{R}^{3})\), and
  \(K(\lambda - L_{0})^{-1}\) is compact on \(H^{2}(\mathbb{R}^{3})\)
  for every \(\lambda\) in the resolvent set of \(L_{0}\). Moreover,
  let \(\rho > -\frac{1}{4}\). Then
  \begin{equation}\label{eq:K-resolvent-decay}
    \|K(\lambda - L_{0})^{-1}\|_{\mathcal{L}(H^{2}(\mathbb{R}^{3}))}
    \to 0
    \quad\text{as } |\imag \lambda| \to \infty,\ \real \lambda \geq \rho.
  \end{equation}
\end{lemma}

\begin{proof}
  Each entry of \(K\) is a quadratic polynomial in the two components
  of \(\hat{Q}\), so \eqref{eq:Q-hat-decay-assumption} gives
  \(K \in W^{2,\infty}(\mathbb{R}^{3})\). In particular,
  multiplication by \(K\) is bounded on \(H^{2}(\mathbb{R}^{3})\),
  which is the first claim.

  We next show that \(K(\lambda - L_{0})^{-1}\) is compact on
  \(H^{2}(\mathbb{R}^{3})\) for every \(\lambda\) in the resolvent set
  of \(L_{0}\). If \(\mu\) is another point in the resolvent set, the
  resolvent identity gives
  \begin{equation}\label{eq:resolvent-identity-L_0}
    K(\lambda - L_{0})^{-1} = K(\mu - L_{0})^{-1}(I + (\mu - \lambda)(\lambda - L_{0})^{-1}),
  \end{equation}
  in which the second factor is bounded on \(H^{2}(\mathbb{R}^{3})\).
  It therefore suffices to verify the compactness for a single point
  in the resolvent set. To that end we fix \(\mu\) with
  \(\real \mu > -\frac{1}{4}\) and let \((\phi_{n})_{n}\) be a bounded
  sequence in \(H^{2}(\mathbb{R}^{3})\). Put
  \(\psi_{n} = (\mu - L_{0})^{-1}\phi_{n}\). By
  Lemma~\ref{lemma:L_0-spectrum}, \((\psi_{n})_{n}\) is uniformly
  bounded in both \(H^{2}(\mathbb{R}^{3})\) and
  \(H^{3}(\mathbb{R}^{3})\). We must show that \((K\psi_{n})_{n}\) has
  a convergent subsequence in \(H^{2}(\mathbb{R}^{3})\).

  To make use of the decay of \(K\) we split the estimates into two
  parts, one on a ball \(B_{R}\) and one on the region outside. For
  all \(m\), \(n\), and all \(R > 0\) we have
  \begin{align*}
    \|K\psi_{m} - K\psi_{n}\|_{H^{2}(\mathbb{R}^{3})}
    &\leq \|K\psi_{m} - K\psi_{n}\|_{H^{2}(B_{R})}
    + \|K\psi_{m} - K\psi_{n}\|_{H^{2}(\mathbb{R}^{3} \setminus B_{R})}\\
    &\lesssim \|K\|_{W^{2,\infty}}\|\psi_{m} - \psi_{n}\|_{H^{2}(B_{R})}
    + \|K\psi_{m} - K\psi_{n}\|_{H^{2}(\mathbb{R}^{3} \setminus B_{R})}.
  \end{align*}

  To bound the first term we note that the sequence \((\psi_{n})_{n}\)
  is bounded in \(H^{3}(B_{R})\) for every \(R\) and since the
  embedding \(H^{3}(B_{R}) \hookrightarrow H^{2}(B_{R})\) is compact,
  by a diagonal argument we can extract a subsequence, still denoted
  \((\psi_{n})_{n}\), that converges in \(H^{2}(B_{R})\) for every
  \(R > 0\).

  To bound the second term we let
  \begin{equation}\label{eq:theta_R}
    \theta_{R} := \max_{|\alpha| \leq 2}\ \sup_{|\bm{\xi}| > R} |\partial^{\alpha}K(\bm{\xi})|.
  \end{equation}
  By~\eqref{eq:Q-hat-decay-assumption}, \(\theta_{R} \to 0\) as
  \(R \to \infty\). It follows that
  \begin{equation*}
    \|K\psi_{m} - K\psi_{n}\|_{H^{2}(\mathbb{R}^{3} \setminus B_{R})}
    \lesssim \theta_{R}\|\psi_{m} - \psi_{n}\|_{H^{2}(\mathbb{R}^{3} \setminus B_{R})}
    \leq 2\theta_{R}\sup_{k}\|\psi_{k}\|_{H^{2}(\mathbb{R}^{3})}.
  \end{equation*}

  From the above it follows that \((K\psi_{n})_{n}\) is a Cauchy
  sequence in \(H^{2}(\mathbb{R}^{3})\). Indeed, the second term can
  be made arbitrarily small by taking \(R\) sufficiently large, after
  which the first term can be made arbitrarily small by taking \(m\)
  and \(n\) sufficiently large. This proves that
  \(K(\lambda - L_{0})^{-1}\) is compact.

  It remains to prove~\eqref{eq:K-resolvent-decay}. Fix
  \(\rho > -\frac{1}{4}\) and suppose, for contradiction, that the
  convergence fails. Then there are \(\eta > 0\) and a sequence
  \((\lambda_{n})_{n}\) with \(\real \lambda_{n} \geq \rho\) and
  \(|\imag \lambda_{n}| \to \infty\) such that
  \(\|K(\lambda_{n} -
  L_{0})^{-1}\|_{\mathcal{L}(H^{2}(\mathbb{R}^{3}))} \geq 2\eta\) for
  every \(n\). We may therefore pick
  \(\phi_{n} \in H^{2}(\mathbb{R}^{3})\) of unit norm with
  \begin{equation}\label{eq:K-decay-contradiction}
    \|K\psi_{n}\|_{H^{2}(\mathbb{R}^{3})} \geq \eta,
    \qquad \psi_{n} = (\lambda_{n} - L_{0})^{-1}\phi_{n}.
  \end{equation}
  Since \(\real \lambda_{n} + \frac{1}{4} \geq \rho + \frac{1}{4} > 0\),
  Lemma~\ref{lemma:L_0-spectrum} bounds \((\psi_{n})_{n}\) in
  \(H^{2}(\mathbb{R}^{3})\) and in \(H^{3}(\mathbb{R}^{3})\),
  uniformly in \(n\).

  We first localise. By~\eqref{eq:theta_R},
  \begin{equation*}
    \|K\psi_{n}\|_{H^{2}(\mathbb{R}^{3} \setminus B_{R})}
    \lesssim \theta_{R}\sup_{k}\|\psi_{k}\|_{H^{2}(\mathbb{R}^{3})},
  \end{equation*}
  so we may fix \(R > 0\), once and for all, such that this is at most
  \(\frac{\eta}{2}\) for every \(n\). Combined
  with~\eqref{eq:K-decay-contradiction} and the boundedness of \(K\),
  this gives
  \begin{equation}\label{eq:K-decay-local-lower-bound}
    \|\psi_{n}\|_{H^{2}(B_{R})} \gtrsim \eta
  \end{equation}
  for every \(n\), with implicit constant depending on
  \(\|K\|_{W^{2,\infty}}\). Since \((\psi_{n})_{n}\) is bounded in
  \(H^{3}(\mathbb{R}^{3})\) and the embedding
  \(H^{3}(B_{R}) \hookrightarrow H^{2}(B_{R})\) is compact, we may
  pass to a subsequence, still denoted \((\psi_{n})_{n}\), converging
  in \(H^{2}(B_{R})\) to some \(\psi\), which
  by~\eqref{eq:K-decay-local-lower-bound} is non-zero.

  We now show that \(\psi_{n}\) goes to zero weakly in
  \(H^{2}(\mathbb{R}^{3})\), which contradicts the above. For
  \(g \in D(L_{0}^{*})\) we have
  \begin{equation*}
    |\langle \psi_{n}, g\rangle_{H^{2}(\mathbb{R}^{3})}|
    = \left|\left\langle \phi_{n}, (\conj{\lambda_{n}} - L_{0}^{*})^{-1}g\right\rangle_{H^{2}(\mathbb{R}^{3})}\right|
    \leq \|(\conj{\lambda_{n}} - L_{0}^{*})^{-1}g\|_{H^{2}(\mathbb{R}^{3})}.
  \end{equation*}
  It hence suffices to show that
  \(\|(\conj{\lambda_{n}} - L_{0}^{*})^{-1}g\|_{H^{2}(\mathbb{R}^{3})}
  \to 0\). Taking adjoints in~\eqref{eq:L_0-resolvent-bound} gives
  \begin{equation}\label{eq:L_0-adjoint-resolvent-bound}
    \|(\conj{\lambda} - L_{0}^{*})^{-1}\|_{\mathcal{L}(H^{2}(\mathbb{R}^{3}))}
    \lesssim (\real \lambda + 1/4)^{-1}.
  \end{equation}
  Applying \((\conj{\lambda_{n}} - L_{0}^{*})^{-1}\) to the
  identity
  \(\conj{\lambda_{n}} g = (\conj{\lambda_{n}} - L_{0}^{*})g +
  L_{0}^{*}g\) yields
  \begin{equation*}
    (\conj{\lambda_{n}} - L_{0}^{*})^{-1}g
    = \frac{1}{\conj{\lambda_{n}}}\left(
      g + (\conj{\lambda_{n}} - L_{0}^{*})^{-1}L_{0}^{*}g
    \right).
  \end{equation*}
  By~\eqref{eq:L_0-adjoint-resolvent-bound} this is uniformly bounded
  in \(H^{2}(\mathbb{R}^{3})\) by a constant times
  \(|\lambda_{n}|^{-1}\left(\|g\|_{H^{2}(\mathbb{R}^{3})} +
    \|L_{0}^{*}g\|_{H^{2}(\mathbb{R}^{3})}\right)\), and hence tends
  to zero as \(|\imag \lambda_{n}| \to \infty\). As \(D(L_{0}^{*})\)
  is dense in \(H^{2}(\mathbb{R}^{3})\) and \((\psi_{n})_{n}\) is
  bounded, it follows that \(\psi_{n}\) goes to zero weakly in
  \(H^{2}(\mathbb{R}^{3})\). Restriction to \(B_{R}\) is bounded and
  therefore weakly continuous, so \(\psi_{n}\) goes to zero weakly in
  \(H^{2}(B_{R})\) as well. Since \(\psi_{n} \to \psi\) strongly in
  \(H^{2}(B_{R})\), we conclude \(\psi = 0\), a contradiction. This
  proves~\eqref{eq:K-resolvent-decay}.
\end{proof}

Since \(K\) is bounded, the domain of \(L_{\hat{Q}}\) is that of
\(L_{0}\), and \(L_{\hat{Q}}\) generates a \(C_{0}\)-semigroup on
\(H^{2}(\mathbb{R}^{3})\) by the bounded perturbation
theorem~\cite[Chapter~III, Theorem~1.3]{Engel2000}. With this in hand,
we are ready to prove the two properties that will be used to
established hyperbolicity for the rescaled semigroup generated by
\(L_{\hat{Q}}\).

\begin{lemma}\label{lemma:unstable-spectrum}
  Let
  \(\Omega = \left\{\lambda \in \mathbb{C} \colon \real \lambda >
    -\frac{1}{4}\right\}\). Then every point of
  \(\sigma(L_{\hat{Q}}) \cap \Omega\) is an isolated eigenvalue of
  finite algebraic multiplicity. Moreover, if \(\rho > -\frac{1}{4}\)
  is such that no point of \(\sigma(L_{\hat{Q}})\) has real part equal
  to \(\rho\) then there exists a constant \(M\) such that for all
  \(\lambda\) with real part \(\rho\) we have
  \begin{equation*}
    \|(\lambda - L_{\hat{Q}})^{-1}\|_{\mathcal{L}(H^{2}(\mathbb{R}^{3}))} \leq M.
  \end{equation*}
\end{lemma}

\begin{proof}
  By Lemma~\ref{lemma:L_0-spectrum}, \(\Omega\) is contained in the
  resolvent set of \(L_{0}\), so for every \(\lambda \in \Omega\) we
  have the factorization
  \begin{equation}\label{eq:resolvent-factorization}
    \lambda - L_{\hat{Q}} = (I - K(\lambda - L_{0})^{-1})(\lambda - L_{0})
  \end{equation}
  on \(D(L_{0})\). Since \(\lambda - L_{0}\) is invertible, the
  spectrum of \(L_{\hat{Q}}\) in \(\Omega\) consists precisely of
  those \(\lambda\) for which \(I - K(\lambda - L_{0})^{-1}\) fails to
  be invertible. By Lemma~\ref{lemma:K-resolvent}, the operator
  \(K(\lambda - L_{0})^{-1}\) is compact, and it depends analytically
  on \(\lambda \in \Omega\) since the resolvent does.

  As \(\Omega\) is connected, the analytic Fredholm theorem applied to
  this family yields the following dichotomy: either
  \(\Omega \subseteq \sigma(L_{\hat{Q}})\), or
  \(\Omega \cap \sigma(L_{\hat{Q}})\) is discrete in \(\Omega\) and
  consists of eigenvalues of finite algebraic multiplicity. It
  therefore suffices to exhibit a single point of \(\Omega\) belonging
  to the resolvent set of \(L_{\hat{Q}}\). To find such a point we
  note that since \(K\) is bounded, it follows from the second claim
  in Lemma~\ref{lemma:L_0-spectrum} that
  \begin{equation}\label{eq:K-resolvent-bound}
    \|K(\lambda - L_{0})^{-1}\|_{\mathcal{L}(H^{2}(\mathbb{R}^{3}))}
    \lesssim \frac{1}{\real \lambda + \frac{1}{4}}.
  \end{equation}
  In particular,
  \(\|K(\lambda - L_{0})^{-1}\|_{\mathcal{L}(H^{2}(\mathbb{R}^{3}))}\)
  is smaller than \(1\) for \(\lambda\) with large enough real part,
  and a Neumann bound together with~\eqref{eq:resolvent-factorization}
  shows that \(\lambda - L_{\hat{Q}}\) is invertible in this case.

  To prove the second statement we note that since \(\lambda\) is in
  the resolvent set by assumption, the resolvent is bounded on any
  compact set of \(\lambda\) and it therefore suffices to verify that
  the resolvent is bounded as \(|\imag \lambda| \to \infty\). For this, we
  note that by Lemma~\ref{lemma:K-resolvent} there is \(N > 0\) such
  that \(\|K(\lambda - L_{0})^{-1}\| \leq \frac{1}{2}\) whenever
  \(\real \lambda = \rho\) and \(|\imag \lambda| \geq N\). For such
  \(N\), the factorization~\eqref{eq:resolvent-factorization} and the
  bound~\eqref{eq:L_0-resolvent-bound} give
  \begin{equation*}
    \|(\lambda - L_{\hat{Q}})^{-1}\| \leq 2\|(\lambda - L_{0})^{-1}\| \lesssim \frac{1}{\real \lambda + \frac{1}{4}},
  \end{equation*}
  proving that the resolvent is bounded whenever
  \(|\imag \lambda| \geq N\).
\end{proof}

Finally, we need the following lemma that gives us control over the
nonlinearity.

\begin{lemma}\label{lemma:nonlinearity-bound}
  For \(\phi, \varphi \in H^{2}(\mathbb{R}^{3})\) with
  \(\|\phi\|_{H^{2}(\mathbb{R}^{3})},
  \|\varphi\|_{H^{2}(\mathbb{R}^{3})} \leq 1\) we have
  \begin{equation*}
    \|\mathcal{N}(\phi)\|_{H^{2}(\mathbb{R}^{3})} \lesssim \|\phi\|_{H^{2}(\mathbb{R}^{3})}^{2}
  \end{equation*}
  and
  \begin{equation*}
    \|\mathcal{N}(\phi) - \mathcal{N}(\varphi)\|_{H^{2}(\mathbb{R}^{3})}
    \lesssim (\|\phi\|_{H^{2}(\mathbb{R}^{3})} + \|\varphi\|_{H^{2}(\mathbb{R}^{3})})
    \|\phi - \varphi\|_{H^{2}(\mathbb{R}^{3})}.
  \end{equation*}
\end{lemma}

\begin{proof}
  Recall from~\eqref{eq:N-R3} that
  \begin{equation*}
    \mathcal{N}(\phi) = \frac{1}{2\kappa}(-\delta I + J)
    \left(2(\hat{Q} \cdot \phi)\phi + (\phi\cdot\phi)(\hat{Q} + \phi)\right).
  \end{equation*}
  By~\eqref{eq:Q-hat-decay-assumption} we have
  \(\hat{Q} \in W^{2,\infty}(\mathbb{R}^{3})\), and hence \(\hat{Q}\)
  is a bounded multiplier on \(H^{2}(\mathbb{R}^{3})\). Combining this
  with the fact that \(H^{2}(\mathbb{R}^{3})\) forms a Banach algebra,
  we get
  \begin{equation*}
    \|\mathcal{N}(\phi)\|_{H^{2}(\mathbb{R}^{3})}
    \lesssim \|\hat{Q}\|_{W^{2,\infty}}\|\phi\|_{H^{2}(\mathbb{R}^{3})}^{2} + \|\phi\|_{H^{2}(\mathbb{R}^{3})}^{3}
    \lesssim (\|\hat{Q}\|_{W^{2,\infty}} + 1)\|\phi\|_{H^{2}(\mathbb{R}^{3})}^{2}
    \lesssim \|\phi\|_{H^{2}(\mathbb{R}^{3})}^{2}.
  \end{equation*}

  For \(\mathcal{N}(\phi) - \mathcal{N}(\varphi)\) we get that it
  splits into the three terms
  \begin{align*}
    2(\hat{Q} \cdot \phi)\phi - 2(\hat{Q} \cdot \varphi)\varphi
    &= 2(\hat{Q} \cdot (\phi - \varphi))\phi + 2(\hat{Q} \cdot \varphi)(\phi - \varphi),\\
    (\phi\cdot\phi - \varphi\cdot\varphi)\hat{Q}
    &= ((\phi - \varphi) \cdot (\phi + \varphi))\hat{Q},\\
    (\phi\cdot\phi)\phi - (\varphi\cdot\varphi)\varphi
    &= (\phi\cdot\phi)(\phi - \varphi) + ((\phi - \varphi) \cdot (\phi + \varphi))\varphi.
  \end{align*}
  The bound then follows in the same way as for \(\mathcal{N}(\phi)\),
  using \(\|\phi\|_{H^{2}(\mathbb{R}^{3})} \leq 1\) and
  \(\|\varphi\|_{H^{2}(\mathbb{R}^{3})} \leq 1\) to absorb the
  quadratic factors \(\phi\cdot\phi\) and \(\varphi\cdot\varphi\) into
  the constant.
\end{proof}

The above established bounds now allow us to prove the following
result regarding the existence of a perturbation solving
Equation~\eqref{eq:perturbation-equation}.

\begin{theorem}\label{thm:decaying-solution}
  Suppose that the operator \(L_{\hat{Q}}\) has at least one
  eigenvalue with positive real part. Then there exists \(\beta > 0\)
  such that for any \(r > 0\), we can find a mild solution
  \(\phi \in C((-\infty, 0], H^{2}(\mathbb{R}^{3}))\) to
  Equation~\eqref{eq:perturbation-equation} with
  \begin{equation*}
    0 < \sup_{\tau \in (-\infty, 0]} e^{-\beta\tau}\|\phi(\cdot, \tau)\|_{H^{2}(\mathbb{R}^{3})} < r.
  \end{equation*}
  Moreover, \(\phi\) can be taken to be real valued and such that
  \(\phi(\cdot, \tau) \neq 0\) for every \(\tau \leq 0\), and if the
  eigenfunction associated with the eigenvalue is radial, then
  \(\phi\) can be taken to be radial.
\end{theorem}

\begin{proof}
  To begin with, we assert that we can take \(\beta > 0\) such that
  \(\sigma(L_{\hat{Q}}) \cap \{\real \lambda > \beta\} \neq
  \emptyset\) and
  \(\sigma(L_{\hat{Q}}) \cap \{\real \lambda = \beta\} = \emptyset\).
  Since \(L_{\hat{Q}}\) has at least one eigenvalue with positive real
  part we get an open interval of values satisfying the first
  condition. To see that we can simultaneously satisfy the second
  condition it suffices to note that by
  Lemma~\ref{lemma:unstable-spectrum} the set
  \(\sigma(L_{\hat{Q}}) \cap \{\real \lambda > 0\}\) consists of
  isolated points and is hence at most countable, so the real part of
  the set cannot cover the full interval.

  Taking \(\beta\) such that the two above conditions are satisfied we
  can now, since \(H^{2}(\mathbb{R}^{3})\) is a Hilbert space,
  apply~\cite[Chapter~V, Section~1, Theorem~1.18]{Engel2000} to show
  that the rescaled semigroup \(e^{(L_{\hat{Q}} - \beta)\tau}\) is
  hyperbolic. The first condition of the theorem,
  \(\sigma(L_{\hat{Q}} - \beta) \cap \{\real \lambda = 0\} =
  \emptyset\), follows immediately from the choice of \(\beta\). For
  the second condition we need to verify that
  \(\|R(\lambda, L_{\hat{Q}} - \beta)\|_{\mathcal{L}(H^{2}(\mathbb{R}^{3}))} =
  \|R(\lambda + \beta, L_{\hat{Q}})\|_{\mathcal{L}(H^{2}(\mathbb{R}^{3}))}\) is
  uniformly bounded for \(\lambda \in \{\real \lambda = 0\}\), which
  is exactly the second statement of
  Lemma~\ref{lemma:unstable-spectrum}.

  The hyperbolicity of \(e^{(L_{\hat{Q}} - \beta)\tau}\) implies that
  we can decompose the space into stable and unstable subspaces as
  \(H^{2}(\mathbb{R}^{3}) = X_{s} \oplus X_{u}\) and that these
  subspaces are invariant under the semigroup. In particular, since
  \(L_{\hat{Q}}\) has at least one eigenvalue with real part greater
  than \(\beta\), \(X_{u}\) is non-trivial. We let \(P_{u}\) denote
  the projection onto \(X_{u}\), and set \(P_{s} = I - P_{u}\).
  Furthermore, we use the notation
  \(A_{s} = \left.L_{\hat{Q}}\right\vert_{X_{s}}\) and
  \(A_{u} = \left.L_{\hat{Q}}\right\vert_{X_{u}}\). From the
  hyperbolicity of \(e^{(L_{\hat{Q}} - \beta)\tau}\) we get that
  \(e^{(A_{u} - \beta)\tau}\) is invertible and that
  \(e^{(A_{s} - \beta)\tau}\) and \(e^{-(A_{u} - \beta)\tau}\) are
  both uniformly exponentially stable. In particular, this gives us
  the estimates
  \begin{align*}
    \|e^{A_{s}\tau}f\|_{H^{2}(\mathbb{R}^{3})} &\lesssim e^{\beta \tau}\|f\|_{H^{2}(\mathbb{R}^{3})} \quad\text{for }f \in X_{s},\\
    \|e^{-A_{u}\tau}f\|_{H^{2}(\mathbb{R}^{3})} &\lesssim e^{-\beta \tau}\|f\|_{H^{2}(\mathbb{R}^{3})} \quad\text{for }f \in X_{u}
  \end{align*}
  for \(\tau > 0\).

  To find a solution to Equation~\eqref{eq:perturbation-equation} we
  will use a Duhamel-type argument. With \(\beta\) as chosen above, we
  work in the Banach space
  \begin{equation*}
    W = \left\{\phi \in C((-\infty, 0], H^{2}(\mathbb{R}^{3})) \colon \|\phi\|_{W} < \infty\right\},
  \end{equation*}
  where \(\|\cdot\|_{W}\) is the norm
  \begin{equation*}
    \|\phi\|_{W} = \sup_{\tau \in (-\infty, 0]} e^{-\beta\tau} \|\phi(\cdot, \tau)\|_{H^{2}(\mathbb{R}^{3})}.
  \end{equation*}
  For \(\phi \in W\), \(a \in \mathbb{R}\) and a non-zero
  \(\phi_{u,0} \in X_{u}\), we let
  \begin{align*}
    T_{u}[\phi](\cdot, \tau) &= ae^{A_{u}\tau}\phi_{u,0} - \int_{\tau}^{0} e^{A_{u}(\tau - s)}P_{u}\mathcal{N}(\phi)(\cdot, s)\,ds,\\
    T_{s}[\phi](\cdot, \tau) &= \int_{-\infty}^{\tau} e^{A_{s}(\tau - s)}P_{s}\mathcal{N}(\phi)(\cdot, s)\,ds,
  \end{align*}
  and set \(T = T_{s} + T_{u}\). We will use a contraction mapping
  argument for \(T\) in \(W\) to construct a solution.

  In all the estimates below we restrict our attention to
  \(\phi, \varphi \in W\) with
  \(\|\phi\|_{W}, \|\varphi\|_{W} \leq 1\). Since
  \(\|\phi(\cdot, \tau)\|_{H^{2}(\mathbb{R}^{3})} \leq
  e^{\beta\tau}\|\phi\|_{W} \leq \|\phi\|_{W}\) for \(\tau \leq 0\),
  and similarly for \(\varphi\), Lemma~\ref{lemma:nonlinearity-bound}
  is then applicable.

  Applying the above bound for \(A_{s}\) together with the bound from
  Lemma~\ref{lemma:nonlinearity-bound}, we get
  \begin{align*}
    \|T_{s}[\phi](\cdot, \tau)\|_{H^{2}(\mathbb{R}^{3})}
    &\leq \int_{-\infty}^{\tau} \left\|e^{A_{s}(\tau - s)}P_{s}\mathcal{N}(\phi)(\cdot, s)\right\|_{H^{2}(\mathbb{R}^{3})}\,ds\\
    &\lesssim \int_{-\infty}^{\tau} e^{\beta(\tau - s)}\left\|P_{s}\mathcal{N}(\phi)(\cdot, s)\right\|_{H^{2}(\mathbb{R}^{3})}\,ds\\
    &\lesssim \int_{-\infty}^{\tau} e^{\beta(\tau - s)}\left\|\mathcal{N}(\phi)(\cdot, s)\right\|_{H^{2}(\mathbb{R}^{3})}\,ds\\
    &\lesssim \int_{-\infty}^{\tau} e^{\beta(\tau - s)}\left\|\phi(\cdot, s)\right\|_{H^{2}(\mathbb{R}^{3})}^{2}\,ds\\
    &\leq \int_{-\infty}^{\tau} e^{\beta(\tau - s)}e^{2\beta s}\,ds \left\|\phi\right\|_{W}^{2}\\
    &\lesssim e^{2\beta \tau}\left\|\phi\right\|_{W}^{2}.
  \end{align*}
  It follows that
  \begin{equation*}
    e^{-\beta\tau}\|T_{s}[\phi](\cdot, \tau)\|_{H^{2}(\mathbb{R}^{3})}
    \lesssim e^{\beta \tau}\left\|\phi\right\|_{W}^{2},
  \end{equation*}
  and taking the supremum for \(\tau \leq 0\) gives us
  \(\|T_{s}[\phi]\|_{W} \lesssim \left\|\phi\right\|_{W}^{2}\). For
  \(T_{u}\) we similarly get
  \begin{align*}
    \|T_{u}[\phi](\cdot, \tau)\|_{H^{2}(\mathbb{R}^{3})}
    &\leq \left\|ae^{A_{u}\tau}\phi_{u,0}\right\|_{H^{2}(\mathbb{R}^{3})}
      + \int_{\tau}^{0} \left\|e^{A_{u}(\tau - s)}P_{u}\mathcal{N}(\phi)(\cdot, s)\right\|_{H^{2}(\mathbb{R}^{3})}\,ds\\
    &\lesssim |a|e^{\beta\tau}\left\|\phi_{u,0}\right\|_{H^{2}(\mathbb{R}^{3})}
      + \int_{\tau}^{0} e^{\beta(\tau - s)}e^{2\beta s}\,ds\left\|\phi\right\|_{W}^{2}\\
    &= |a|e^{\beta\tau}\left\|\phi_{u,0}\right\|_{H^{2}(\mathbb{R}^{3})}
      + \beta^{-1}\left(
      e^{\beta\tau}
      - e^{2\beta\tau}
      \right)\left\|\phi\right\|_{W}^{2}.
  \end{align*}
  It follows that
  \begin{equation*}
    e^{-\beta\tau}\|T_{u}[\phi](\cdot, \tau)\|_{H^{2}(\mathbb{R}^{3})}
    \lesssim |a|\left\|\phi_{u,0}\right\|_{H^{2}(\mathbb{R}^{3})}
    + \left\|\phi\right\|_{W}^{2}.
  \end{equation*}
  Taking the supremum for \(\tau \leq 0\), we get
  \begin{equation*}
    \|T_{u}[\phi]\|_{W} \lesssim |a|\left\|\phi_{u,0}\right\|_{H^{2}(\mathbb{R}^{3})} + \left\|\phi\right\|_{W}^{2}.
  \end{equation*}

  Using
  \begin{equation*}
    \left\|\mathcal{N}(\phi)(\cdot, \tau) - \mathcal{N}(\varphi)(\cdot, \tau)\right\|_{H^{2}(\mathbb{R}^{3})}
    \lesssim (\|\phi(\cdot, \tau)\|_{H^{2}(\mathbb{R}^{3})} + \|\varphi(\cdot, \tau)\|_{H^{2}(\mathbb{R}^{3})})\|\phi(\cdot, \tau) - \varphi(\cdot, \tau)\|_{H^{2}(\mathbb{R}^{3})},
  \end{equation*}
  from Lemma~\ref{lemma:nonlinearity-bound}, similar calculations give
  us that for \(\phi, \varphi \in W\),
  \begin{equation*}
    \|T_{s}[\phi] - T_{s}[\varphi]\|_{W}
    \lesssim (\|\phi\|_{W} + \|\varphi\|_{W})\|\phi - \varphi\|_{W}
  \end{equation*}
  and
  \begin{equation*}
    \|T_{u}[\phi] - T_{u}[\varphi]\|_{W}
    \lesssim (\|\phi\|_{W} + \|\varphi\|_{W})\|\phi - \varphi\|_{W}.
  \end{equation*}

  Combining the bounds for \(T_{s}\) and \(T_{u}\), we get
  \begin{align*}
    \|T[\phi]\|_{W} &\leq C\left(|a|\left\|\phi_{u,0}\right\|_{H^{2}(\mathbb{R}^{3})} + \left\|\phi\right\|_{W}^{2}\right),\\
    \|T[\phi] - T[\varphi]\|_{W} &\leq C(\|\phi\|_{W} + \|\varphi\|_{W})\|\phi - \varphi\|_{W},
  \end{align*}
  for some constant \(C > 0\) depending on \(\beta\) but not on \(a\).

  Now let \(r > 0\) be given, put
  \(\tilde{r} = \min\left(\frac{r}{2}, 1, \frac{1}{4C}\right)\) and
  let
  \(B_{\tilde{r}} = \{\phi \in W \colon \|\phi\|_{W} \leq \tilde{r}\}\).
  For \(\phi, \varphi \in B_{\tilde{r}}\) the second estimate gives
  \begin{equation*}
    \|T[\phi] - T[\varphi]\|_{W} \leq 2C\tilde{r}\|\phi - \varphi\|_{W} \leq \frac{1}{2}\|\phi - \varphi\|_{W},
  \end{equation*}
  so \(T\) is a strict contraction. Choosing \(a \neq 0\) so small
  that
  \(C|a|\|\phi_{u,0}\|_{H^{2}(\mathbb{R}^{3})} \leq \frac{\tilde{r}}{2}\),
  the first estimate gives
  \(\|T[\phi]\|_{W} \leq \frac{\tilde{r}}{2} + C\tilde{r}^{2} \leq
  \tilde{r}\), so that \(T\) maps \(B_{\tilde{r}}\) into itself. By
  the Banach fixed-point theorem, \(T\) has a unique fixed point
  \(\phi\) in \(B_{\tilde{r}}\), and by construction of \(T\) this
  fixed point is a solution to
  Equation~\eqref{eq:perturbation-equation} in the mild, Duhamel
  sense.

  We note that \(\phi\) does not vanish identically: since
  \(\mathcal{N}(0) = 0\) we have
  \(T[0](\cdot, 0) = a\phi_{u,0} \neq 0\), so \(0\) is not a fixed
  point. Hence
  \begin{equation*}
    0 < \|\phi\|_{W} \leq \tilde{r} \leq \frac{r}{2} < r,
  \end{equation*}
  as required.

  To prove that \(\phi\) can be taken to be real valued we note that
  since \(L_{\hat{Q}}\) has real coefficients, \(X_{u}\) and \(X_{s}\)
  are invariant under conjugation and \(P_{u}\) and \(P_{s}\) are real
  valued. It follows that we can take \(\phi_{u,0}\) real valued, that
  \(T\) preserves this and as a result that the associated fixed point
  is real valued. The claim that \(\phi\) can be taken radial if the
  eigenfunction is radial follows from the same principles: we can
  then take \(\phi_{u,0}\) to be radial and \(T\) preserves this
  property.

  Finally, we prove that \(\phi\) can be taken such that
  \(\phi(\cdot, \tau) \neq 0\) for every \(\tau \leq 0\). Since
  \(\|\phi\|_{W} > 0\), there is a \(\tau_{2} \leq 0\) with
  \(\phi(\cdot, \tau_{2}) \neq 0\). As
  Equation~\eqref{eq:perturbation-equation} is autonomous, the
  translate \(\phi(\cdot, \cdot + \tau_{2})\), restricted to
  \((-\infty, 0]\), is again a mild solution, with \(W\)-norm
  \(e^{\beta\tau_{2}}\|\phi\|_{W} \in (0, r)\). Replacing \(\phi\) by
  this translate we may hence assume that
  \(\phi(\cdot, 0) \neq 0\). That \(\phi(\cdot, \tau) \neq 0\) also
  for \(\tau < 0\) then follows from forward uniqueness in
  \(H^{2}(\mathbb{R}^{3})\) for~\eqref{eq:perturbation-equation}.
\end{proof}

It remains to translate Theorem~\ref{thm:decaying-solution} back into
a statement about the original equation~\eqref{eq:CGL}. The solution
\(\phi\) decays as \(\tau \to -\infty\), which corresponds to
\(t \to 0^{+}\), and the perturbed solution therefore attains the same
initial data as the unperturbed one.

\begin{theorem}[Perturbed solution]\label{thm:perturbed-solution}
  Suppose that the assumptions of Theorem~\ref{thm:decaying-solution}
  hold and let \(\phi\) be a real-valued solution as in that theorem.
  With
  \(T = \frac{1}{2\kappa}\), \(\bm{\xi} = \frac{x}{\sqrt{2\kappa t}}\)
  and \(\tau = \log(2\kappa t)\), define
  \begin{align*}
    u_{\hat{Q}}(x, t) &= \frac{1}{(2\kappa t)^{\frac{1}{2}\left(1 + i\frac{\omega}{\kappa}\right)}}\hat{Q}(\bm{\xi}),\\
    u_{\phi}(x, t) &= \frac{1}{(2\kappa t)^{\frac{1}{2}\left(1 + i\frac{\omega}{\kappa}\right)}}\left(\hat{Q}(\bm{\xi}) + \phi(\bm{\xi}, \tau)\right),
  \end{align*}
  for \(t \in (0, T]\). Then, for every \(\alpha \in [0, 1 / 2)\), the
  difference \(u_{\phi} - u_{\hat{Q}}\) converges to zero in
  \(H^{\alpha}(\mathbb{R}^{3})\) as \(t \to 0^{+}\) and, extended by
  \(0\) at \(t = 0\), it satisfies
  \(u_{\phi} - u_{\hat{Q}} \in C([0, T], H^{\alpha}(\mathbb{R}^{3}))\).
\end{theorem}

\begin{proof}
  To begin with we note that, using the scaling of the
  \(L^{2}(\mathbb{R}^{3})\) and \(\dot{H}^{\alpha}(\mathbb{R}^{3})\)
  norms,
  \begin{align*}
    \|u_{\phi}(\cdot, t) - u_{\hat{Q}}(\cdot, t)\|_{L^{2}(\mathbb{R}^{3})}
    &= \frac{1}{(2\kappa t)^{\frac{1}{2}}}
    \left\|\phi\left(\frac{\cdot}{\sqrt{2\kappa t}}, \tau\right)\right\|_{L^{2}(\mathbb{R}^{3})}
    = (2\kappa t)^{\frac{1}{4}}\left\|\phi\left(\cdot, \tau\right)\right\|_{L^{2}(\mathbb{R}^{3})},\\
    \|u_{\phi}(\cdot, t) - u_{\hat{Q}}(\cdot, t)\|_{\dot{H}^{\alpha}(\mathbb{R}^{3})}
    &= \frac{1}{(2\kappa t)^{\frac{1}{2}}}
    \left\|\phi\left(\frac{\cdot}{\sqrt{2\kappa t}}, \tau\right)\right\|_{\dot{H}^{\alpha}(\mathbb{R}^{3})}
    = (2\kappa t)^{\frac{1 - 2\alpha}{4}}\left\|\phi\left(\cdot, \tau\right)\right\|_{\dot{H}^{\alpha}(\mathbb{R}^{3})}.
  \end{align*}
  Using that
  \(\|\cdot\|_{H^{\alpha}(\mathbb{R}^{3})} \simeq
  \|\cdot\|_{L^{2}(\mathbb{R}^{3})} +
  \|\cdot\|_{\dot{H}^{\alpha}(\mathbb{R}^{3})}\), that the
  \(H^{2}(\mathbb{R}^{3})\) norm controls both the
  \(L^{2}(\mathbb{R}^{3})\) and \(\dot{H}^{\alpha}(\mathbb{R}^{3})\)
  norms and the bound from Theorem~\ref{thm:decaying-solution} we get
  \begin{equation*}
    \|u_{\phi}(\cdot, t) - u_{\hat{Q}}(\cdot, t)\|_{H^{\alpha}(\mathbb{R}^{3})}
    \lesssim (2\kappa t)^{\frac{1 - 2\alpha}{4}}
    \left\|\phi\left(\cdot, \tau\right)\right\|_{H^{2}(\mathbb{R}^{3})}
    \lesssim (2\kappa t)^{\frac{1 - 2\alpha}{4}}e^{\beta\tau}
    = (2\kappa t)^{\beta + \frac{1 - 2\alpha}{4}}.
  \end{equation*}
  The exponent \(\beta + \frac{1 - 2\alpha}{4}\) is positive for every
  \(\alpha \in [0, 1/2)\) and it follows that the
  \(H^{\alpha}(\mathbb{R}^{3})\) norm is uniformly bounded for
  \(t \in (0, T]\) and goes to zero as \(t \to 0^{+}\). Since
  \(\phi \in C((-\infty, 0], H^{2}(\mathbb{R}^{3}))\), the map
  \(t \mapsto u_{\phi}(\cdot, t) - u_{\hat{Q}}(\cdot, t)\) is
  continuous from \((0, T]\) to \(H^{\alpha}(\mathbb{R}^{3})\).
  Extending it by \(0\) at \(t = 0\) we hence get
  \(u_{\phi} - u_{\hat{Q}} \in C([0, T],
  H^{\alpha}(\mathbb{R}^{3}))\).
\end{proof}

Finally, let us give the formal version of Theorem~\ref{thm-informal}
from the introduction. In this case, we specialize the result
specifically to the unstable eigenvalue produced in
Theorem~\ref{thm:unstable-eigenvalue}. We therefore let \(\delta = 0\)
and \(\epsilon \in \codenumber{[0.1681 \pm 10^{-15}]}\) and consider
the backward and forward self-similar solutions \(Q\) and \(\hat{Q}\)
from Theorems~\ref{thm:backward-existence}
and~\ref{thm:forward-existence}. By
Theorem~\ref{thm:unstable-eigenvalue}, \(L_{\hat{Q}}\) then has an
unstable eigenvalue with an associated eigenfunction that is radial.
We can hence take \(\phi\) to be a real valued, radial solution to
Equation~\eqref{eq:perturbation-equation} as given by
Theorem~\ref{thm:decaying-solution}. We let
\begin{align*}
  u_{Q}(x, t) &= \frac{1}{(-2\kappa t)^{\frac{1}{2}\left(1 + i\frac{\omega}{\kappa}\right)}}Q\left(\frac{|x|}{\sqrt{-2\kappa t}}\right),\\
  u_{\hat{Q}}(x, t) &= \frac{1}{(2\kappa t)^{\frac{1}{2}\left(1 + i\frac{\omega}{\kappa}\right)}}\hat{Q}\left(\frac{|x|}{\sqrt{2\kappa t}}\right),\\
  u_{\phi}(x, t) &= \frac{1}{(2\kappa t)^{\frac{1}{2}\left(1 + i\frac{\omega}{\kappa}\right)}}\left(\hat{Q}\left(\frac{|x|}{\sqrt{2\kappa t}}\right) + \phi\left(\frac{|x|}{\sqrt{2\kappa t}}, \log(2\kappa t)\right)\right).
\end{align*}
Furthermore, we let
\begin{equation*}
  u_{0}(x) = p_{Q}|x|^{-1 - i\frac{\omega}{\kappa}}
\end{equation*}
with \(p_{Q}\) as in Theorem~\ref{thm:backward-existence}. Let
\(T = \frac{1}{2\kappa}\) and define
\begin{equation}\label{eq:u_1-u_2}
  u_{1}(x, t) = \begin{cases}
    u_{Q}(x, t) & t \in [-T, 0)\\
    u_{0}(x) & t = 0\\
    u_{\hat{Q}}(x, t) & t \in (0, T]
  \end{cases},\qquad
  u_{2}(x, t) = \begin{cases}
    u_{Q}(x, t) & t \in [-T, 0)\\
    u_{0}(x) & t = 0\\
    u_{\phi}(x, t) & t \in (0, T]
  \end{cases}.
\end{equation}
Then \(u_{1}\) and \(u_{2}\) are two different radial solutions to the
CGL equation having the same initial data at \(t = -T\). We also
recall the notion of a suitable weak solution of~\eqref{eq:CGL}, see
e.g.~\cite{Yan1999}, which we in our case localize in space. We say
that \(u\) is a locally suitable weak solution of~\eqref{eq:CGL} if
\begin{enumerate}
\item
  \(u \in L^{\infty}([-T, T], L_{\text{loc}}^{2}(\mathbb{R}^{3})) \cap
  L^{2}([-T, T], W_{\text{loc}}^{1,2}(\mathbb{R}^{3}))\);
\item \(u\) satisfies~\eqref{eq:CGL} in the distributional sense;
\item For each real valued
  \(\psi \in C_{c}^{\infty}(\mathbb{R}^{3} \times (-T, T))\) with
  \(\psi \geq 0\), and with \(\delta = 0\) in~\eqref{eq:CGL}, the
  following inequality holds for \(t \in (-T, T)\)
  \begin{equation}\label{eq:energy-inequality}
    \int_{\mathbb{R}^{3}} |u(\cdot, t)|^{2}\psi(\cdot, t)
    + 2\epsilon \int_{-T}^{t}\int_{\mathbb{R}^{3}} |\nabla u|^{2}\psi
    \leq 2\real\left((\epsilon + i)\int_{-T}^{t}\int_{\mathbb{R}^{3}} u\nabla \conj{u}\nabla\psi\right)
    + \int_{-T}^{t}\int_{\mathbb{R}^{3}} |u|^{2}(\psi_{t} + 2\epsilon \Delta\psi).
  \end{equation}
\end{enumerate}
Then we have the following result.

\begin{theorem}\label{thm:main}
  The functions \(u_{1}\) and \(u_{2}\) both belong to
  \begin{equation*}
    L^{\infty}([-T, T], L^{3,\infty}(\mathbb{R}^{3}))
    \cap L_{\text{loc}}^{3}(\mathbb{R}^{3} \times (-T, T))
    \cap C^{\infty}((\mathbb{R}^{3} \times (-T, T)) \setminus \{(0, 0)\}),
  \end{equation*}
  are locally suitable weak solutions of~\eqref{eq:CGL} with the
  inequality in~\eqref{eq:energy-inequality} in fact being an
  equality, and share the same initial data
  \begin{equation*}
    u_{1}(\cdot, -T) = u_{2}(\cdot, -T) \in C^{\infty}(\mathbb{R}^{3}) \cap L^{3,\infty}(\mathbb{R}^{3}).
  \end{equation*}
  They satisfy
  \begin{equation*}
    u_{1} - u_{0}, u_{2} - u_{0} \in C([-T, T], H^{\alpha}(\mathbb{R}^{3}))\qquad \text{for any } \alpha \in [0, 1 / 2).
  \end{equation*}
  On \((0, T]\), only \(u_{1}\) is forward self-similar and they are
  distinct for all \(t \in (0, T]\).
\end{theorem}

\begin{proof}
  That \(u_{1}\) lies in
  \(L^{\infty}([-T, T], L^{3,\infty}(\mathbb{R}^{3}))\) follows from
  it having \(|x|^{-1}\) decay at infinity for all \(t\) and that the
  norm is invariant under the scaling~\eqref{scaling2}. For
  \(L_{\text{loc}}^{3}\) we use the substitution
  \(x = \sqrt{-2\kappa t}\,\bm{\xi}\), for which the \(L^{3}\) norm is
  invariant, together with \(|Q(\xi)| \lesssim \min(1, \xi^{-1})\), to
  get, for \(R \geq \sqrt{2\kappa T}\),
  \begin{equation*}
    \int_{|x| \leq R} |u_{Q}(x, t)|^{3}\,dx
    = \int_{|\bm{\xi}| \leq R / \sqrt{-2\kappa t}} |Q(|\bm{\xi}|)|^{3}\,d\bm{\xi}
    \lesssim 1 + \log\left(\frac{R}{\sqrt{-2\kappa t}}\right),
  \end{equation*}
  which is integrable in \(t\) on \([-T, 0)\). A similar computation
  for \(u_{\hat{Q}}\) gives us that
  \(u_{1} \in L_{\text{loc}}^{3}(\mathbb{R}^{3} \times (-T, T))\). For
  \(u_{2}\) the same arguments apply, with the difference
  \(w = u_{\phi} - u_{\hat{Q}}\) handled separately. Since the
  \(L^{3}(\mathbb{R}^{3})\) norm is invariant under the
  scaling~\eqref{scaling2}, we have
  \begin{equation*}
    \|w(\cdot, t)\|_{L^{3}(\mathbb{R}^{3})}
    = \|\phi(\cdot, \tau)\|_{L^{3}(\mathbb{R}^{3})}
    \lesssim \|\phi(\cdot, \tau)\|_{H^{2}(\mathbb{R}^{3})},
  \end{equation*}
  which by Theorem~\ref{thm:decaying-solution} is uniformly bounded
  for \(\tau \leq 0\). Together with
  \(L^{3}(\mathbb{R}^{3}) \subset L^{3,\infty}(\mathbb{R}^{3})\) this
  gives both the \(L^{3,\infty}\) and the \(L_{\text{loc}}^{3}\)
  bounds for \(u_{2}\).

  The smoothness for \(t \neq 0\) is immediate for \(u_{1}\) from the
  smoothness of \(Q\) and \(\hat{Q}\). The same applies to \(u_{2}\)
  for \(t < 0\), whereas for \(t > 0\) it is a consequence of a
  standard parabolic bootstrap. In particular, they both
  solve~\eqref{eq:CGL} pointwise for \(t \neq 0\). To prove smoothness
  across \(t = 0\) we first establish that both \(u_{1}\) and
  \(u_{2}\) are distributional solutions also across \(t = 0\).

  Note first that \(u_{1}, u_{2} \in L_{\text{loc}}^{3}\) gives
  \(|u_{1}|^{2}u_{1}, |u_{2}|^{2}u_{2} \in L_{\text{loc}}^{1}\), so
  that the distributional formulation makes sense. We next note that
  they are continuous across \(t = 0\) for \(x \neq 0\). For \(u_{1}\)
  this follows from \(Q\) and \(\hat{Q}\) having the same leading
  asymptotic behavior at infinity, so that both \(u_{Q}\) and
  \(u_{\hat{Q}}\) converge to \(u_{0}\) as \(t \to 0^{\mp}\), locally
  uniformly in \(x \neq 0\). For \(u_{2}\) we use that \(\phi\) is
  radial, so that the radial Sobolev inequality gives
  \(|\phi(\bm{\xi}, \tau)| \lesssim |\bm{\xi}|^{-1}\|\phi(\cdot,
  \tau)\|_{H^{1}(\mathbb{R}^{3})}\) and hence, the prefactor in \(w\)
  cancelling against \(|\bm{\xi}|^{-1}\),
  \begin{equation}\label{eq:w-pointwise-bound}
    |w(x, t)| \lesssim \frac{\|\phi(\cdot, \tau)\|_{H^{1}(\mathbb{R}^{3})}}{|x|}
    \lesssim \frac{(2\kappa t)^{\beta}}{|x|}
  \end{equation}
  by Theorem~\ref{thm:decaying-solution}. In particular, \(w\) is
  bounded near any point \((x_{0}, 0)\) with \(x_{0} \neq 0\) and
  tends to zero as \(t \to 0^{+}\), locally uniformly in
  \(x \neq 0\). Testing against a function supported away from the
  origin and splitting the integral into \(t < 0\) and \(t > 0\), the
  boundary terms at \(t = 0\) coming from the integration by parts in
  \(t\) cancel by this continuity, so that both solve~\eqref{eq:CGL}
  in \(\mathcal{D}'\) away from \((0, 0)\). It hence only remains to
  verify that the singularity at \((0, 0)\) is removable, which can be
  done directly, by multiplying the test function by a cutoff which
  cuts out a parabolic ball of radius \(\rho\) around the origin, and
  letting \(\rho \to 0\). They are hence solutions of~\eqref{eq:CGL}
  in \(\mathcal{D}'(\mathbb{R}^{3} \times (-T, T))\).

  This also gives us the smoothness across \(t = 0\). Indeed, both
  \(u_{1}\) and \(u_{2}\) are bounded near any point \((x_{0}, 0)\)
  with \(x_{0} \neq 0\), by~\eqref{eq:w-pointwise-bound} for the
  latter, and hence bounded solutions of~\eqref{eq:CGL} in
  \(\mathcal{D}'\) in a neighborhood of such a point. The same
  parabolic bootstrap as above therefore gives that they are smooth
  there, so that
  \(u_{1}, u_{2} \in C^{\infty}((\mathbb{R}^{3} \times (-T, T))
  \setminus \{(0, 0)\})\).

  It remains to verify that \(u_{1}\) and \(u_{2}\) are locally
  suitable weak solutions. The second condition is proved above, so we
  only have to check the first and the third one. For the first one,
  the bound \(|u_{1}| \lesssim \min(|t|^{-\frac{1}{2}}, |x|^{-1})\)
  gives
  \(u_{1} \in L^{\infty}([-T, T],
  L_{\text{loc}}^{2}(\mathbb{R}^{3}))\), and the substitution
  \(x = \sqrt{-2\kappa t}\,\bm{\xi}\) together with
  \(|Q'(\xi)| \lesssim \min(1, \xi^{-2})\) gives
  \begin{equation*}
    \int_{|x| \leq R} |\nabla u_{Q}(x, t)|^{2}\,dx
    = \frac{1}{\sqrt{-2\kappa t}}
    \int_{|\bm{\xi}| \leq R / \sqrt{-2\kappa t}} |Q'(|\bm{\xi}|)|^{2}\,d\bm{\xi}
    \lesssim \frac{1}{\sqrt{-2\kappa t}},
  \end{equation*}
  which is integrable in \(t\) on \([-T, 0)\). A similar computation
  for \(u_{\hat{Q}}\) gives us that
  \(u_{1} \in L^{2}([-T, T], W_{\text{loc}}^{1,2}(\mathbb{R}^{3}))\).
  For \(u_{2}\) the difference \(w\) is again handled separately: it
  is bounded in \(L^{3}(\mathbb{R}^{3})\), and hence in
  \(L_{\text{loc}}^{2}(\mathbb{R}^{3})\), uniformly in \(t\), and
  \begin{equation*}
    \|\nabla w(\cdot, t)\|_{L^{2}(\mathbb{R}^{3})}
    = (2\kappa t)^{-\frac{1}{4}}\|\nabla \phi(\cdot, \tau)\|_{L^{2}(\mathbb{R}^{3})}
    \lesssim (2\kappa t)^{\beta - \frac{1}{4}},
  \end{equation*}
  which is square integrable in \(t\) on \((0, T]\).

  For the third condition, multiplying~\eqref{eq:CGL} by
  \(\conj{u}\psi\), taking real parts and integrating by parts gives
  the identity corresponding to~\eqref{eq:energy-inequality}, with
  equality since \(\delta = 0\), for every \(\psi\) supported away
  from the origin. As for the distributional formulation, the general
  case follows by multiplying \(\psi\) with a cutoff which cuts out a
  parabolic ball of radius \(\rho\) around the origin, and letting
  \(\rho \to 0\).

  That \(u_{1}\) and \(u_{2}\) have the same initial data at
  \(t = -T\) is immediate from the definitions, both being given by
  \(u_{Q}\) for \(t < 0\). Since \(T = \frac{1}{2\kappa}\), it is
  given by \(Q(|\cdot|)\), which lies in
  \(C^{\infty}(\mathbb{R}^{3}) \cap L^{3,\infty}(\mathbb{R}^{3})\) by
  the above.

  For \(u_{1} - u_{0}\), let
  \(g(\bm{\xi}) = Q(|\bm{\xi}|) - p_{Q}|\bm{\xi}|^{-1 -
    i\frac{\omega}{\kappa}}\). Since \(Q\) is bounded, \(g\) behaves
  like \(|\bm{\xi}|^{-1}\) near the origin, and it is
  \(\mathcal{O}(|\bm{\xi}|^{-3})\) at infinity by
  Theorem~\ref{thm:backward-existence}, so that
  \(g \in H^{\alpha}(\mathbb{R}^{3})\) for every
  \(\alpha \in [0, 1/2)\). For \(t < 0\) we have
  \begin{equation*}
    u_{1}(x, t) - u_{0}(x)
    = \frac{1}{(-2\kappa t)^{\frac{1}{2}\left(1 + i\frac{\omega}{\kappa}\right)}}
    g\left(\frac{x}{\sqrt{-2\kappa t}}\right),
  \end{equation*}
  and hence, as in the proof of Theorem~\ref{thm:perturbed-solution},
  \(\|u_{1}(\cdot, t) - u_{0}\|_{H^{\alpha}(\mathbb{R}^{3})} \lesssim
  (-2\kappa t)^{\frac{1 -
      2\alpha}{4}}\|g\|_{H^{\alpha}(\mathbb{R}^{3})}\). The same holds
  for \(t > 0\) with \(\hat{Q}\) in place of \(Q\), and since
  \(u_{1}(\cdot, 0) = u_{0}\) this gives
  \(u_{1} - u_{0} \in C([-T, T], H^{\alpha}(\mathbb{R}^{3}))\). For
  \(u_{2}\) the same conclusion follows by combining this with
  Theorem~\ref{thm:perturbed-solution}, which gives
  \(w \in C([0, T], H^{\alpha}(\mathbb{R}^{3}))\) with
  \(w(\cdot, 0) = 0\).

  That \(u_{1}\) is forward self-similar on \((0, T]\) holds by
  construction. Self-similarity for \(u_{2}\) would require \(\phi\)
  to be independent of \(\tau\), which is impossible since \(\phi\) is
  non-zero but decays as \(\tau \to -\infty\) by
  Theorem~\ref{thm:decaying-solution}. Moreover, \(u_{1}\) and
  \(u_{2}\) are distinct for \(t \in (0, T]\) since \(\phi\) is
  non-zero for all \(\tau \leq 0\).
\end{proof}

\section{Backward, forward and linearized solutions}
\label{sec:backward-forward-linearized-solutions}

To prove the existence of connecting backward and forward self-similar
solutions, as well as an unstable eigenvalue for the forward
self-similar solution, we make use of the same type of shooting method
as in~\cite{Dahne2024}. Let us briefly describe the method when
applied to the backward self-similar solution and then describe the
adjustments required for the forward self-similar solution as well as
the unstable eigenvalue.

The backward self-similar ansatz~\eqref{eq:CGL-backward} gives us the
equation
\begin{equation}\label{eq:Q}
  (1 - i\epsilon)\left(Q'' + \frac{2}{\xi}Q'\right) + i\kappa\xi Q'
  + i \kappa Q - \omega Q + (1 + i\delta)|Q|^{2}Q = 0.
\end{equation}
The idea of the shooting method is to construct two solutions,
\(Q_{0}\) and \(Q_{\infty}\), to this equation, with \(Q_{0}\)
satisfying the boundary condition at zero and \(Q_{\infty}\) the
condition at infinity. We then show that with the right choice of
parameters, these two solutions match at an intermediate point,
forming a global solution satisfying both boundary conditions. One of
these parameters is the self-similar parameter \(\kappa\); the others
are related to the parametrization of the respective boundary
conditions.

The function \(Q_{\infty}\) should be a solution to
Equation~\eqref{eq:Q} satisfying the boundary condition
\(Q_{\infty}(\xi) \sim \xi^{-1 - i\frac{\omega}{\kappa}}\) as
\(\xi \to \infty\). There exists a one-dimensional complex manifold of
such solutions, associated with the constant in front of the leading
term in the asymptotic expansion.
We parametrize this manifold by \(\gamma \in \mathbb{C}\). To
emphasize the solution's dependence on \(\gamma\) as well as
\(\kappa\), we use the notation
\(Q_{\infty}(\xi) = Q_{\infty}(\gamma, \kappa; \xi)\).

The function \(Q_{0}\) should be a solution to Equation~\eqref{eq:Q}
satisfying the boundary condition \(Q_{0}'(0) = 0\). In this case we
parametrize such solutions by the value of \(Q_{0}(0)\). For scaling
reasons we can take \(Q_{0}(0) = \mu\) with \(\mu > 0\). We use the
notation \(Q_{0}(\xi) = Q_{0}(\mu, \kappa; \xi)\) to indicate the
solution's dependence on \(\mu\) and \(\kappa\).

Since we are dealing with a second-order ODE, the required matching
condition is that the values and derivatives of \(Q_{0}\) and
\(Q_{\infty}\) should match at some intermediate point
\(0 < \xi_{1} < \infty\). If indeed we have parameters \(\mu\),
\(\gamma\) and \(\kappa\) such that
\begin{equation}\label{eq:Q-match}
  Q_{0}(\mu, \kappa; \xi_{1}) = Q_{\infty}(\gamma, \kappa; \xi_{1})
  \quad\text{and}\quad
  Q_{0}'(\mu, \kappa; \xi_{1}) = Q_{\infty}'(\gamma, \kappa; \xi_{1}),
\end{equation}
then the two functions can be glued together to form a global solution
satisfying both boundary conditions. With this condition, it is
natural to define the map
\begin{equation}\label{eq:G}
  G(\mu, \gamma, \kappa) = \left(
    Q_{0}(\mu, \kappa; \xi_{1}) - Q_{\infty}(\gamma, \kappa; \xi_{1}),
    Q_{0}'(\mu, \kappa; \xi_{1}) - Q_{\infty}'(\gamma, \kappa; \xi_{1})
  \right) \colon \mathbb{R} \times \mathbb{C} \times \mathbb{R} \to \mathbb{C}^{2}.
\end{equation}
By identifying \(\mathbb{C}\) with \(\mathbb{R}^{2}\), we can
interpret \(G\) as a map from \(\mathbb{R}^{4}\) to
\(\mathbb{R}^{4}\). A zero of the function \(G\) then corresponds to a
matching as in~\eqref{eq:Q-match} and ultimately implies the
existence of a backward self-similar solution.

The existence of a zero of \(G\) is proved using an interval Newton
method. This requires computing interval enclosures of
\(Q_{0}(\mu, \kappa; \xi_{1})\) and
\(Q_{\infty}(\gamma, \kappa; \xi_{1})\) as well as their derivatives
with respect to both \(\xi\) and the parameters \(\mu\), \(\gamma\)
and \(\kappa\). We refer to~\cite{Dahne2024} for details on how these
enclosures are computed.

The approach for proving the existence of a forward self-similar
solution is largely the same. What changes is the equation, which
from the ansatz~\eqref{eq:CGL-forward} becomes
\begin{equation}\label{eq:Q-hat}
  (1 - i\epsilon)\left(\hat{Q}'' + \frac{2}{\xi}\hat{Q}'\right) - i\kappa\xi \hat{Q}'
  - i \kappa \hat{Q} + \omega \hat{Q} + (1 + i\delta)|\hat{Q}|^{2}\hat{Q} = 0,
\end{equation}
as well as the parametrization of the boundary conditions. In this
case, we denote the solution satisfying the appropriate boundary
condition at zero by \(\hat{Q}_{0}\) and the one satisfying the
appropriate boundary condition at infinity by \(\hat{Q}_{\infty}\).

The function \(\hat{Q}_{\infty}\) should be a solution to
Equation~\eqref{eq:Q-hat} with the same leading asymptotic behavior as
\(Q_{\infty}\) as \(\xi \to \infty\). As we will see in
Section~\ref{sec:forward-solution-infinity}, the forward self-similar
equation has a two-dimensional manifold of decaying solutions at
infinity. One of the dimensions is associated with solutions with
asymptotic behavior \(\xi^{-1 - i \frac{\omega}{\kappa}}\) and the
other dimension with solutions with asymptotic behavior
\(e^{-c\xi^{2}}\xi^{-2 + i \frac{\omega}{\kappa}}\), where
\(c = \frac{\kappa}{2}\frac{\epsilon - i}{1 + \epsilon^{2}}\). We can
parametrize this manifold by \(\hat{\gamma}_{1}\) and
\(\hat{\gamma}_{2}\), associated with these two asymptotic behaviors.
The value of \(\hat{\gamma}_{1}\) is directly determined by the
condition that \(\hat{Q}_{\infty}\) and \(Q_{\infty}\) should have the
same asymptotic behavior; this leaves us with only
\(\hat{\gamma}_{2}\) to vary. We emphasize the dependence on
\(\hat{\gamma}_{2}\) using the notation
\(\hat{Q}_{\infty}(\xi) = \hat{Q}_{\infty}(\hat{\gamma}_{2}; \xi)\).

The function \(\hat{Q}_{0}\) should be a solution to
Equation~\eqref{eq:Q-hat} satisfying the boundary condition
\(\hat{Q}_{0}'(0) = 0\). As for \(Q_{0}\), we parametrize the solution
by the value of \(\hat{Q}_{0}(0)\). In this case we are, however, not
free to choose the scaling, and we therefore take
\(\hat{Q}_{0}(0) = \nu\) with \(\nu \in \mathbb{C}\). We use the
notation \(\hat{Q}_{0}(\xi) = \hat{Q}_{0}(\nu; \xi)\).

As for the backward solution, we have the matching condition
\begin{equation}\label{eq:Q-hat-match}
  \hat{Q}_{0}(\nu; \xi_{1}) = \hat{Q}_{\infty}(\hat{\gamma}_{2}; \xi_{1})
  \quad\text{and}\quad
  \hat{Q}_{0}'(\nu; \xi_{1}) = \hat{Q}_{\infty}'(\hat{\gamma}_{2}; \xi_{1}).
\end{equation}
For this reason we define the map
\begin{equation}\label{eq:G-hat}
  \hat{G}(\nu, \hat{\gamma}_{2}) = \left(
  \hat{Q}_{0}(\nu; \xi_{1}) - \hat{Q}_{\infty}(\hat{\gamma}_{2}; \xi_{1}),
  \hat{Q}_{0}'(\nu; \xi_{1}) - \hat{Q}_{\infty}'(\hat{\gamma}_{2}; \xi_{1})
  \right) \colon \mathbb{C}^{2} \to \mathbb{C}^{2}.
\end{equation}
A zero of the function \(\hat{G}\) then corresponds to the existence
of a forward self-similar solution, with asymptotic behavior
determined by \(\hat{\gamma}_{1}\).

To prove the existence of a zero of \(\hat{G}\), we will, as for
\(G\), apply an interval Newton method. Since the nonlinearity
\(|Q|^{2}Q\) is not holomorphic, \(\hat{G}\) does not depend
holomorphically on \(\nu\) and \(\hat{\gamma}_{2}\). We therefore
split into real and imaginary parts and consider it as a function from
\(\mathbb{R}^{4}\) to \(\mathbb{R}^{4}\). In this setting, the map is
differentiable and the interval Newton method can be applied. This
requires computing interval enclosures of
\(\hat{Q}_{0}(\nu; \xi_{1})\) and
\(\hat{Q}_{\infty}(\hat{\gamma}_{2}; \xi_{1})\) as well as their
derivatives with respect to \(\xi\) and the real and imaginary parts
of \(\nu\) and \(\hat{\gamma}_{2}\), respectively. How to compute
enclosures related to \(\hat{Q}_{\infty}\) is discussed in
Section~\ref{sec:forward-solution-infinity}, and enclosures related to
\(\hat{Q}_{0}\) are covered in
Section~\ref{sec:forward-solution-zero}.

What remains is handling the unstable eigenvalue and associated
eigenfunction of the linear operator \(L_{\hat{Q}}\). Due to the
non-analyticity of the nonlinearity, it is in this case beneficial to
work in the complexification of the real space. The operator is then
given by
\begin{equation}\label{eq:L_Q-hat}
  L_{\hat{Q}}\phi = \frac{1}{2\kappa}\left[
  (\epsilon I + J)\left(\phi'' + \frac{2}{\xi}\phi'\right)
  + \kappa\xi\phi'
  + \left(\kappa I + \omega J\right)\phi
  + (-\delta I + J)((\hat{Q} \cdot \hat{Q})I + 2\hat{Q}\hat{Q}^{T})\phi
  \right],
\end{equation}
with \(\phi \colon [0, \infty) \to \mathbb{C}^{2}\) and where \(I\) is
the identity matrix and
\begin{equation*}
  J = \begin{pmatrix} 0 & -1 \\ 1 & 0 \end{pmatrix}.
\end{equation*}
We are looking for an eigenfunction \(Y\) and an eigenvalue
\(\lambda\) satisfying \(L_{\hat{Q}}Y = \lambda Y\). As with the
handling of \(\kappa\) for the backward solution, we treat
\(\lambda\) as a parameter. To simplify the notation we let
\begin{equation*}
  A = \epsilon I + J,\
  B_{1} = \kappa I,\
  B_{2} = 2A,\
  C = \kappa I + \omega J,\
  J_{N}(\xi) = (-\delta I + J)((\hat{Q}(\xi) \cdot \hat{Q}(\xi))I + 2\hat{Q}(\xi)\hat{Q}(\xi)^{T}),
\end{equation*}
and write the equation as
\begin{equation}\label{eq:Y}
  AY'' + (B_{1}\xi + B_{2}\xi^{-1})Y' + (C + J_{N}(\xi) - 2\kappa\lambda I)Y = 0.
\end{equation}
As in the earlier cases, we let \(Y_{\infty}\) and \(Y_{0}\) denote
solutions to this equation satisfying appropriate boundary conditions
at infinity and zero, respectively.

As we will see in Section~\ref{sec:linearized-solution-infinity}, the
equation has two linearly independent solutions with polynomial
behavior at infinity, and two linearly independent solutions with
exponential decay. The solutions with polynomial behavior do not give
solutions in \(H^{2}(\mathbb{R}^{3})\) for \(\lambda\) with positive
real part; we are hence only interested in the solutions with
exponential decay. This leaves us with a two-parameter family of
solutions, which we parametrize by \(c_{0} \in \mathbb{C}^{2}\). As we
are dealing with a linear equation, we take this parametrization to be
linear. Since the solution also depends on \(\lambda\), we use the
notation \(Y_{\infty}(\xi) = Y_{\infty}(c_{0}, \lambda; \xi)\).

As before, the function \(Y_{0}\) should be a solution to
Equation~\eqref{eq:Y} satisfying the boundary condition
\(Y_{0}'(0) = 0\). In this case we parametrize the solution by
\(y_{0} \in \mathbb{C}^{2}\) such that \(Y_{0}(0) = y_{0}\), and use
the notation \(Y_{0}(\xi) = Y_{0}(y_{0}, \lambda; \xi)\).

We have a matching condition similar to those in the earlier cases,
\begin{equation}\label{eq:Y-match}
  Y_{0}(y_{0}, \lambda; \xi_{1}) = Y_{\infty}(c_{0}, \lambda; \xi_{1})
  \quad\text{and}\quad
  Y_{0}'(y_{0}, \lambda; \xi_{1}) = Y_{\infty}'(c_{0}, \lambda; \xi_{1}).
\end{equation}
However, since in this case we are dealing with a linear ODE, we can
significantly simplify the matching procedure. If we let
\begin{align*}
  Y_{0,1}(\lambda; \xi_{1}) &= Y_{0}\left(\begin{pmatrix} 1 \\ 0 \end{pmatrix}, \lambda; \xi_{1}\right),\quad
  Y_{0,2}(\lambda; \xi_{1}) = Y_{0}\left(\begin{pmatrix} 0 \\ 1 \end{pmatrix}, \lambda; \xi_{1}\right),\\
  Y_{\infty,1}(\lambda; \xi_{1}) &= Y_{\infty}\left(\begin{pmatrix} 1 \\ 0 \end{pmatrix}, \lambda; \xi_{1}\right),\quad
  Y_{\infty,2}(\lambda; \xi_{1}) = Y_{\infty}\left(\begin{pmatrix} 0 \\ 1 \end{pmatrix}, \lambda; \xi_{1}\right),
\end{align*}
then the matching condition reduces to verifying that these solutions
are linearly dependent. Letting \(H(\lambda)\) denote the
determinant of the associated \(4 \times 4\) fundamental matrix,
\begin{equation}\label{eq:H}
  H(\lambda) =
  \begin{vmatrix}
    Y_{0,1}(\lambda; \xi_{1}) & Y_{0,2}(\lambda; \xi_{1}) & Y_{\infty,1}(\lambda; \xi_{1}) & Y_{\infty,2}(\lambda; \xi_{1}) \\
    Y_{0,1}'(\lambda; \xi_{1}) & Y_{0,2}'(\lambda; \xi_{1}) & Y_{\infty,1}'(\lambda; \xi_{1}) & Y_{\infty,2}'(\lambda; \xi_{1})
  \end{vmatrix} \colon \mathbb{C} \to \mathbb{C},
\end{equation}
we see that verifying linear dependence corresponds to proving the
existence of a zero of \(H\).

In this case, \(H\) is holomorphic in \(\lambda\), and we will
apply a winding number argument to prove the existence of a zero; see
Theorem~\ref{thm:unstable-eigenvalue} for more details. In particular,
this means that we do not need any derivatives with respect to
\(\lambda\). Details on how to compute the required interval
enclosures of \(Y_{\infty}\) and \(Y_{0}\) are given in
Sections~\ref{sec:linearized-solution-infinity}
and~\ref{sec:linearized-solution-zero}, respectively.

\section{Existence of an unstable eigenvalue}
\label{sec:existence-unstable-eigenvalue}

Proving the existence of an unstable eigenvalue requires first proving
the existence of a backward self-similar solution, then of an
associated forward self-similar solution and finally of an unstable
eigenvalue for the corresponding linear operator. It is therefore
natural to split the proof into three parts: one for the backward
solution, one for the forward solution and one for the eigenvalue.

In all three cases, the proof of existence is based on proving that
one of the functions \(G\)~\eqref{eq:G}, \(\hat{G}\)~\eqref{eq:G-hat}
or \(H\)~\eqref{eq:H} has a (locally unique) zero. For \(G\) and
\(\hat{G}\), this mostly follows the same strategy as
in~\cite[Section~5.1]{Dahne2024}, whereas for \(H\), a slightly
different approach based on a winding argument is used.

The approach in~\cite{Dahne2024} is based on the \emph{Krawczyk
  interval Newton method}, see e.g.~\cite{Moore1977}. Given a
continuously differentiable function
\(f\colon\mathbb{R}^{n} \to \mathbb{R}^{n}\) and a set \(X\) given by
a box in \(\mathbb{R}^{n}\) (a Cartesian product of intervals), we
define the interval Newton operator by
\begin{equation*}
  N(X) = \operatorname{mid}(X) - Yf(\operatorname{mid}(X)) + (I - YJ_{f}(X))(X - \operatorname{mid}(X)),
\end{equation*}
where \(\operatorname{mid}(X)\) denotes the midpoint of \(X\),
\(J_{f}(X)\) is (an enclosure of) the convex hull of the image of the
Jacobian of \(f\) on \(X\) and \(Y\) is a preconditioner that in
practice is taken to be an approximate inverse of the midpoint of
\(J_{f}(X)\). If we have \(N(X) \subseteq \operatorname{int}(X)\),
then the function \(f\) has a unique zero in \(X\) and this zero is
contained in \(N(X)\).

To apply the interval Newton method, one usually first uses
non-rigorous numerical methods to find an approximate zero
\(x \in \mathbb{R}^{n}\). One then takes a (small) box \(X\) around
this approximation on which the interval Newton method is applied. If
the box is too small or too large, then
\(N(X) \subseteq \operatorname{int}(X)\) will not be satisfied. For
choosing the box \(X\), we follow the heuristic approach discussed
in~\cite[Section~10]{Dahne2024}.

For the backward solution, we prove its existence and also determine
its leading-order asymptotic behavior at infinity, which is used when
computing the associated forward solution.

\begin{theorem}[Existence of a backward self-similar solution]\label{thm:backward-existence}
  For \(\epsilon \in \codenumber{[0.1681 \pm 10^{-15}]}\), the CGL
  equation has a backward self-similar solution. It is associated with
  a zero of the function \(G(\mu, \gamma, \kappa)\)~\eqref{eq:G},
  using \(\xi_{1} = \codenumber{16}\), with \((\mu, \gamma, \kappa)\)
  contained in
  \begin{align*}
    \mu &\in \resultnumber{2.2600773707_{82}^{95}},\\
    \gamma &\in \resultnumber{1.0571_{81}^{95} - 1.4043_{175}^{298} i},\\
    \kappa &\in \resultnumber{0.7912415632_{49}^{52}}.
  \end{align*}
  Moreover, as \(\xi \to \infty\), the profile satisfies
  \(Q(\xi) = p_{Q}\xi^{-1 - i\frac{\omega}{\kappa}} +
  \mathcal{O}(\xi^{-3})\), with
  \begin{equation*}
    p_{Q} \in \resultnumber{1.061_{84}^{93} - 1.400_{73}^{82} i}.
  \end{equation*}
\end{theorem}

\begin{proof}
  The first step in applying the interval Newton method is to find a
  (non-rigorous) numerical approximation. As a starting point, we use
  a rough approximation given by
  \begin{equation*}
    \mu = \codenumber{2.2600451},\quad
    \kappa = \codenumber{0.7912471} \quad\text{and}\quad
    \gamma = \frac{Q_{0}(\mu, \kappa; \xi_{1})}{P(\xi_{1})},
  \end{equation*}
  with \(P\) as in~\cite[Section~7]{Dahne2024}. These approximations
  for \(\mu\) and \(\kappa\) were originally found by following the
  first branch of solutions in \(\epsilon\), as discussed
  in~\cite[Section~4.1]{Dahne2024}, but at this point, we treat them
  as fixed. This initial approximation is then refined using a
  (non-rigorous) Newton--Raphson scheme implemented in the Julia
  package \texttt{NonlinearSolve.jl}~\cite{Pal2024}. This gives us
  (rounded to 16 digits) the approximations
  \begin{align*}
    \mu_{0} &= \resultnumber{2.260077370790754},\\
    \gamma_{0} &= \resultnumber{1.057188097010809 - 1.404323666369636i},\\
    \kappa_{0} &= \resultnumber{0.7912415632503764}.
  \end{align*}

  To apply the interval Newton method, we need to take a box \(X\)
  enclosing this approximation. This box is found using the heuristic
  method described in~\cite[Section~10]{Dahne2024}. The result is the
  existence of a locally unique zero enclosed in the intervals given
  in the statement of the theorem. Figure~\ref{fig:Q} shows an
  approximation of the profile \(Q\) associated with these parameters.

  The enclosure for \(p_{Q}\) is based on Lemma~\ref{lemma:p_Q_0},
  from which we have
  \begin{equation*}
    p_{Q} = c^{-a}p_{Q,0} = c^{-a}(\gamma + I_{E,\infty}).
  \end{equation*}
  To enclose \(I_{E,\infty}\), we make use of
  Lemmas~\ref{lemma:I_E_infty} and~\ref{lemma:I_E_infty-integral}.
  This results in the enclosure given in the statement of the theorem.
\end{proof}

\begin{figure}
  \centering
  \begin{subfigure}[t]{0.45\textwidth}
    \includegraphics[width=\textwidth]{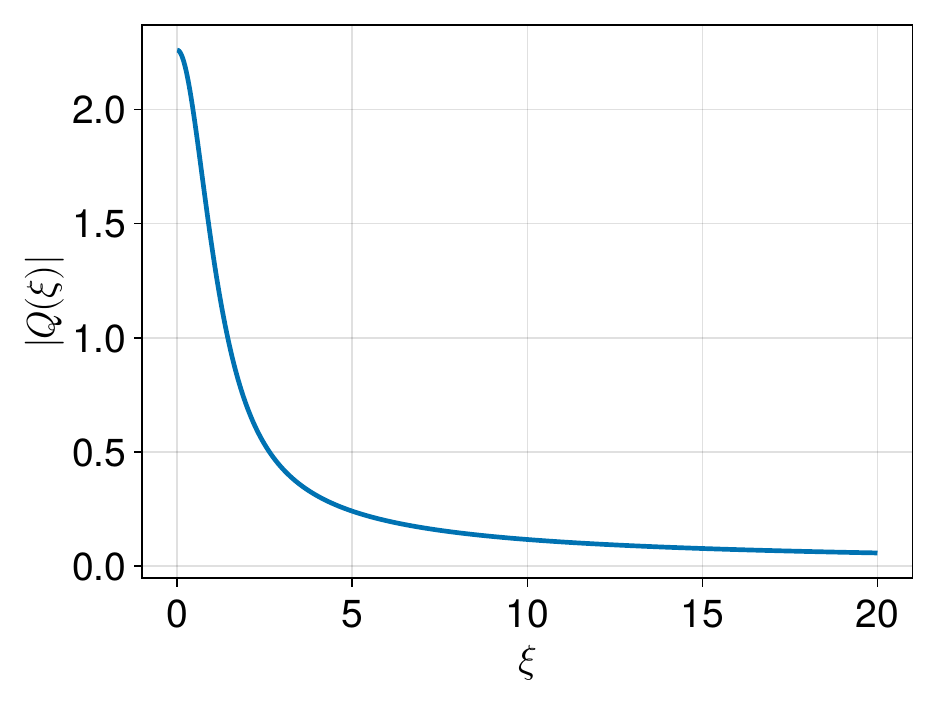}
    \caption{}
    \label{fig:Q}
  \end{subfigure}
  \hspace{0.05\textwidth}
  \begin{subfigure}[t]{0.45\textwidth}
    \includegraphics[width=\textwidth]{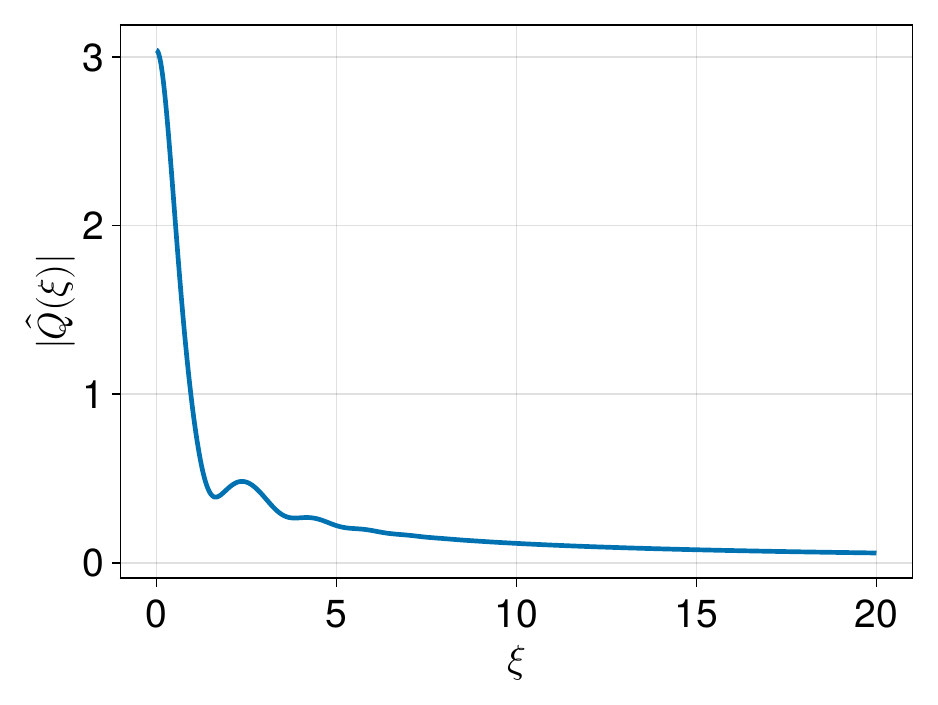}
    \caption{}
    \label{fig:Q-hat}
  \end{subfigure}
  \caption{Approximations of the profiles \(Q\) (left) and \(\hat{Q}\)
    (right) from Theorems~\ref{thm:backward-existence}
    and~\ref{thm:forward-existence}.}
\end{figure}

The next step is to compute a forward self-similar solution with the
same asymptotic behavior as the backward self-similar solution in the
previous theorem.

\begin{theorem}[Existence of a forward self-similar solution]\label{thm:forward-existence}
  For \(\epsilon \in \codenumber{[0.1681 \pm 10^{-15}]}\), the CGL
  equation has a forward self-similar solution with the same
  asymptotic behavior as the backward self-similar solution from
  Theorem~\ref{thm:backward-existence}. It is associated with a zero
  of the function \(\hat{G}(\nu, \hat{\gamma}_{2})\)~\eqref{eq:G-hat},
  using \(\xi_{1} = \codenumber{16}\) and
  \(\hat{\gamma}_{1} \in \resultnumber{0.192_{38}^{41} +
    0.1458_{43}^{55} i}\), with \((\nu, \hat{\gamma}_{2})\) contained
  in
  \begin{align*}
    \nu &\in \resultnumber{2.190_{28}^{76} + 2.1_{09}^{11} i},\\
    \hat{\gamma}_{2} &\in \resultnumber{-1_{058.6}^{184.6} + [-1053.2, -922.1] i}.
  \end{align*}
\end{theorem}

\begin{proof}
  The first step is to determine \(\hat{\gamma}_{1}\) such that
  \(\hat{Q}_{\infty}\) has the same asymptotic behavior as \(Q\). From
  Lemma~\ref{lemma:Q-hat-leading-term}, we have that
  \begin{equation*}
    \hat{Q}(\xi) = p_{\hat{Q}}\xi^{-1 - i\frac{\omega}{\kappa}} + \mathcal{O}(\xi^{-3}),
  \end{equation*}
  with \(p_{\hat{Q}} = \hat{\gamma}_{1}(-c)^{-a}\). We therefore let
  \(\hat{\gamma}_{1} = p_{Q} / (-c)^{-a} = (-1)^{a}p_{Q,0}\). Here the
  last equality uses \(p_{Q} = c^{-a}p_{Q,0}\) from
  Lemma~\ref{lemma:p_Q_0} together with
  \(c^{-a} = (-1)^{a}(-c)^{-a}\), which holds for the principal branch
  since \(\imag(c) < 0\). Based on the enclosures from
  Theorem~\ref{thm:backward-existence}, this gives us
  \begin{equation*}
    \hat{\gamma}_{1} \in \resultnumber{0.192_{38}^{41} + 0.1458_{43}^{55} i}.
  \end{equation*}

  To find a numerical approximation, we apply the same Newton--Raphson
  method as in Theorem~\ref{thm:backward-existence}, this time to the
  initial guess \((\nu, \hat{\gamma}_{2}) = (0, 0)\). This gives us the
  approximations
  \begin{equation*}
    \nu_{0} = \resultnumber{2.190519866535078 + 2.110080808821536i},\quad
    \hat{\gamma}_{2,0} = \resultnumber{-1121.578763109324 - 987.6373181813229i},
  \end{equation*}
  rounded to 16 digits. As in Theorem~\ref{thm:backward-existence}, we
  then apply the interval Newton method to a box around this
  approximation. The resulting enclosure is the one given in the
  statement of the theorem. Figure~\ref{fig:Q-hat} shows an
  approximation of the profile \(\hat{Q}\) associated with these
  parameters.
\end{proof}

Finally, we prove the existence of an unstable eigenvalue for the
linear operator associated with the forward self-similar solution.

\begin{theorem}[Existence of an unstable eigenvalue]\label{thm:unstable-eigenvalue}
  For \(\epsilon \in \codenumber{[0.1681 \pm 10^{-15}]}\), the linear
  operator \(L_{\hat{Q}}\)~\eqref{eq:L_Q-hat} associated with the
  forward self-similar solution from
  Theorem~\ref{thm:forward-existence} has a locally unique unstable
  eigenvalue \(\lambda\) satisfying
  \begin{equation*}
    \lambda \in \resultnumber{0.0_{02294}^{22295} + 1._{5997}^{6198} i}.
  \end{equation*}
  It is associated with a zero of the function
  \(H(\lambda)\)~\eqref{eq:H}, using
  \(\xi_{1} = \codenumber{16}\).
\end{theorem}

\begin{proof}
  To compute a first estimate of \(\lambda\), we use a
  finite-difference method that is described in more detail in
  Appendix~\ref{sec:finite-differences-approximations}. Rounded to 8
  digits, this gives us the approximation
  \begin{equation*}
    \lambda_{0} = \resultnumber{0.012294299 + 1.6097332i}.
  \end{equation*}

  In principle, one could apply the interval Newton method to prove
  the existence of a zero in the neighborhood of this approximation.
  However, due to the accumulation of overestimations in the interval
  enclosures when determining \(Q\) and then \(\hat{Q}\), and finally
  when evaluating \(H\), the resulting enclosures are insufficient to
  close the argument using an interval Newton iteration. Instead, we
  apply a winding number argument, which is more robust against the
  overestimations in the enclosures.

  To apply the winding argument, the first step is to take a contour
  \(\Gamma\) around the approximate solution. We take the contour to
  be given by a square centered at the approximation, with side length
  \(2r\), where \(r = \codenumber{0.01}\). The resulting square is
  shown in Figure~\ref{fig:contour} and is contained in
  \(\resultnumber{0.0_{02294}^{22295} + 1._{5997}^{6198} i}\). We
  parametrize the contour by a piecewise-linear function
  \(z(t) \colon [0, 4] \to \mathbb{C}\), taken so that \(z(0)\) is the
  lower left corner of the square, \(z(1)\) is the lower right corner,
  \(z(2)\) is the upper right corner and \(z(3)\) is the upper left
  corner.

  \begin{figure}
    \centering
    \begin{subfigure}[t]{0.45\textwidth}
      \includegraphics[width=\textwidth]{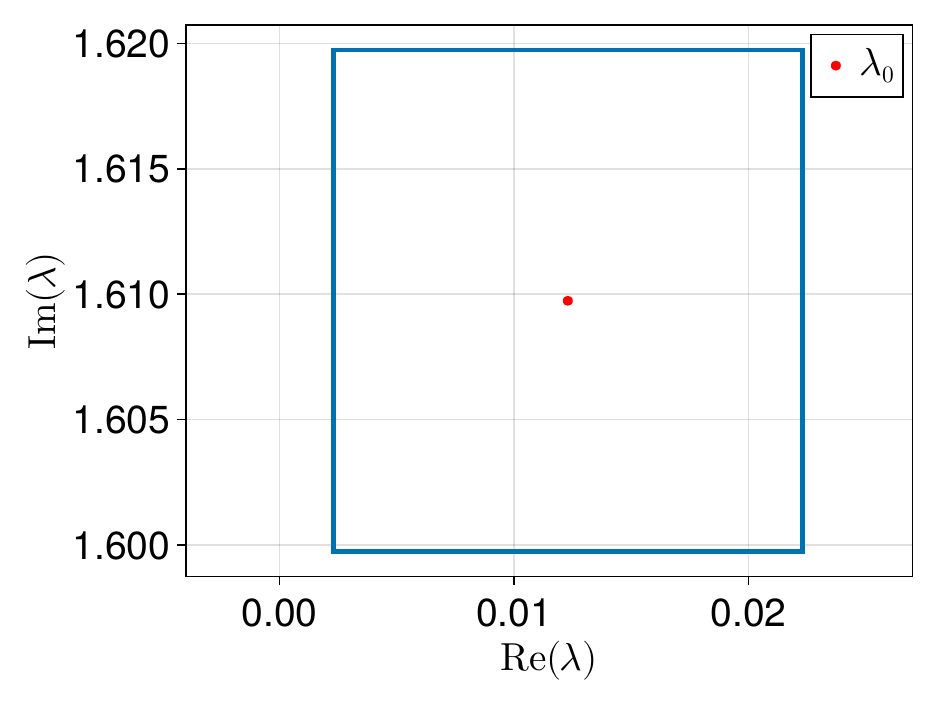}
      \caption{The square given by the contour \(\Gamma\).}
      \label{fig:contour}
    \end{subfigure}
    \hspace{0.05\textwidth}
    \begin{subfigure}[t]{0.45\textwidth}
      \includegraphics[width=\textwidth]{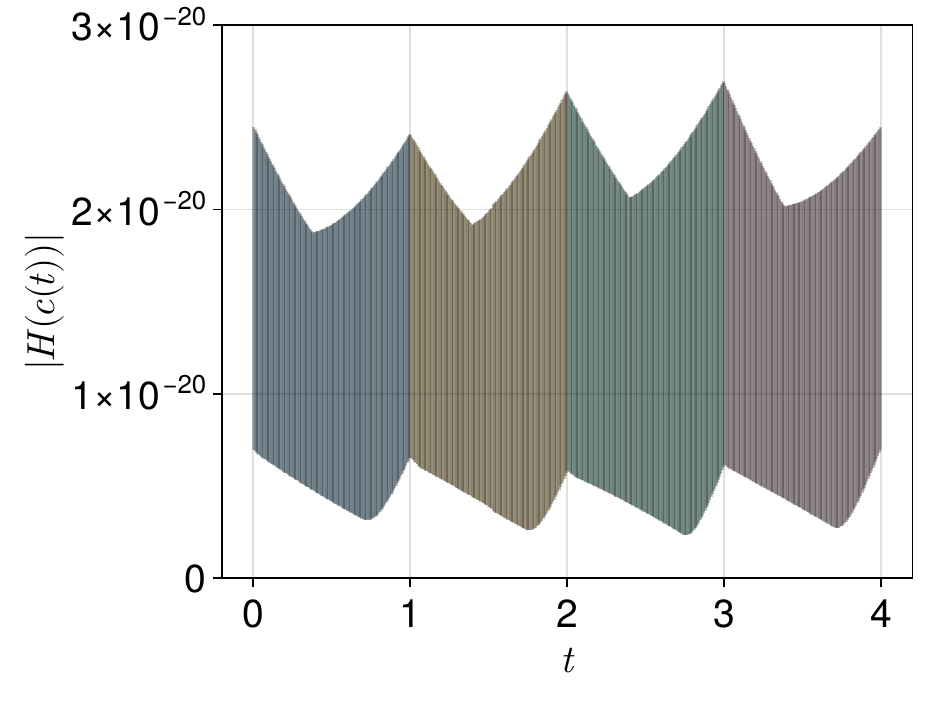}
      \caption{Enclosure of \(|H(z(t))|\).}
      \label{fig:H-abs}
    \end{subfigure}
    \caption{}
    \label{fig:contour-H-abs}
  \end{figure}

  If \(H\) is analytic inside the contour \(\Gamma\), Cauchy's
  argument principle implies that the number of zeros of \(H\) inside
  the contour, counted with multiplicity, is exactly given by the
  winding number of the path \(H(z(t))\) for \(t \in [0, 4]\). Proving
  the existence of a unique zero inside \(\Gamma\) thus reduces to
  proving that \(H\) is analytic and that the winding number is one.

  To prove analyticity of \(H\), it suffices to show that the
  components of the matrix in~\eqref{eq:H} are all analytic in
  \(\lambda\). For \(Y_{0,1}\), \(Y_{0,2}\), \(Y_{0,1}'\) and
  \(Y_{0,2}'\), this is immediate from the fact that they are the
  solutions to an initial value problem that depends analytically on
  \(\lambda\). To prove analyticity of \(Y_{\infty,1}\),
  \(Y_{\infty,2}\), \(Y_{\infty,1}'\) and \(Y_{\infty,2}'\), we make
  use of Proposition~\ref{prop:Z-fixed-point}. To apply the
  proposition, it suffices to verify~\eqref{eq:T_12-ineq-2} for all
  \(\lambda\) inside the contour, namely the inequality
  \begin{equation*}
    C_{T_{12}}\xi_{1}^{-2} < 1.
  \end{equation*}
  The inequality~\eqref{eq:T_12-ineq-1} can then always be satisfied
  by taking \(\rho\) sufficiently large. To verify the above
  inequality, we simply compute an enclosure of
  \(C_{T_{12}} = C_{T_{12}}(\lambda)\) inside the contour, giving us
  \begin{equation*}
    C_{T_{12}}\xi_{1}^{-2} \in \resultnumber{0.0_{09}^{14}}.
  \end{equation*}
  Note that \(C_{T_{12}}\) is independent of the \(c_{0}\) parameter
  in \(Y_{\infty}\), so this covers both \(Y_{\infty,1}\) and
  \(Y_{\infty,2}\).

  To compute the winding number of \(H\) along \(\Gamma\), we start by
  computing an enclosure of the path \(H(z(t))\) for \(t \in [0, 4]\).
  For this, we first split the interval into four parts: \([0, 1]\),
  \([1, 2]\), \([2, 3]\) and \([3, 4]\), corresponding to the
  different sides of the square. Each of these intervals is then
  further subdivided into \(n = \codenumber{128}\) pieces. This gives
  us a total of \(4n\) complex intervals for which we compute an
  enclosure of \(H(z(t))\) using interval arithmetic.

  Figure~\ref{fig:H-abs} shows an enclosure of \(|H(z(t))|\) along the
  path, and Figure~\ref{fig:H-arg} shows an enclosure of
  \(\arg(H(z(t)))\). Due to the discontinuity of the argument along
  the negative real axis, we include two versions of the plot for
  \(\arg(H(z(t)))\). The first one shows the principal argument
  (\(\operatorname{Arg}(H(z(t)))\)), where we can see that the
  discontinuity is confined to \(t \in [2, 3]\). For the second
  version, we adjust for this discontinuity to get a continuous value;
  for \(t \in [2, 3]\), this is done by computing
  \(\operatorname{Arg}(-H(z(t))) + \pi\), and for \(t \in [3, 4]\), by
  computing \(\operatorname{Arg}(H(z(t))) + 2\pi\).

  \begin{figure}
    \centering
    \begin{subfigure}[t]{0.45\textwidth}
      \includegraphics[width=\textwidth]{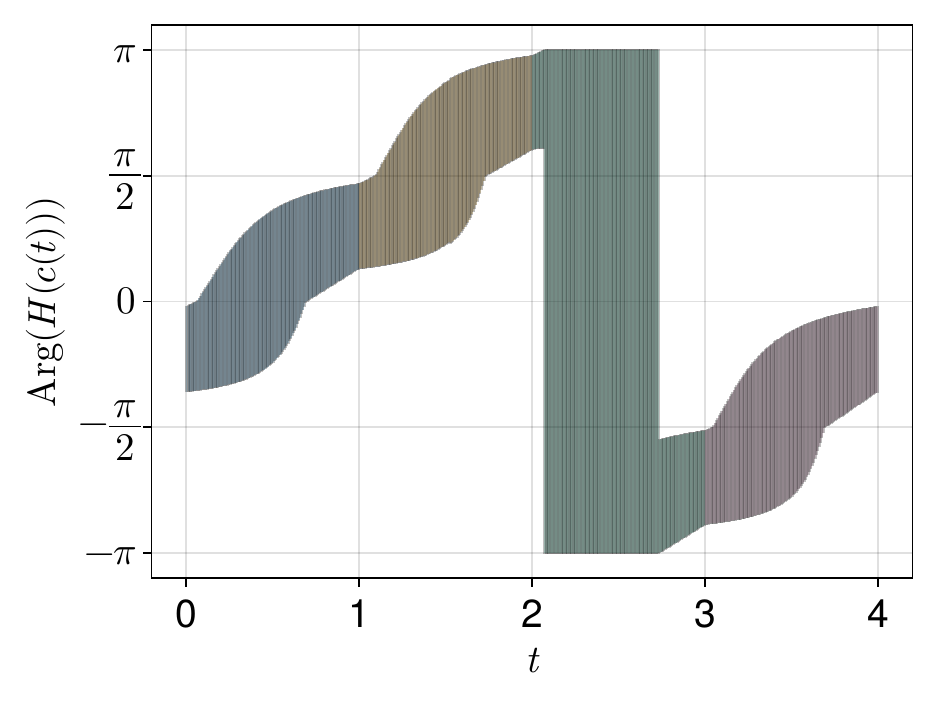}
      \caption{}
      \label{fig:H-arg-principal}
    \end{subfigure}
    \hspace{0.05\textwidth}
    \begin{subfigure}[t]{0.45\textwidth}
      \includegraphics[width=\textwidth]{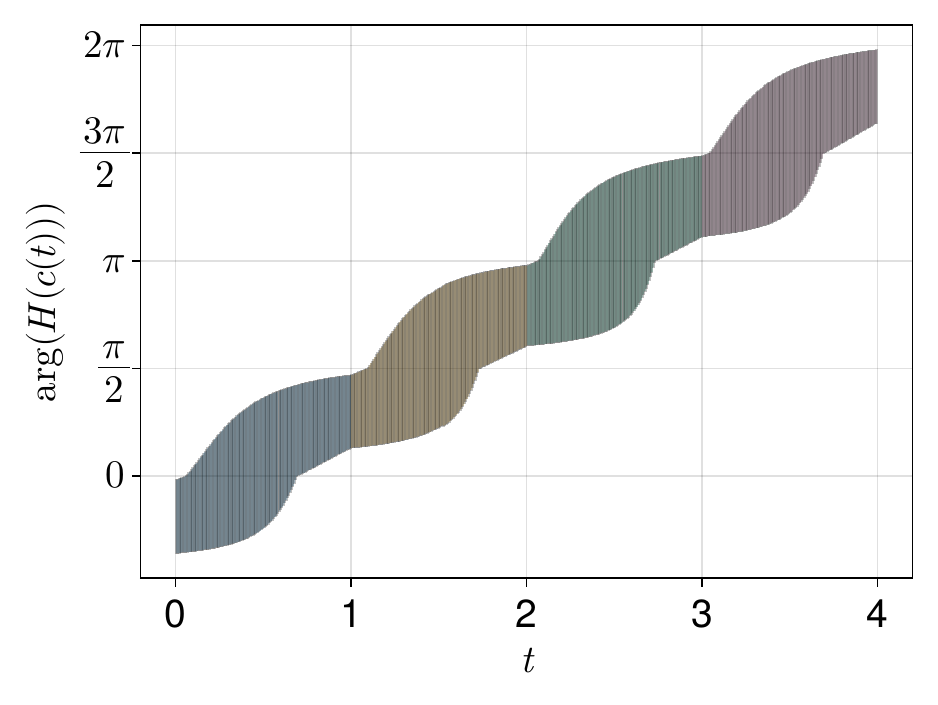}
      \caption{}
      \label{fig:H-arg-continuous}
    \end{subfigure}
    \caption{Enclosure of \(\arg(H(z(t)))\). The left figure shows the
      principal argument, \(\operatorname{Arg}(H(z(t)))\), which has a
      discontinuity along the negative real axis. The right figure
      shows the argument when adjusted to be continuous along the
      path.}
    \label{fig:H-arg}
  \end{figure}

  While Figure~\ref{fig:H-arg-continuous} contains the details needed
  to prove that the winding number is one, we also provide a more
  formal verification that \(H(z(t))\) encircles the origin exactly
  once. We start by verifying the following four properties:
  \begin{enumerate}
  \item For \(t \in [0, 1]\), the curve \(H(z(t))\) has positive real
    part;
  \item For \(t \in [1, 2]\), the curve \(H(z(t))\) has positive
    imaginary part;
  \item For \(t \in [2, 3]\), the curve \(H(z(t))\) has negative real
    part;
  \item For \(t \in [3, 4]\), the curve \(H(z(t))\) has negative
    imaginary part.
  \end{enumerate}
  In particular, \(H(z(t)) \neq 0\) along the whole contour, so it has
  a continuous argument \(\theta(t)\). Each of the four properties
  confines the curve to an open half-plane, on which \(\theta\) varies
  by less than \(\pi\). At the corners, the properties combine
  pairwise, placing \(H(z(t))\) in the fourth quadrant for \(t = 0\),
  the first for \(t = 1\), the second for \(t = 2\) and the third for
  \(t = 3\). Normalizing so that
  \(\theta(0) \in (-\frac{\pi}{2}, 0)\), these two observations give
  in turn \(\theta(1) \in (0, \frac{\pi}{2})\),
  \(\theta(2) \in (\frac{\pi}{2}, \pi)\),
  \(\theta(3) \in (\pi, \frac{3\pi}{2})\) and
  \(\theta(4) \in (\frac{3\pi}{2}, 2\pi)\). Hence
  \(\theta(4) - \theta(0) \in (\frac{3\pi}{2}, \frac{5\pi}{2})\), and
  since it is an integer multiple of \(2\pi\), it equals \(2\pi\). The
  winding number is therefore exactly one and, by Cauchy's argument
  principle, \(H\) has a unique zero inside the square, giving us the
  enclosure of the eigenvalue \(\lambda\) of \(L_{\hat{Q}}\) stated in
  the theorem. In particular, since the real part of this enclosure is
  positive, the eigenvalue is unstable.
\end{proof}

\section{Forward solution at infinity}
\label{sec:forward-solution-infinity}

In this section, we adapt the approach for computing enclosures near
spatial infinity of backward self-similar solutions
from~\cite[Section~7]{Dahne2024} to handle forward self-similar
solutions. To stay as close as possible to the notation
in~\cite{Dahne2024}, we keep the strength of the nonlinearity,
\(\sigma\), and the spatial dimension, \(d\), symbolic throughout this
section. In our application they take the values \(\sigma = 1\) and
\(d = 3\). In this notation, Equation~\eqref{eq:Q-hat} for the
forward self-similar solutions reads
\begin{equation}\label{eq:Q-hat-infty-equation}
  (1 - i\epsilon)\left(\hat{Q}'' + \frac{d - 1}{\xi}\hat{Q}'\right) - i\kappa\xi \hat{Q}'
  - i \frac{\kappa}{\sigma}\hat{Q} + \omega \hat{Q} + (1 + i\delta)|\hat{Q}|^{2\sigma}\hat{Q} = 0,
\end{equation}
where we are interested in solutions satisfying the condition
\(\hat{Q}(\xi) \sim \xi^{-\frac{1}{\sigma} - i\frac{\omega}{\kappa}}\).

As in~\cite{Dahne2024}, the main idea is to apply variation of
parameters to the linearized equation to construct an operator whose
fixed points are solutions to the equation. The necessary bounds are
then derived from this fixed-point equation. Compared to the
presentation in~\cite{Dahne2024}, there are three main differences:
\begin{enumerate}
\item The linear equation for the forward solution has a solution that
  converges exponentially fast to zero, rather than diverging
  exponentially as in the backward case.
\item We only treat the CGL case, which means we can assume that
  \(\epsilon > 0\). This simplifies many of the bounds that otherwise
  deteriorate in the limit as \(\epsilon \to 0\).
\item We do not need derivatives with respect to the parameters
  \(\kappa\) or \(\epsilon\).
\end{enumerate}
Combined, all of these serve to significantly reduce the complexity of
acquiring the necessary bounds.

Since the only difference between the backward and forward equations is
in the sign of \(\kappa\) and \(\omega\), the structure mostly remains
the same. To highlight this correspondence, we use mostly the same
notation as in~\cite{Dahne2024}, but with hats added over the forward
versions.

For the forward solutions, the associated linear equation is given by
\begin{equation}\label{eq:Q-hat-infty-linear}
  (1 - i\epsilon)\left(\hat{Q}'' + \frac{d - 1}{\xi}\hat{Q}'\right) - i\kappa\xi \hat{Q}' - i \frac{\kappa}{\sigma}\hat{Q} + \omega \hat{Q} = 0.
\end{equation}
Making the change of variables
\begin{equation}
  \label{eq:abc}
  a = \frac{1}{2}\left(\frac{1}{\sigma} + i \frac{\omega}{\kappa}\right),\quad
  b = \frac{d}{2},\quad
  c = \frac{-i \kappa}{2(1 - i\epsilon)},\quad
  \hat{z} = -c\xi^{2},
\end{equation}
this can be transformed into the well-studied Kummer equation,
\begin{equation*}
  \hat{z} \frac{d^{2}w}{d\hat{z}^{2}} + (b - \hat{z})\frac{dw}{d\hat{z}} - aw = 0.
\end{equation*}
Note that the only difference compared to the backward case is in the
sign of \(\hat{z}\). One solution of Kummer's equation is the
confluent hypergeometric function \(U(a, b, \hat{z})\), with another,
linearly independent, solution given by
\(V(a, b, \hat{z}) = e^{\hat{z}}U(b - a, b, -\hat{z})\). This gives us
two linearly independent solutions \(\hat{P}(\xi)\) and
\(\hat{E}(\xi)\) of~\eqref{eq:Q-hat-infty-linear}, given by
\begin{equation*}
  \hat{P}(\xi) = U(a, b, -c\xi^{2}) \quad\text{and}\quad
  \hat{E}(\xi) = e^{-c\xi^{2}}U(b - a, b, c\xi^{2}).
\end{equation*}
We also note that the associated Wronskian is given by
\begin{equation*}
  \hat{W}(\xi) = \hat{P}(\xi)\hat{E}'(\xi) - \hat{P}'(\xi)\hat{E}(\xi) = -2ce^{-\sign(\imag(c)) \pi i (b - a)}\xi \left(-c\xi^{2}\right)^{-b}e^{-c\xi^{2}}.
\end{equation*}
Here the function \(\hat{P}(\xi)\) has the right asymptotic behavior
at infinity, whereas \(\hat{E}(\xi)\) decays exponentially.

Following the method of variation of parameters, we look for solutions
to~\eqref{eq:Q-hat-infty-equation} of the form
\begin{equation*}
  \hat{Q}(\xi) = \hat{c}_{1}(\xi)\hat{P}(\xi) + \hat{c}_{2}(\xi)\hat{E}(\xi),
\end{equation*}
where
\begin{equation*}
  \hat{c}_{1}'(\xi) = \frac{1 + i\delta}{1 - i \epsilon}\hat{E}(\xi)\hat{W}(\xi)^{-1}|\hat{Q}(\xi)|^{2\sigma}\hat{Q}(\xi)
  \quad\text{and}\quad
  \hat{c}_{2}'(\xi) = -\frac{1 + i\delta}{1 - i \epsilon}\hat{P}(\xi)\hat{W}(\xi)^{-1}|\hat{Q}(\xi)|^{2\sigma}\hat{Q}(\xi).
\end{equation*}
To get the right asymptotic behavior at infinity, we need
\(\hat{c}_{1}\) to have a non-zero limit at infinity and
\(\hat{c}_{2}\hat{E}\) to be of lower order than \(\hat{P}\). It is
therefore natural to use the representations
\begin{align*}
  \hat{c}_{1}(\xi) &= \hat{\gamma}_{1} - \int_{\xi}^{\infty}\frac{1 + i\delta}{1 - i \epsilon}\hat{E}(\eta)\hat{W}(\eta)^{-1}|\hat{Q}(\eta)|^{2\sigma}\hat{Q}(\eta)\,d\eta,\\
  \hat{c}_{2}(\xi) &= \hat{\gamma}_{2} - \int_{\xi_{1}}^{\xi}\frac{1 + i\delta}{1 - i \epsilon}\hat{P}(\eta)\hat{W}(\eta)^{-1}|\hat{Q}(\eta)|^{2\sigma}\hat{Q}(\eta)\,d\eta.
\end{align*}
This ensures that \(\hat{c}_{1}\) converges to \(\hat{\gamma}_{1}\) and
allows \(\hat{c}_{2}\) to diverge. Using the notation
\begin{align*}
  \hat{J}_{P}(\xi) &= \frac{1 + i\delta}{1 - i \epsilon}\hat{P}(\xi)\hat{W}(\xi)^{-1},\\
  \hat{J}_{E}(\xi) &= \frac{1 + i\delta}{1 - i \epsilon}\hat{E}(\xi)\hat{W}(\xi)^{-1}
\end{align*}
and
\begin{align*}
  I_{\hat{P}}(\hat{Q}, \xi) &= -\int_{\xi_{1}}^{\xi} \hat{J}_{P}(\eta)|\hat{Q}(\eta)|^{2\sigma}\hat{Q}(\eta)\,d\eta,\\
  I_{\hat{E}}(\hat{Q}, \xi) &= -\int_{\xi}^{\infty} \hat{J}_{E}(\eta)|\hat{Q}(\eta)|^{2\sigma}\hat{Q}(\eta)\,d\eta,
\end{align*}
this gives us the equation
\begin{equation}\label{eq:Q-hat-infty-fixed-point-equation}
  \hat{Q}(\xi) = \hat{\gamma}_{1} \hat{P}(\xi) + \hat{\gamma}_{2}\hat{E}(\xi)
  + \hat{P}(\xi)I_{\hat{E}}(\hat{Q}, \xi) + \hat{E}(\xi)I_{\hat{P}}(\hat{Q}, \xi),
\end{equation}
for \(\hat{Q}\). Specifically, if we let \(\hat{T}\) be the operator
\begin{equation}\label{eq:T-hat}
  \hat{T}(\hat{Q})(\xi) = \hat{\gamma}_{1}\hat{P}(\xi) + \hat{\gamma}_{2}\hat{E}(\xi) + \hat{P}(\xi)I_{\hat{E}}(\hat{Q}, \xi) + \hat{E}(\xi)I_{\hat{P}}(\hat{Q}, \xi),
\end{equation}
then, by standard variation-of-parameters arguments, fixed points of
this operator correspond to solutions
of~\eqref{eq:Q-hat-infty-equation}.

As in~\cite{Dahne2024}, the strategy for computing enclosures of
\(\hat{Q}\) is to study the operator \(\hat{T}\) and the fixed-point
equation~\eqref{eq:Q-hat-infty-fixed-point-equation}. The work is
split into two parts:
\begin{enumerate}
\item The first step is to prove the existence of a fixed point of
  \(\hat{T}\) by proving that it is a contraction in a ball in a
  weighted function space. This gives both existence of a solution and
  initial bounds for its norm.
\item The second step is to compute refined enclosures of the fixed
  point using a bootstrapping approach applied to the fixed-point
  equation~\eqref{eq:Q-hat-infty-fixed-point-equation}.
\end{enumerate}
Handling the fixed point follows mostly the same procedure as for the
backward solution. The refined enclosures, however, require a slightly
different approach. For the forward solution, the main error at
\(\xi = \xi_{1}\) comes from \(I_{\hat{E}}\), whereas for the backward
solution it comes from \(I_{P}\). The integrand for \(I_{\hat{E}}\)
does not have the same oscillating behavior as for \(I_{P}\), and we
therefore cannot use integration by parts to get better enclosures.
Instead, we isolate the leading behavior of the integrand and
integrate that explicitly. Avoiding integration by parts means that we
do not need enclosures for any higher-order derivatives of
\(\hat{Q}\), and allows us to simplify the final computation of the
enclosures at \(\xi = \xi_{1}\).

The derivatives with respect to the real and imaginary parts of
\(\hat{\gamma}_{2}\) are controlled by directly differentiating the
fixed-point equation. Since the functions \(\hat{P}\) and \(\hat{E}\)
do not depend on this parameter, the resulting equation has similar
behavior to the original one.

We proceed as follows:
\begin{enumerate}
\item Section~\ref{sec:fixed-point-hat} is devoted to studying the
  operator \(\hat{T}\), and we give explicit conditions for the
  existence of a fixed point.
\item In Section~\ref{sec:fixed-point-enclosures-hat}, we discuss how
  to use the fixed-point
  equation~\eqref{eq:Q-hat-infty-fixed-point-equation} to compute
  tight enclosures of \(\hat{Q}\) and its derivatives.
\end{enumerate}

Lastly, let us comment on some of the notation and other conventions
used in this section. For the derivatives with respect to the real and
imaginary parts of \(\hat{\gamma}_{2}\), we use the notation
\(\hat{Q}_{\real\hat{\gamma}_{2}}\) and
\(\hat{Q}_{\imag\hat{\gamma}_{2}}\). For the most part, we do not make
explicit the functions' dependence on parameters such as \(\kappa\)
and \(\epsilon\). In the same way, we generally drop the \(\hat{Q}\)
argument for \(I_{\hat{P}}\) and \(I_{\hat{E}}\) defined above, i.e.,
we prefer to write just \(I_{\hat{P}}(\xi)\) instead of
\(I_{\hat{P}}(\hat{Q}, \xi)\). The following sections contain a number
of asymptotic bounds with explicitly given constants. In general,
these explicit constants will depend on parameters such as \(\kappa\)
and \(\epsilon\), as well as the value of \(\xi_{1}\). However, to
keep the notation brief, we do not make this dependence explicit. We
use the variable \(C\) for all of these types of constants, with a
subscript to denote which function it is associated with.

\subsection{Existence of a fixed point}
\label{sec:fixed-point-hat}

We perform the fixed-point argument for the operator \(\hat{T}\) in
the Banach space \(\Space_{0}\) of continuous functions
\(\hat{Q} \colon [\xi_{1}, \infty) \to \mathbb{C}\) for which the norm
\begin{equation*}
  \|\hat{Q}\|_{0} = \sup_{\xi \geq \xi_{1}} \xi^{\frac{1}{\sigma}}|\hat{Q}(\xi)|
\end{equation*}
is finite. For the backward solutions, we instead worked with a norm of
the form
\begin{equation*}
  \|Q\|_{\normv} = \sup_{\xi \geq \xi_{1}} \xi^{\frac{1}{\sigma} - \normv}|Q(\xi)|.
\end{equation*}
Here, the parameter \(\normv\) was needed to account for logarithmic
factors appearing in some of the derivatives. However, for the forward
case we do not encounter these logarithmic factors and we can
therefore take \(\normv = 0\).

The first step in bounding the operator \(\hat{T}\) is to bound the
functions \(\hat{P}\) and \(\hat{E}\). We let
\begin{equation*}
  \hat{B}_{W} = \frac{1 + i\delta}{i\kappa}e^{\sign(\imag(c))\pi i(b - a)}(-c)^{b},
\end{equation*}
so that \(\hat{J}_{P}\) and \(\hat{J}_{E}\) can be written as
\begin{equation*}
  \hat{J}_{P}(\xi) = \hat{B}_{W}\hat{P}(\xi)e^{c\xi^{2}}\xi^{d - 1}
  \quad\text{and}\quad
  \hat{J}_{E}(\xi) = \hat{B}_{W}\hat{E}(\xi)e^{c\xi^{2}}\xi^{d - 1}.
\end{equation*}
Compared to the backward solutions, we need bounds for significantly
fewer functions. The ones we need are given in the following lemma,
closely related to~\cite[Lemma~7.3]{Dahne2024}.

\begin{lemma}\label{lemma:P-hat-E-hat-bounds}
  Let \(\xi_{1} > 1\). Assuming that the parameters \(\kappa\),
  \(\omega\) and \(\epsilon\) are such that \(U(a, b, -c\xi_{1}^{2})\)
  and \(U(b - a, b, c\xi_{1}^{2})\) satisfy the conditions
  of~\cite[Lemma~7.1]{Dahne2024}, we have, for \(\xi \geq \xi_{1}\),
  the bounds
  \begin{align*}
    |\hat{P}(\xi)| &\leq C_{\hat{P}}\xi^{-\frac{1}{\sigma}},\\
    |\hat{E}(\xi)| &\leq C_{\hat{E}}e^{-\real(c)\xi^{2}}\xi^{\frac{1}{\sigma} - d},\\
    |\hat{J}_{E}(\xi)| &\leq C_{\hat{J}_{E}}\xi^{\frac{1}{\sigma} - 1},\\
    |\hat{J}_{P}(\xi)| &\leq C_{\hat{J}_{P}}e^{\real(c)\xi^{2}}\xi^{-\frac{1}{\sigma} + d - 1},
  \end{align*}
  where the constants are given by
  \begin{align*}
    C_{\hat{P}} &= |(-c)^{-a}|C_{U}(a, b, n, -c\xi_{1}^{2}),\\
    C_{\hat{E}} &= |c^{a - b}|C_{U}(b - a, b, n, c\xi_{1}^{2}),\\
    C_{\hat{J}_{E}} &= |\hat{B}_{W}|C_{\hat{E}},\\
    C_{\hat{J}_{P}} &= |\hat{B}_{W}|C_{\hat{P}}.
  \end{align*}
  Here \(C_{U}\) is as in~\cite[Lemma~7.1]{Dahne2024} and \(n\) is any
  non-negative integer.
\end{lemma}

\begin{proof}
  From the definitions of \(\hat{P}\) and \(\hat{E}\), we get
  \begin{align*}
    \left|\hat{P}(\xi)\right| &= |U(a, b, -c\xi^{2})|,\\
    \left|\hat{E}(\xi)\right| &= |e^{-c\xi^{2}}U(b - a, b, c\xi^{2})|.
  \end{align*}
  The bounds then follow from the bounds for \(U\)
  from~\cite[Lemma~7.1]{Dahne2024}. The bounds for \(\hat{J}_{E}\) and
  \(\hat{J}_{P}\) are immediate consequences of their definitions and
  the bounds for \(\hat{E}\) and \(\hat{P}\), respectively.
\end{proof}

For the rest of this section, whenever the bounds from
Lemma~\ref{lemma:P-hat-E-hat-bounds} are used, we will implicitly
assume that the parameters \(\kappa\), \(\omega\) and
\(\epsilon\) satisfy its assumptions.

We then have the following lemma that bounds the integrals
\(I_{\hat{E}}\) and \(I_{\hat{P}}\).

\begin{lemma}\label{lemma:I_E-hat-I_P-hat-bounds}
  Assume that \(\xi_{1} > 1\) and \(\real(c) > 0\), and that the
  following inequalities are satisfied:
  \begin{align*}
    -\frac{2}{\sigma} + d - 4 &< 0,\\
    2\real(c) + \left(-\frac{2}{\sigma} + d - 4\right)\xi_{1}^{-2} &> 0.
  \end{align*}
  Then, for \(\hat{Q} \in \Space_{0}\) and \(\xi \geq \xi_{1}\), we
  have the following bounds:
  \begin{align*}
    |I_{\hat{E}}(\xi)| &\leq C_{I_{\hat{E}}}\xi^{-2}\|\hat{Q}\|_{0}^{2\sigma + 1},\\
    |I_{\hat{P}}(\xi)| &\leq C_{I_{\hat{P}}}e^{\real(c)\xi^{2}}\xi^{-\frac{2}{\sigma} + d - 4}\|\hat{Q}\|_{0}^{2\sigma + 1},
  \end{align*}
  with
  \begin{equation*}
    C_{I_{\hat{E}}} = \frac{C_{\hat{J}_{E}}}{2},\quad
    C_{I_{\hat{P}}} = \frac{C_{\hat{J}_{P}}}{2\real(c) + \left(-\frac{2}{\sigma} + d - 4\right)\xi_{1}^{-2}}.
  \end{equation*}
\end{lemma}

\begin{proof}
  Using the fact that
  \begin{equation*}
    |\hat{Q}(\eta)|^{2\sigma + 1}
    \leq \|\hat{Q}\|_{0}^{2\sigma + 1}\eta^{-\frac{1}{\sigma} - 2}
  \end{equation*}
  together with the bounds for \(\hat{J}_{E}\) and \(\hat{J}_{P}\) from
  Lemma~\ref{lemma:P-hat-E-hat-bounds}, we get
  \begin{equation*}
    |I_{\hat{E}}(\xi)|
    \leq \int_{\xi}^{\infty} |\hat{J}_{E}(\eta)||\hat{Q}(\eta)|^{2\sigma + 1}\,d\eta
    \leq C_{\hat{J}_{E}}\|\hat{Q}\|_{0}^{2\sigma + 1}
    \int_{\xi}^{\infty}\eta^{\frac{1}{\sigma} - 1}\eta^{-\frac{1}{\sigma} - 2}\,d\eta
    = \frac{C_{\hat{J}_{E}}}{2}\|\hat{Q}\|_{0}^{2\sigma + 1}\xi^{-2},
  \end{equation*}
  and
  \begin{equation*}
    \begin{split}
      |I_{\hat{P}}(\xi)|
      &\leq \int_{\xi_{1}}^{\xi} |\hat{J}_{P}(\eta)||\hat{Q}(\eta)|^{2\sigma + 1}\,d\eta\\
      &\leq C_{\hat{J}_{P}}\|\hat{Q}\|_{0}^{2\sigma + 1}
        \int_{\xi_{1}}^{\xi} e^{\real(c)\eta^{2}}\eta^{-\frac{1}{\sigma} + d - 1}\eta^{-\frac{1}{\sigma} - 2}\,d\eta\\
      &= C_{\hat{J}_{P}}\|\hat{Q}\|_{0}^{2\sigma + 1}
        \int_{\xi_{1}}^{\xi} e^{\real(c)\eta^{2}}\eta^{-\frac{2}{\sigma} + d - 3}\,d\eta.
    \end{split}
  \end{equation*}
  To bound the last integral, we integrate by parts, giving us
  \begin{equation*}
    \begin{split}
      \int_{\xi_{1}}^{\xi}e^{\real(c)\eta^{2}}\eta^{-\frac{2}{\sigma} + d - 3}\,d\eta
      &= \frac{1}{2\real(c)}\left[e^{\real(c)\eta^{2}}\eta^{-\frac{2}{\sigma} + d - 4}\right]_{\xi_{1}}^{\xi}
      - \frac{-\frac{2}{\sigma} + d - 4}{2\real(c)}\int_{\xi_{1}}^{\xi}e^{\real(c)\eta^{2}}\eta^{-\frac{2}{\sigma} + d - 5}\,d\eta\\
      &\leq \frac{1}{2\real(c)}e^{\real(c)\xi^{2}}\xi^{-\frac{2}{\sigma} + d - 4}
      - \frac{-\frac{2}{\sigma} + d - 4}{2\real(c)}\xi_{1}^{-2}\int_{\xi_{1}}^{\xi}e^{\real(c)\eta^{2}}\eta^{-\frac{2}{\sigma} + d - 3}\,d\eta.
    \end{split}
  \end{equation*}
  Letting
  \begin{equation*}
    I = \int_{\xi_{1}}^{\xi}e^{\real(c)\eta^{2}}\eta^{-\frac{2}{\sigma} + d - 3}\,d\eta
  \end{equation*}
  gives
  \begin{equation*}
    I \leq \frac{1}{2\real(c)}e^{\real(c)\xi^{2}}\xi^{-\frac{2}{\sigma} + d - 4}
    - \frac{-\frac{2}{\sigma} + d - 4}{2\real(c)}\xi_{1}^{-2}I.
  \end{equation*}
  Since, by assumption,
  \(2\real(c) + \left(-\frac{2}{\sigma} + d - 4\right)\xi_{1}^{-2} >
  0\), it follows that
  \begin{equation*}
    -\frac{-\frac{2}{\sigma} + d - 4}{2\real(c)}\xi_{1}^{-2} < 1,
  \end{equation*}
  and hence we can solve for \(I\), giving us
  \begin{equation*}
    I \leq \frac{1}{2\real(c) + (-\frac{2}{\sigma} + d - 4)\xi_{1}^{-2}}e^{\real(c)\xi^{2}}\xi^{-\frac{2}{\sigma} + d - 4}.
  \end{equation*}
\end{proof}

For the existence of a fixed point, we have the following adaptation
of~\cite[Lemma~A.2]{Plech2001} and~\cite[Lemma~7.4]{Dahne2024}.

\begin{lemma}\label{lemma:Q-hat-fixed-point-bounds}
  Under the assumptions of Lemma~\ref{lemma:I_E-hat-I_P-hat-bounds}
  and \(2 / \sigma - d < 0\), the operator \(\hat{T}\) defines a
  continuous mapping \(\hat{T} \colon \Space_{0} \to \Space_{0}\).
  Moreover,
  \begin{equation}\label{eq:T-1-hat}
    \|\hat{T}(\hat{Q})\|_{0}
    \leq C_{\hat{P}}|\hat{\gamma}_{1}|
    + C_{\hat{E}}|\hat{\gamma}_{2}|e^{-\real(c)\xi_{1}^{2}}\xi_{1}^{\frac{2}{\sigma} - d}
    + C_{\hat{T}}\xi_{1}^{-2}\|\hat{Q}\|_{0}^{2\sigma + 1}
  \end{equation}
  and
  \begin{equation}\label{eq:T-2-hat}
    \|\hat{T}(\hat{Q}) - \hat{T}(\hat{Z})\|_{0} \leq M_{\sigma}C_{\hat{T}}\|\hat{Q} - \hat{Z}\|_{0}(\|\hat{Q}\|_{0}^{2\sigma} + \|\hat{Z}\|_{0}^{2\sigma})\xi_{1}^{-2},
  \end{equation}
  for all \(\hat{Q}, \hat{Z} \in \Space_{0}\). Here,
  \begin{equation*}
    C_{\hat{T}} = C_{\hat{P}}C_{I_{\hat{E}}} + C_{\hat{E}}C_{I_{\hat{P}}}\xi_{1}^{-2}
  \end{equation*}
  and \(M_{\sigma}\) is as in~\cite[Lemma~A.7]{Dahne2024}.
\end{lemma}

\begin{proof}
  For~\eqref{eq:T-1-hat}, we want to bound
  \(\xi^{\frac{1}{\sigma}}|\hat{T}(\hat{Q})(\xi)|\). Bounding it
  termwise, we get
  \begin{equation*}
    \xi^{\frac{1}{\sigma}}|\hat{T}(\hat{Q})(\xi)| \leq
    \xi^{\frac{1}{\sigma}}|\hat{\gamma}_{1}| |\hat{P}(\xi)|
    + \xi^{\frac{1}{\sigma}}|\hat{\gamma}_{2}| |\hat{E}(\xi)|
    + \xi^{\frac{1}{\sigma}}|\hat{P}(\xi)||I_{\hat{E}}(\xi)|
    + \xi^{\frac{1}{\sigma}}|\hat{E}(\xi)||I_{\hat{P}}(\xi)|.
  \end{equation*}
  From Lemma~\ref{lemma:P-hat-E-hat-bounds}, we get for the first two
  terms
  \begin{equation*}
    \xi^{\frac{1}{\sigma}}|\hat{\gamma}_{1}| |\hat{P}(\xi)| \leq C_{\hat{P}}|\hat{\gamma}_{1}|
  \end{equation*}
  and
  \begin{equation*}
    \xi^{\frac{1}{\sigma}}|\hat{\gamma}_{2}| |\hat{E}(\xi)|
    \leq C_{\hat{E}}|\hat{\gamma}_{2}|e^{-\real(c)\xi^{2}}\xi^{\frac{2}{\sigma} - d}
    \leq C_{\hat{E}}|\hat{\gamma}_{2}|e^{-\real(c)\xi_{1}^{2}}\xi_{1}^{\frac{2}{\sigma} - d},
  \end{equation*}
  where we have used that \(2 / \sigma - d < 0\) and \(\real(c) > 0\)
  to bound the last term at \(\xi_{1}\). Using
  Lemmas~\ref{lemma:P-hat-E-hat-bounds}
  and~\ref{lemma:I_E-hat-I_P-hat-bounds}, we get for the other two
  terms that
  \begin{equation*}
    \xi^{\frac{1}{\sigma}}|\hat{P}(\xi)||I_{\hat{E}}(\xi)|
    \leq C_{\hat{P}}C_{I_{\hat{E}}}\|\hat{Q}\|_{0}^{2\sigma + 1}\xi_{1}^{-2}
  \end{equation*}
  and
  \begin{equation*}
    \xi^{\frac{1}{\sigma}}|\hat{E}(\xi)||I_{\hat{P}}(\xi)|
    \leq C_{\hat{E}}C_{I_{\hat{P}}}\|\hat{Q}\|_{0}^{2\sigma + 1}\xi_{1}^{-4}.
  \end{equation*}
  Combining all these bounds gives us~\eqref{eq:T-1-hat}.

  For~\eqref{eq:T-2-hat}, we get
  \begin{multline*}
    \xi^{\frac{1}{\sigma}}|\hat{T}(\hat{Q})(\xi) - \hat{T}(\hat{Z})(\xi)|
    \leq \xi^{\frac{1}{\sigma}}|\hat{P}(\xi)|
    \int_{\xi}^{\infty}|\hat{J}_{E}(\eta)|\left||\hat{Q}(\eta)|^{2\sigma}\hat{Q}(\eta) - |\hat{Z}(\eta)|^{2\sigma}\hat{Z}(\eta)\right|\,d\eta\\
    + \xi^{\frac{1}{\sigma}}|\hat{E}(\xi)|
    \int_{\xi_{1}}^{\xi}|\hat{J}_{P}(\eta)|\left||\hat{Q}(\eta)|^{2\sigma}\hat{Q}(\eta) - |\hat{Z}(\eta)|^{2\sigma}\hat{Z}(\eta)\right|\,d\eta.
  \end{multline*}
  By~\cite[Lemma~A.7]{Dahne2024}, we have the bound
  \begin{equation*}
    \left||\hat{Q}(\eta)|^{2\sigma}\hat{Q}(\eta) - |\hat{Z}(\eta)|^{2\sigma}\hat{Z}(\eta)\right|
    \leq M_{\sigma} |\hat{Q}(\eta) - \hat{Z}(\eta)|(|\hat{Q}(\eta)|^{2\sigma} + |\hat{Z}(\eta)|^{2\sigma}),
  \end{equation*}
  giving us
  \begin{equation*}
    \left||\hat{Q}(\eta)|^{2\sigma}\hat{Q}(\eta) - |\hat{Z}(\eta)|^{2\sigma}\hat{Z}(\eta)\right|
    \leq M_{\sigma} \|\hat{Q} - \hat{Z}\|_{0}(\|\hat{Q}\|_{0}^{2\sigma} + \|\hat{Z}\|_{0}^{2\sigma})\eta^{-\frac{1}{\sigma} - 2}.
  \end{equation*}
  Together with Lemma~\ref{lemma:P-hat-E-hat-bounds}, this gives us
  \begin{equation*}
    \int_{\xi}^{\infty}|\hat{J}_{E}(\eta)|\left||\hat{Q}(\eta)|^{2\sigma}\hat{Q}(\eta) - |\hat{Z}(\eta)|^{2\sigma}\hat{Z}(\eta)\right|\,d\eta
    \leq C_{\hat{J}_{E}}M_{\sigma} \|\hat{Q} - \hat{Z}\|_{0}(\|\hat{Q}\|_{0}^{2\sigma} + \|\hat{Z}\|_{0}^{2\sigma})
    \int_{\xi}^{\infty}\eta^{-3}\,d\eta
  \end{equation*}
  and
  \begin{multline*}
    \int_{\xi_{1}}^{\xi}|\hat{J}_{P}(\eta)|\left||\hat{Q}(\eta)|^{2\sigma}\hat{Q}(\eta) - |\hat{Z}(\eta)|^{2\sigma}\hat{Z}(\eta)\right|\,d\eta\\
    \leq C_{\hat{J}_{P}}M_{\sigma} \|\hat{Q} - \hat{Z}\|_{0}(\|\hat{Q}\|_{0}^{2\sigma} + \|\hat{Z}\|_{0}^{2\sigma})
    \int_{\xi_{1}}^{\xi}e^{\real(c)\eta^{2}}\eta^{-\frac{2}{\sigma} + d - 3}\,d\eta.
  \end{multline*}
  The remaining integrals can be bounded as in
  Lemma~\ref{lemma:I_E-hat-I_P-hat-bounds}, giving us the bounds
  \begin{equation*}
    \int_{\xi}^{\infty}|\hat{J}_{E}(\eta)|\left||\hat{Q}(\eta)|^{2\sigma}\hat{Q}(\eta) - |\hat{Z}(\eta)|^{2\sigma}\hat{Z}(\eta)\right|\,d\eta
    \leq C_{I_{\hat{E}}}M_{\sigma} \|\hat{Q} - \hat{Z}\|_{0}(\|\hat{Q}\|_{0}^{2\sigma} + \|\hat{Z}\|_{0}^{2\sigma})
    \xi^{-2}
  \end{equation*}
  and
  \begin{multline*}
    \int_{\xi_{1}}^{\xi}|\hat{J}_{P}(\eta)|\left||\hat{Q}(\eta)|^{2\sigma}\hat{Q}(\eta) - |\hat{Z}(\eta)|^{2\sigma}\hat{Z}(\eta)\right|\,d\eta\\
    \leq C_{I_{\hat{P}}}M_{\sigma} \|\hat{Q} - \hat{Z}\|_{0}(\|\hat{Q}\|_{0}^{2\sigma} + \|\hat{Z}\|_{0}^{2\sigma})
    e^{\real(c)\xi^{2}}\xi^{-\frac{2}{\sigma} + d - 4}.
  \end{multline*}
  Combined with Lemma~\ref{lemma:P-hat-E-hat-bounds}, this gives us
  \begin{equation*}
    \xi^{\frac{1}{\sigma}}|\hat{T}(\hat{Q})(\xi) - \hat{T}(\hat{Z})(\xi)|
    \leq (C_{\hat{P}}C_{I_{\hat{E}}} + C_{\hat{E}}C_{I_{\hat{P}}}\xi_{1}^{-2})M_{\sigma} \|\hat{Q} - \hat{Z}\|_{0}(\|\hat{Q}\|_{0}^{2\sigma} + \|\hat{Z}\|_{0}^{2\sigma})\xi_{1}^{-2}.
  \end{equation*}

  That \(\hat{T}\) maps \(\Space_{0}\) to \(\Space_{0}\) follows from
  the continuity of \(\hat{T}(\hat{Q})(\xi)\) in \(\xi\) and the
  bound~\eqref{eq:T-1-hat}. That the mapping is continuous is
  immediate from the Lipschitz bound~\eqref{eq:T-2-hat}.
\end{proof}

The process then follows the same path as for the backward equation.
The estimates~\eqref{eq:T-1-hat} and~\eqref{eq:T-2-hat} show that
\(\hat{T}\) is a contraction of the ball
\(B_{\rho} = \{\hat{Q} \in \Space_{0} \colon \|\hat{Q}\|_{0} \leq \rho\}\)
into itself if \(\rho\) is such that
\begin{align}
  \label{eq:T-hat-ineq-1}
  C_{\hat{P}}|\hat{\gamma}_{1}|
  + C_{\hat{E}}|\hat{\gamma}_{2}|e^{-\real(c)\xi_{1}^{2}}\xi_{1}^{\frac{2}{\sigma} - d}
  + C_{\hat{T}}\xi_{1}^{-2}\rho^{2\sigma + 1} &\leq \rho,\\
  \label{eq:T-hat-ineq-2}
  2M_{\sigma}C_{\hat{T}}\xi_{1}^{-2}\rho^{2\sigma} &< 1.
\end{align}
This gives us the following result.

\begin{proposition}\label{prop:Q-hat-fixed-point}
  Under the assumptions of Lemma~\ref{lemma:Q-hat-fixed-point-bounds},
  if \(\rho\) is such that the two
  inequalities~\eqref{eq:T-hat-ineq-1} and~\eqref{eq:T-hat-ineq-2} are
  satisfied, then the map \(\hat{T}\) has a unique fixed point
  \(\hat{Q}\) in the ball \(B_{\rho}\).
\end{proposition}

\subsection{Enclosures for fixed point}
\label{sec:fixed-point-enclosures-hat}

Given the existence of a fixed point, we are interested in computing
refined enclosures of \(\hat{Q}\), \(\hat{Q}'\),
\(\hat{Q}_{\real\hat{\gamma}_{2}}\),
\(\hat{Q}_{\imag\hat{\gamma}_{2}}\),
\(\hat{Q}_{\real\hat{\gamma}_{2}}'\) and
\(\hat{Q}_{\imag\hat{\gamma}_{2}}'\) at \(\xi = \xi_{1}\). Compared
to~\cite{Dahne2024}, a slightly different approach is used and the
presentation therefore follows a different structure. First, we
discuss how to compute enclosures of \(\hat{Q}\), after which we
describe the adjustments needed to handle
\(\hat{Q}_{\real\hat{\gamma}_{2}}\) and
\(\hat{Q}_{\imag\hat{\gamma}_{2}}\); finally, we go through how to
enclose \(\hat{Q}'\), \(\hat{Q}_{\real\hat{\gamma}_{2}}'\) and
\(\hat{Q}_{\imag\hat{\gamma}_{2}}'\).

At \(\xi = \xi_{1}\), we have \(I_{\hat{P}}(\xi_{1}) = 0\), hence
\begin{equation*}
  \hat{Q}(\xi_{1})
  = \hat{\gamma}_{1}\hat{P}(\xi_{1}) + \hat{\gamma}_{2}\hat{E}(\xi_{1})
  + \hat{P}(\xi_{1})I_{\hat{E}}(\xi_{1}).
\end{equation*}
We can readily compute accurate enclosures of both
\(\hat{P}(\xi_{1})\) and \(\hat{E}(\xi_{1})\). What remains is thus
enclosing \(I_{\hat{E}}(\xi_{1})\). While an enclosure can be computed
using Lemma~\ref{lemma:I_E-hat-I_P-hat-bounds}, the resulting bounds
are not quite good enough for our purposes. Instead, we extract the
leading term of the integrand and integrate that explicitly. As a
first step, we isolate the leading term of \(\hat{Q}\), which is also
needed for matching the backward and forward solutions.

\begin{lemma}\label{lemma:Q-hat-leading-term}
  Suppose the assumptions of Lemma~\ref{lemma:I_E-hat-I_P-hat-bounds}
  hold. Then, for \(\xi \geq \xi_{1}\), we have
  \begin{equation*}
    \hat{Q}(\xi) = \hat{\gamma}_{1}\hat{P}(\xi) + R_{\hat{Q}}(\xi)
  \end{equation*}
  where
  \begin{equation*}
    |R_{\hat{Q}}(\xi)| \leq C_{R,\hat{Q}}\xi^{-\frac{1}{\sigma} - 2}
  \end{equation*}
  with
  \begin{equation*}
    C_{R,\hat{Q}} = |\hat{\gamma}_{2}|C_{\hat{E}}e^{-\real(c)\xi_{1}^{2}}\xi_{1}^{\frac{2}{\sigma} - d + 2}
    + C_{\hat{T}}\|\hat{Q}\|_{0}^{2\sigma + 1}.
  \end{equation*}
  In particular,
  \begin{equation*}
    \hat{Q}(\xi) = p_{\hat{Q}}\xi^{-2a} + \mathcal{O}(\xi^{-\frac{1}{\sigma} - 2})
  \end{equation*}
  with \(p_{\hat{Q}} = \hat{\gamma}_{1}(-c)^{-a}\).
\end{lemma}

\begin{proof}
  We have
  \begin{equation*}
    \hat{Q}(\xi) = \hat{\gamma}_{1}\hat{P}(\xi) + \hat{\gamma}_{2}\hat{E}(\xi)
    + \hat{P}(\xi)I_{\hat{E}}(\xi) + \hat{E}(\xi)I_{\hat{P}}(\xi).
  \end{equation*}
  This gives us
  \(\hat{Q}(\xi) = \hat{\gamma}_{1}\hat{P}(\xi) + R_{\hat{Q}}(\xi)\),
  where
  \begin{equation*}
    R_{\hat{Q}}(\xi) = \hat{\gamma}_{2}\hat{E}(\xi) + \hat{P}(\xi)I_{\hat{E}}(\xi) + \hat{E}(\xi)I_{\hat{P}}(\xi).
  \end{equation*}
  To bound \(R_{\hat{Q}}(\xi)\), we make use of
  Lemmas~\ref{lemma:P-hat-E-hat-bounds}
  and~\ref{lemma:I_E-hat-I_P-hat-bounds}, yielding
  \begin{equation*}
    \begin{split}
      |R_{\hat{Q}}(\xi)|
      &\leq |\hat{\gamma}_{2}|C_{\hat{E}}e^{-\real(c)\xi^{2}}\xi^{\frac{1}{\sigma} - d}\\
      &\quad+ C_{\hat{P}}\xi^{-\frac{1}{\sigma}}C_{I_{\hat{E}}}\xi^{-2}\|\hat{Q}\|_{0}^{2\sigma + 1}\\
      &\quad+ C_{\hat{E}}e^{-\real(c)\xi^{2}}\xi^{\frac{1}{\sigma} - d}C_{I_{\hat{P}}}e^{\real(c)\xi^{2}}\xi^{-\frac{2}{\sigma} + d - 4}\|\hat{Q}\|_{0}^{2\sigma + 1}\\
      &= |\hat{\gamma}_{2}|C_{\hat{E}}e^{-\real(c)\xi^{2}}\xi^{\frac{1}{\sigma} - d}
      + C_{\hat{P}}C_{I_{\hat{E}}}\xi^{-\frac{1}{\sigma} - 2}\|\hat{Q}\|_{0}^{2\sigma + 1}
      + C_{\hat{E}}C_{I_{\hat{P}}}\xi^{-\frac{1}{\sigma} - 4}\|\hat{Q}\|_{0}^{2\sigma + 1}.
    \end{split}
  \end{equation*}
  Factoring out \(\xi^{-\frac{1}{\sigma} - 2}\) from the bound, we get
  \begin{equation*}
    |R_{\hat{Q}}(\xi)|
    \leq \left(|\hat{\gamma}_{2}|C_{\hat{E}}e^{-\real(c)\xi^{2}}\xi^{\frac{2}{\sigma} - d + 2}
      + (C_{\hat{P}}C_{I_{\hat{E}}} + C_{\hat{E}}C_{I_{\hat{P}}}\xi^{-2})\|\hat{Q}\|_{0}^{2\sigma + 1}
    \right)\xi^{-\frac{1}{\sigma} - 2}.
  \end{equation*}
  To get a uniform bound for
  \(e^{-\real(c)\xi^{2}}\xi^{\frac{2}{\sigma} - d + 2}\), we
  differentiate, giving us
  \begin{equation*}
    \left(-2\real(c) + \left(\frac{2}{\sigma} - d + 2\right)\xi^{-2}\right)
    e^{-\real(c)\xi^{2}}\xi^{\frac{2}{\sigma} - d + 3}.
  \end{equation*}
  By the assumptions of Lemma~\ref{lemma:I_E-hat-I_P-hat-bounds}, we
  have
  \(2\real(c) > \left(\frac{2}{\sigma} - d + 4\right)\xi_{1}^{-2}\),
  and hence, for \(\xi \geq \xi_{1}\), the factor
  \(-2\real(c) + \left(\frac{2}{\sigma} - d + 2\right)\xi^{-2}\) is
  negative and the function is decreasing. This gives us
  \begin{equation*}
    |R_{\hat{Q}}(\xi)|
    \leq \left(|\hat{\gamma}_{2}|C_{\hat{E}}e^{-\real(c)\xi_{1}^{2}}\xi_{1}^{\frac{2}{\sigma} - d + 2}
      + (C_{\hat{P}}C_{I_{\hat{E}}} + C_{\hat{E}}C_{I_{\hat{P}}}\xi_{1}^{-2})\|\hat{Q}\|_{0}^{2\sigma + 1}
    \right)\xi^{-\frac{1}{\sigma} - 2}.
  \end{equation*}

  That
  \(\hat{Q}(\xi) = p_{\hat{Q}}\xi^{-2a} +
  \mathcal{O}(\xi^{-\frac{1}{\sigma} - 2})\) follows immediately from
  the fact that
  \begin{equation*}
    \hat{P}(\xi) = (-c)^{-a}\xi^{-2a} + \mathcal{O}(\xi^{-\frac{1}{\sigma} - 2}).
  \end{equation*}
\end{proof}

The above lemma allows us to extract the leading term of
\(I_{\hat{E}}\). To simplify the handling of the nonlinearity, we
specialize to the case \(\sigma = 1\). However, to make comparison
with other lemmas simpler, we keep \(\sigma\) in parts of the
formulas.

\begin{lemma}\label{lemma:I_E-hat-enclosure}
  Let \(\sigma = 1\). Under the assumptions of
  Lemma~\ref{lemma:Q-hat-leading-term}, we have, for
  \(\xi \geq \xi_{1}\),
  \begin{equation*}
    I_{\hat{E}}(\xi)
    = -|\hat{\gamma}_{1}|^{2}\hat{\gamma}_{1}\int_{\xi}^{\infty} \hat{J}_{E}(\eta)|\hat{P}(\eta)|^{2}\hat{P}(\eta)\,d\eta
    + R_{I_{\hat{E}}}(\xi),
  \end{equation*}
  with
  \begin{equation*}
    |R_{I_{\hat{E}}}(\xi)|
    \leq C_{\hat{J}_{E}}
    \left(
    \frac{3|\hat{\gamma}_{1}|^{2}C_{\hat{P}}^{2}C_{R,\hat{Q}}}{|-\frac{2}{\sigma} - 2|}
    + \frac{3|\hat{\gamma}_{1}|C_{\hat{P}}C_{R,\hat{Q}}^{2}}{|-\frac{2}{\sigma} - 4|}\xi^{-2}
    + \frac{C_{R,\hat{Q}}^{3}}{|-\frac{2}{\sigma} - 6|}\xi^{-4}
    \right)\xi^{-\frac{2}{\sigma} - 2}.
  \end{equation*}
\end{lemma}

\begin{proof}
  From Lemma~\ref{lemma:Q-hat-leading-term} we have
  \begin{equation*}
    \hat{Q}(\xi) = \hat{\gamma}_{1}\hat{P}(\xi) + R_{\hat{Q}}(\xi).
  \end{equation*}
  With \(\sigma = 1\) this gives us
  \begin{multline*}
    |\hat{Q}(\xi)|^{2\sigma}\hat{Q}(\xi)
    = |\hat{\gamma}_{1}|^{2}\hat{\gamma}_{1}|\hat{P}(\xi)|^{2}\hat{P}(\xi)
    + 2|\hat{\gamma}_{1}|^{2}|\hat{P}(\xi)|^{2} R_{\hat{Q}}(\xi)
    + \hat{\gamma}_{1}^{2}\hat{P}(\xi)^{2}\conj{R_{\hat{Q}}(\xi)}\\
    + 2\hat{\gamma}_{1}\hat{P}(\xi) |R_{\hat{Q}}(\xi)|^{2}
    + \conj{\hat{\gamma}_{1}}\conj{\hat{P}(\xi)} R_{\hat{Q}}(\xi)^{2}
    + |R_{\hat{Q}}(\xi)|^{2} R_{\hat{Q}}(\xi).
  \end{multline*}
  Let \(R_{\hat{Q},N}\) denote the sum of the remainder terms, so that
  \begin{equation*}
    |\hat{Q}(\xi)|^{2\sigma}\hat{Q}(\xi) = |\hat{\gamma}_{1}|^{2}\hat{\gamma}_{1}|\hat{P}(\xi)|^{2}\hat{P}(\xi)
    + R_{\hat{Q},N}(\xi).
  \end{equation*}
  Inserting this into \(I_{\hat{E}}(\xi)\) gives us
  \begin{equation*}
    I_{\hat{E}}(\xi)
    = -|\hat{\gamma}_{1}|^{2}\hat{\gamma}_{1}\int_{\xi}^{\infty} \hat{J}_{E}(\eta)|\hat{P}(\eta)|^{2}\hat{P}(\eta)\,d\eta
    - \int_{\xi}^{\infty} \hat{J}_{E}(\eta)R_{\hat{Q},N}(\eta)\,d\eta.
  \end{equation*}

  To bound the remainder integral, we first bound \(R_{\hat{Q},N}\).
  Using Lemma~\ref{lemma:P-hat-E-hat-bounds} and the inequality
  \(|R_{\hat{Q}}(\xi)| \leq C_{R,\hat{Q}}\xi^{-\frac{1}{\sigma} -
    2}\), we get
  \begin{equation*}
    |R_{\hat{Q},N}(\xi)|
    \leq 3|\hat{\gamma}_{1}|^{2}C_{\hat{P}}^{2}C_{R,\hat{Q}}\xi^{-\frac{3}{\sigma} - 2}
    + 3|\hat{\gamma}_{1}|C_{\hat{P}}C_{R,\hat{Q}}^{2}\xi^{-\frac{3}{\sigma} - 4}
    + C_{R,\hat{Q}}^{3}\xi^{-\frac{3}{\sigma} - 6}.
  \end{equation*}
  Inserting this back into the integral and combining it with the
  bound for \(\hat{J}_{E}\) from
  Lemma~\ref{lemma:P-hat-E-hat-bounds}, we get
  \begin{equation*}
    \begin{split}
      \left|\int_{\xi}^{\infty} \hat{J}_{E}(\eta)R_{\hat{Q},N}(\eta)\,d\eta\right|
      &\leq C_{\hat{J}_{E}}\int_{\xi}^{\infty}
        3|\hat{\gamma}_{1}|^{2}C_{\hat{P}}^{2}C_{R,\hat{Q}}\eta^{-\frac{2}{\sigma} - 3}
      + 3|\hat{\gamma}_{1}|C_{\hat{P}}C_{R,\hat{Q}}^{2}\eta^{-\frac{2}{\sigma} - 5}
      + C_{R,\hat{Q}}^{3}\eta^{-\frac{2}{\sigma} - 7}\,d\eta\\
      &= C_{\hat{J}_{E}}
        \left(
        \frac{3|\hat{\gamma}_{1}|^{2}C_{\hat{P}}^{2}C_{R,\hat{Q}}}{|-\frac{2}{\sigma} - 2|}
      + \frac{3|\hat{\gamma}_{1}|C_{\hat{P}}C_{R,\hat{Q}}^{2}}{|-\frac{2}{\sigma} - 4|}\xi^{-2}
      + \frac{C_{R,\hat{Q}}^{3}}{|-\frac{2}{\sigma} - 6|}\xi^{-4}
        \right)\xi^{-\frac{2}{\sigma} - 2}.
    \end{split}
  \end{equation*}
\end{proof}

To help us enclose the integral
\begin{equation*}
  \int_{\xi}^{\infty} \hat{J}_{E}(\eta)|\hat{P}(\eta)|^{2}\hat{P}(\eta)\,d\eta,
\end{equation*}
we have the following lemma.

\begin{lemma}\label{lemma:integral-J_E-hat-P-hat}
  We have
  \begin{equation*}
    \int_{\xi}^{\infty} \hat{J}_{E}(\eta)|\hat{P}(\eta)|^{2}\hat{P}(\eta)\,d\eta
    = \hat{B}_{W}c^{a - b}|(-c)^{-a}|^{2}(-c)^{-a}I_{U}(\xi),
  \end{equation*}
  with
  \begin{equation*}
    \begin{split}
      I_{U}(\xi)
      = \int_{\xi}^{\infty}
      &\left(
        \sum_{k = 0}^{n - 1} \frac{(b - a)_{k}(-a + 1)_{k}}{k!(-c)^{k}}\eta^{-2k}
        + R_{U}(b - a, b, n, c\eta^{2})c^{-n}\eta^{-2n}
        \right)\\
      &\quad\left(
        \sum_{k = 0}^{n - 1} \frac{(a)_{k}(a - b + 1)_{k}}{k!c^{k}}\eta^{-2k}
        + R_{U}(a, b, n, -c\eta^{2})(-c)^{-n}\eta^{-2n}
        \right)^{2}\\
      &\quad\left(
        \sum_{k = 0}^{n - 1} \conj{\left(\frac{(a)_{k}(a - b + 1)_{k}}{k!c^{k}}\right)}\eta^{-2k}
        + \conj{R_{U}(a, b, n, -c\eta^{2})(-c)^{-n}}\eta^{-2n}
        \right)\eta^{-\frac{2}{\sigma} - 1}\,d\eta.
    \end{split}
  \end{equation*}
  Here \(R_{U}\) is the remainder term in the asymptotic expansion of
  \(U\) from~\cite[Lemma~7.1]{Dahne2024} and \(n\) is any
  non-negative integer.
\end{lemma}

\begin{proof}
  Since
  \begin{equation*}
    \hat{J}_{E}(\xi) = \hat{B}_{W}\hat{E}(\xi)e^{c\xi^{2}}\xi^{d - 1}
    = \hat{B}_{W}U(b - a, b, c\xi^{2})\xi^{d - 1}
  \end{equation*}
  we can write the integral as
  \begin{equation*}
    \int_{\xi}^{\infty} \hat{J}_{E}(\eta)|\hat{P}(\eta)|^{2}\hat{P}(\eta)\,d\eta
    = \hat{B}_{W}\int_{\xi}^{\infty} U(b - a, b, c\eta^{2})\hat{P}(\eta)^{2}\conj{\hat{P}(\eta)}\eta^{d - 1}\,d\eta.
  \end{equation*}
  Using the expansion for \(U\) from~\cite[Lemma~7.1]{Dahne2024}, we
  have
  \begin{equation*}
    \begin{split}
      \hat{P}(\xi)
      = U(a, b, -c\xi^{2})
      &= \left(
        \sum_{k = 0}^{n - 1} \frac{(a)_{k}(a - b + 1)_{k}}{k!(c\xi^{2})^{k}}
        + R_{U}(a, b, n, -c\xi^{2})\left(-c\xi^{2}\right)^{-n}
        \right)\left(-c\xi^{2}\right)^{-a}\\
      &= \left(
        \sum_{k = 0}^{n - 1} \frac{(a)_{k}(a - b + 1)_{k}}{k!c^{k}}\xi^{-2k}
        + R_{U}(a, b, n, -c\xi^{2})(-c)^{-n}\xi^{-2n}
        \right)(-c)^{-a}\xi^{-2a}
    \end{split}
  \end{equation*}
  and
  \begin{equation*}
    \begin{split}
      U(b - a, b, c\xi^{2})
      &= \left(
        \sum_{k = 0}^{n - 1} \frac{(b - a)_{k}(-a + 1)_{k}}{k!(-c\xi^{2})^{k}}
        + R_{U}(b - a, b, n, c\xi^{2})\left(c\xi^{2}\right)^{-n}
        \right)\left(c\xi^{2}\right)^{a - b}\\
      &= \left(
        \sum_{k = 0}^{n - 1} \frac{(b - a)_{k}(-a + 1)_{k}}{k!(-c)^{k}}\xi^{-2k}
        + R_{U}(b - a, b, n, c\xi^{2})c^{-n}\xi^{-2n}
        \right)c^{a - b}\xi^{2(a - b)}.
    \end{split}
  \end{equation*}
  We can then write the integral as
  \begin{equation*}
    \begin{split}
      \hat{B}_{W}c^{a - b}|(-c)^{-a}|^{2}(-c)^{-a}\int_{\xi}^{\infty}
      &\left(
        \sum_{k = 0}^{n - 1} \frac{(b - a)_{k}(-a + 1)_{k}}{k!(-c)^{k}}\eta^{-2k}
        + R_{U}(b - a, b, n, c\eta^{2})c^{-n}\eta^{-2n}
        \right)\\
      &\ \left(
        \sum_{k = 0}^{n - 1} \frac{(a)_{k}(a - b + 1)_{k}}{k!c^{k}}\eta^{-2k}
        + R_{U}(a, b, n, -c\eta^{2})(-c)^{-n}\eta^{-2n}
        \right)^{2}\\
      &\ \left(
        \sum_{k = 0}^{n - 1} \conj{\left(\frac{(a)_{k}(a - b + 1)_{k}}{k!c^{k}}\right)}\eta^{-2k}
        + \conj{R_{U}(a, b, n, -c\eta^{2})(-c)^{-n}}\eta^{-2n}
        \right)\eta^{-\frac{2}{\sigma} - 1}\,d\eta,
    \end{split}
  \end{equation*}
  where we have used that \(2a + 2\conj{a} = \frac{2}{\sigma}\)
  and \(2b = d\).
\end{proof}

These lemmas allow us to compute a sufficiently tight enclosure of
\(I_{\hat{E}}\) for our purposes. The remainder term
\(R_{I_{\hat{E}}}\) is bounded as in
Lemma~\ref{lemma:I_E-hat-enclosure}. The next step is to enclose
\begin{equation*}
  \int_{\xi}^{\infty} \hat{J}_{E}(\eta)|\hat{P}(\eta)|^{2}\hat{P}(\eta)\,d\eta
\end{equation*}
using the above lemma. In the implementation, this is done by
expanding the sum using nested for-loops and integrating it termwise.
For terms without any factors \(R_{U}\), the integral is computed
explicitly, giving very tight enclosures. When there are \(R_{U}\)
factors, we use the bound
\(|R_{U}(a, b, n, z)| \leq C_{R_{U}}(a, b, n, z_{1})\)
from~\cite[Lemma~7.1]{Dahne2024}.

With the above, we can compute enclosures of \(\hat{Q}\). The next
step is to enclose the derivatives with respect to
\(\real\hat{\gamma}_{2}\) and \(\imag\hat{\gamma}_{2}\). In this case,
we have
\begin{align*}
  \hat{Q}_{\real\hat{\gamma}_{2}}(\xi)
  &= \hat{E}(\xi)
    + \hat{P}(\xi)I_{\hat{E},\real\hat{\gamma}_{2}}(\xi)
    + \hat{E}(\xi)I_{\hat{P},\real\hat{\gamma}_{2}}(\xi),\\
  \hat{Q}_{\imag\hat{\gamma}_{2}}(\xi)
  &= i\hat{E}(\xi)
    + \hat{P}(\xi)I_{\hat{E},\imag\hat{\gamma}_{2}}(\xi)
    + \hat{E}(\xi)I_{\hat{P},\imag\hat{\gamma}_{2}}(\xi),
\end{align*}
where \(I_{\hat{P},\real\hat{\gamma}_{2}}(\xi)\) and
\(I_{\hat{P},\imag\hat{\gamma}_{2}}(\xi)\) are zero at
\(\xi = \xi_{1}\). As for \(\hat{Q}\), the only problematic part is
bounding \(I_{\hat{E},\real\hat{\gamma}_{2}}\) and
\(I_{\hat{E},\imag\hat{\gamma}_{2}}\), where the former is given by
\begin{equation*}
  I_{\hat{E},\real\hat{\gamma}_{2}}(\xi) = -\int_{\xi}^{\infty} \hat{J}_{E}(\eta)\frac{\partial}{\partial \real\hat{\gamma}_{2}}\left(|\hat{Q}(\eta)|^{2\sigma}\hat{Q}(\eta)\right)\,d\eta
\end{equation*}
and similarly for \(\imag\hat{\gamma}_{2}\).

To bound \(I_{\hat{E},\real\hat{\gamma}_{2}}\) and
\(I_{\hat{E},\imag\hat{\gamma}_{2}}\), we need an initial bound for
\(\|\hat{Q}_{\real\hat{\gamma}_{2}}\|_{0}\) and
\(\|\hat{Q}_{\imag\hat{\gamma}_{2}}\|_{0}\). For this, we need bounds
for \(I_{\hat{E},\real\hat{\gamma}_{2}}\),
\(I_{\hat{E},\imag\hat{\gamma}_{2}}\),
\(I_{\hat{P},\real\hat{\gamma}_{2}}\) and
\(I_{\hat{P},\imag\hat{\gamma}_{2}}\), for which we have the following
lemma.

\begin{lemma}\label{lemma:I_E-hat-I_P-hat-dgamma-bounds}
  Under the assumptions of Lemma~\ref{lemma:I_E-hat-I_P-hat-bounds},
  we have, for \(\xi \geq \xi_{1}\) and
  \(t \in \{\real\hat{\gamma}_{2}, \imag\hat{\gamma}_{2}\}\),
  \begin{align*}
    |I_{\hat{E},t}(\xi)| &\leq (2\sigma + 1)C_{I_{\hat{E}}}\xi^{-2}\|\hat{Q}\|_{0}^{2\sigma}\|\hat{Q}_{t}\|_{0},\\
    |I_{\hat{P},t}(\xi)| &\leq (2\sigma + 1)C_{I_{\hat{P}}}e^{\real(c)\xi^{2}}\xi^{-\frac{2}{\sigma} + d - 4}\|\hat{Q}\|_{0}^{2\sigma}\|\hat{Q}_{t}\|_{0},
  \end{align*}
  with \(C_{I_{\hat{E}}}\) and \(C_{I_{\hat{P}}}\) as in
  Lemma~\ref{lemma:I_E-hat-I_P-hat-bounds}.
\end{lemma}

\begin{proof}
  To begin with, we note that, for \(\eta \geq \xi_{1}\), we have
  \begin{equation*}
    \left|\frac{\partial}{\partial t}\left(|\hat{Q}(\eta)|^{2\sigma}\hat{Q}(\eta)\right)\right|
    = \left|
      (\sigma + 1)|\hat{Q}|^{2\sigma}\frac{\partial \hat{Q}}{\partial t}
      + \sigma|\hat{Q}|^{2\sigma - 2}\hat{Q}^{2}\conj{\left(\frac{\partial \hat{Q}}{\partial t}\right)}
    \right|
    \leq (2\sigma + 1)\|\hat{Q}\|_{0}^{2\sigma}\|\hat{Q}_{t}\|_{0}\eta^{-\frac{1}{\sigma} - 2}.
  \end{equation*}
  The result then follows from the same calculations as in
  Lemma~\ref{lemma:I_E-hat-I_P-hat-bounds}.
\end{proof}

This allows us to bound \(\|\hat{Q}_{\real\hat{\gamma}_{2}}\|_{0}\)
and \(\|\hat{Q}_{\imag\hat{\gamma}_{2}}\|_{0}\) in terms of
\(\|\hat{Q}\|_{0}\).

\begin{lemma}\label{lemma:Q-hat-dgamma-bound}
  Under the assumptions of Lemma~\ref{lemma:Q-hat-fixed-point-bounds},
  we have, for
  \(t \in \{\real\hat{\gamma}_{2}, \imag\hat{\gamma}_{2}\}\),
  \begin{equation*}
    \|\hat{Q}_{t}\|_{0} \leq \frac{
      C_{\hat{E}}e^{-\real(c)\xi_{1}^{2}}\xi_{1}^{\frac{2}{\sigma} - d}
    }{1 - (2\sigma + 1)C_{\hat{T}}\xi_{1}^{-2}\|\hat{Q}\|_{0}^{2\sigma}},
  \end{equation*}
  as long as the denominator is positive.
\end{lemma}

\begin{proof}
  We have
  \begin{equation*}
    |\hat{Q}_{t}(\xi)| \leq |\hat{E}(\xi)|
    + |\hat{P}(\xi)I_{\hat{E},t}(\xi)|
    + |\hat{E}(\xi)I_{\hat{P},t}(\xi)|.
  \end{equation*}
  Using Lemma~\ref{lemma:I_E-hat-I_P-hat-dgamma-bounds} and following
  the same approach as in Lemma~\ref{lemma:Q-hat-fixed-point-bounds},
  we get
  \begin{equation*}
    \|\hat{Q}_{t}\|_{0}
    \leq C_{\hat{E}}e^{-\real(c)\xi_{1}^{2}}\xi_{1}^{\frac{2}{\sigma} - d}
    + (2\sigma + 1)C_{\hat{T}}\xi_{1}^{-2}\|\hat{Q}\|_{0}^{2\sigma}\|\hat{Q}_{t}\|_{0}.
  \end{equation*}
  Solving for \(\|\hat{Q}_{t}\|_{0}\) gives us the result.
\end{proof}

Combining the bounds for \(\hat{Q}\) with the two lemmas above allows
us to compute an enclosure of \(I_{\hat{E},\real\hat{\gamma}_{2}}\)
and \(I_{\hat{E},\imag\hat{\gamma}_{2}}\), and hence
\(\hat{Q}_{\real\hat{\gamma}_{2}}\) and
\(\hat{Q}_{\imag\hat{\gamma}_{2}}\). Compared to \(\hat{Q}\), we do not
need as accurate enclosures, so the bounds for
\(I_{\hat{E},\real\hat{\gamma}_{2}}\) and
\(I_{\hat{E},\imag\hat{\gamma}_{2}}\) from
Lemma~\ref{lemma:I_E-hat-I_P-hat-dgamma-bounds} suffice in this case.

Finally, we have \(\hat{Q}'\), \(\hat{Q}_{\real\hat{\gamma}_{2}}'\)
and \(\hat{Q}_{\imag\hat{\gamma}_{2}}'\). As a consequence of
variation of parameters, we have
\(\hat{P}(\xi)I_{\hat{E}}'(\xi) + \hat{E}(\xi)I_{\hat{P}}'(\xi) = 0\),
giving us
\begin{equation*}
  \hat{Q}'(\xi)
  = \hat{\gamma}_{1}\hat{P}'(\xi) + \hat{\gamma}_{2}\hat{E}'(\xi)
  + \hat{P}'(\xi)I_{\hat{E}}(\xi) + \hat{E}'(\xi)I_{\hat{P}}(\xi)
\end{equation*}
and
\begin{align*}
  \hat{Q}_{\real\hat{\gamma}_{2}}'(\xi)
  &= \hat{E}'(\xi)
  + \hat{P}'(\xi)I_{\hat{E},\real\hat{\gamma}_{2}}(\xi) + \hat{E}'(\xi)I_{\hat{P},\real\hat{\gamma}_{2}}(\xi),\\
  \hat{Q}_{\imag\hat{\gamma}_{2}}'(\xi)
  &= i\hat{E}'(\xi)
  + \hat{P}'(\xi)I_{\hat{E},\imag\hat{\gamma}_{2}}(\xi) + \hat{E}'(\xi)I_{\hat{P},\imag\hat{\gamma}_{2}}(\xi).
\end{align*}
The only new functions to enclose are \(\hat{E}'(\xi)\) and
\(\hat{P}'(\xi)\), which are easily handled.

\section{Forward solution at zero}
\label{sec:forward-solution-zero}

In this section, we discuss how to enclose \(\hat{Q}_{0}\). As in
the previous section, we keep the strength of the nonlinearity,
\(\sigma\), and the spatial dimension, \(d\), symbolic throughout this
section. The initial value problem is then given by
\begin{equation*}
  \begin{split}
    (1 - i\epsilon)\left(\hat{Q}_{0}'' + \frac{d - 1}{\xi}\hat{Q}_{0}'\right) - i\kappa\xi \hat{Q}_{0}'
    - i \frac{\kappa}{\sigma}\hat{Q}_{0} + \omega \hat{Q}_{0} + (1 + i\delta)|\hat{Q}_{0}|^{2\sigma}\hat{Q}_{0} &= 0,\\
    \hat{Q}_{0}(0) &= \nu,\\
    \hat{Q}_{0}'(0) &= 0.
  \end{split}
\end{equation*}
The only difference between this equation and the associated equation
for the backward self-similar solution is a change of sign in
\(\kappa\) and \(\omega\). Since the signs of \(\kappa\) and
\(\omega\) do not play any role in the approach, we are able to reuse
most of the existing code for the backward solutions. We therefore
refer to~\cite[Section~8]{Dahne2024} for full details on the approach
and only briefly describe the main ideas here.

As a first step, the equation is split into real and imaginary parts,
giving a system of two second-order equations. To handle the removable
singularity at \(\xi = 0\), the interval \([0, \xi_{1}]\) is split
into \([0, \xi_{0}]\) and \([\xi_{0}, \xi_{1}]\) for some small
\(\xi_{0}\). On the first interval, \([0, \xi_{0}]\), the solution is
enclosed using a Taylor expansion at \(\xi = 0\),
see~\cite[Section~8.1]{Dahne2024} for more details. On the interval
\([\xi_{0}, \xi_{1}]\), the system is integrated using a rigorous
numerical integrator implemented in the CAPD
library~\cite{Kapela2021}. Derivatives with respect to the real and
imaginary parts of \(\nu\) are handled similarly to the derivatives
with respect to \(\mu\) in~\cite{Dahne2024}.

\section{Linearized solution at infinity}
\label{sec:linearized-solution-infinity}

In this section, we study the solutions to Equation~\eqref{eq:Y}, given
by
\begin{equation*}
  AY'' + (B_{1}\xi + B_{2}\xi^{-1})Y' + (C + J_{N}(\xi) - 2\kappa\lambda I)Y = 0,
\end{equation*}
which have exponential decay as \(\xi \to \infty\). As in the previous
sections, we keep the strength of the nonlinearity and the spatial
dimension symbolic throughout this section, giving us
\begin{equation}\label{eq:ABC-sigma-d}
  A = \epsilon I + J,\
  B_{1} = \kappa I,\
  B_{2} = (d - 1)A,\
  C = \frac{\kappa}{\sigma} I + \omega J
\end{equation}
as well as
\begin{equation}\label{eq:J_N-sigma-d}
  J_{N}(\xi) = (-\delta I + J)(\hat{Q}(\xi) \cdot \hat{Q}(\xi))^{\sigma - 1}((\hat{Q}(\xi) \cdot \hat{Q}(\xi))I + 2\sigma\hat{Q}(\xi)\hat{Q}(\xi)^{T}).
\end{equation}
We will see that there exists a manifold of such solutions, which can
be parametrized by
\(\begin{pmatrix} c_{0,1} \\ c_{0,2}\end{pmatrix} \in
\mathbb{C}^{2}\). For fixed \(\xi_{1} > 1\), the goal is to compute
enclosures of \(Y(\xi_{1})\) and \(Y'(\xi_{1})\). Note that, compared
to the backward and forward self-similar solutions, in this case we do
not need to compute derivatives with respect to any of the parameters.

For large values of \(\xi\), the value of \(J_{N}(\xi)\) is small. We
therefore take an approach where we write the equation as
\begin{equation*}
  AY'' + (B_{1}\xi + B_{2}\xi^{-1})Y' + (C - 2\kappa\lambda I)Y = -J_{N}(\xi)Y,
\end{equation*}
and treat the right-hand side as a perturbation.

To simplify the treatment of the left-hand side, we start by
diagonalizing it. We observe that all coefficients on the left-hand
side are given by sums of \(I\) and \(J\). We therefore let \(V\) be
the matrix whose columns are eigenvectors of \(J\), so that
\begin{equation*}
  V =
  \begin{pmatrix}
    i & -i \\ 1 & 1
  \end{pmatrix}
  \quad\text{and}\quad
  V^{-1} = \frac{1}{2}
  \begin{pmatrix}
    -i & 1 \\ i & 1
  \end{pmatrix}.
\end{equation*}
Letting \(Y = VZ\) and left-multiplying the ODE by \(V^{-1}\), we get
\begin{equation}\label{eq:Z}
  V^{-1}AVZ'' + (B_{1}\xi + V^{-1}B_{2}V\xi^{-1})Z' + (V^{-1}CV - 2\kappa\lambda I)Z = -V^{-1}J_{N}(\xi)VZ,
\end{equation}
which we will treat as a perturbation of
\begin{equation}\label{eq:Z-homogeneous}
  V^{-1}AVZ'' + (B_{1}\xi + V^{-1}B_{2}V\xi^{-1})Z' + (V^{-1}CV - 2\kappa\lambda I)Z = 0.
\end{equation}
All the matrices on the left-hand side are now diagonal. The ODE
therefore splits into the two scalar ODEs
\begin{align}
  \label{eq:Z-scalar-1}
  (\epsilon + i)Z_{1}'' + (\kappa\xi + (d - 1)(\epsilon + i)\xi^{-1})Z_{1}'
  + \left(\frac{\kappa}{\sigma} + \omega i - 2\kappa\lambda\right)Z_{1} &= 0,\\
  \label{eq:Z-scalar-2}
  (\epsilon - i)Z_{2}'' + (\kappa\xi + (d - 1)(\epsilon - i)\xi^{-1})Z_{2}'
  + \left(\frac{\kappa}{\sigma} - \omega i - 2\kappa\lambda\right)Z_{2} &= 0.
\end{align}
Note that if \(\lambda\) is real, then these two equations are complex
conjugates and the solutions can be chosen with
\(Z_{1} = \conj{Z_{2}}\), corresponding to a real-valued \(Y\). When
\(\lambda\) is not real, this is no longer the case.

Up to multiplication by \(i\) and replacement of
\(\omega + 2i\kappa\lambda\) by \(\omega\),
Equation~\eqref{eq:Z-scalar-1} is the same as
Equation~\eqref{eq:Q-hat-infty-linear}, the ODE for the forward
self-similar solution. From
Section~\ref{sec:forward-solution-infinity}, it follows that two
linearly independent solutions are given by
\begin{equation*}
  P_{1}(\xi) = U\left(a - \lambda, b, -c\xi^{2}\right) \quad\text{and}\quad
  E_{1}(\xi) = e^{-c\xi^{2}}U\left(b - a + \lambda, b, c\xi^{2}\right),
\end{equation*}
with associated Wronskian
\begin{equation*}
  W_{1}(\xi) = P_{1}(\xi)E_{1}'(\xi) - P_{1}'(\xi)E_{1}(\xi)
  = -2ce^{-\sign(\imag(c)) \pi i (b - a + \lambda)}\xi \left(-c\xi^{2}\right)^{-b}e^{-c\xi^{2}}.
\end{equation*}
For Equation~\eqref{eq:Z-scalar-2}, we similarly get that two linearly
independent solutions are given by
\begin{equation*}
  P_{2}(\xi) = U\left(\conj{a} - \lambda, b, -\conj{c}\xi^{2}\right) \quad\text{and}\quad
  E_{2}(\xi) = e^{-\conj{c}\xi^{2}}U\left(b - \conj{a} + \lambda, b, \conj{c}\xi^{2}\right),
\end{equation*}
with associated Wronskian
\begin{equation*}
  W_{2}(\xi) = P_{2}(\xi)E_{2}'(\xi) - P_{2}'(\xi)E_{2}(\xi)
  = -2\conj{c}e^{\sign(\imag(c)) \pi i (b - \conj{a} + \lambda)}\xi \left(-\conj{c}\xi^{2}\right)^{-b}e^{-\conj{c}\xi^{2}}.
\end{equation*}
For Equation~\eqref{eq:Z-homogeneous}, we hence get the four linearly
independent solutions
\begin{equation*}
  P_{1}(\xi)\begin{pmatrix} 1 \\ 0 \end{pmatrix},\quad
  P_{2}(\xi)\begin{pmatrix} 0 \\ 1 \end{pmatrix},\quad
  E_{1}(\xi)\begin{pmatrix} 1 \\ 0 \end{pmatrix} \quad\text{and}\quad
  E_{2}(\xi)\begin{pmatrix} 0 \\ 1 \end{pmatrix}.
\end{equation*}

Following the same approach as for the backward and forward equations,
we now search for solutions to~\eqref{eq:Z} by making use of the
method of variation of parameters. We hence look for solutions of the
form
\begin{equation*}
  Z(\xi) =
  c_{E,1}(\xi)E_{1}(\xi)\begin{pmatrix} 1 \\ 0 \end{pmatrix}
  + c_{E,2}(\xi)E_{2}(\xi)\begin{pmatrix} 0 \\ 1 \end{pmatrix}
  + c_{P,1}(\xi)P_{1}(\xi)\begin{pmatrix} 1 \\ 0 \end{pmatrix}
  + c_{P,2}(\xi)P_{2}(\xi)\begin{pmatrix} 0 \\ 1 \end{pmatrix}.
\end{equation*}
If we let
\begin{equation*}
  E_{12}(\xi) = \begin{pmatrix} E_{1}(\xi) & 0 \\ 0 & E_{2}(\xi) \end{pmatrix}
  \quad\text{and}\quad
  P_{12}(\xi) = \begin{pmatrix} P_{1}(\xi) & 0 \\ 0 & P_{2}(\xi) \end{pmatrix},
\end{equation*}
then we can write this as
\begin{equation*}
  Z(\xi) =
  E_{12}(\xi)\begin{pmatrix} c_{E,1}(\xi) \\ c_{E,2}(\xi) \end{pmatrix}
  + P_{12}(\xi)\begin{pmatrix} c_{P,1}(\xi) \\ c_{P,2}(\xi) \end{pmatrix}.
\end{equation*}
Let \(F(\xi) = -V^{-1}J_{N}(\xi)VZ(\xi)\) denote the right-hand side
of Equation~\eqref{eq:Z}. Since the equation splits into the two
scalar equations~\eqref{eq:Z-scalar-1} and~\eqref{eq:Z-scalar-2}, with
\(F_{1}\) and \(F_{2}\) as their respective right-hand sides, we can
apply the standard variation-of-parameters formulas for scalar
second-order equations to each of them separately. Dividing by the
leading coefficients \(\epsilon + i\) and \(\epsilon - i\), which
amounts to multiplying \(F\) by the diagonal matrix
\(V^{-1}A^{-1}V\), this gives us, for \(j \in \{1, 2\}\),
\begin{align*}
  c_{E,j}'(\xi) &= \frac{P_{j}(\xi)}{W_{j}(\xi)}\left(V^{-1}A^{-1}VF(\xi)\right)_{j},\\
  c_{P,j}'(\xi) &= -\frac{E_{j}(\xi)}{W_{j}(\xi)}\left(V^{-1}A^{-1}VF(\xi)\right)_{j}.
\end{align*}
Written in matrix form, this becomes
\begin{align*}
  \begin{pmatrix} c_{E,1}(\xi) \\ c_{E,2}(\xi) \end{pmatrix}'
  &= K_{1}(\xi)F(\xi),\\
  \begin{pmatrix} c_{P,1}(\xi) \\ c_{P,2}(\xi) \end{pmatrix}'
  &= -K_{2}(\xi)F(\xi),
\end{align*}
with
\begin{equation}\label{eq:K_1-K_2}
  K_{1}(\xi) =
  \begin{pmatrix}
    \frac{P_{1}(\xi)}{W_{1}(\xi)} & 0\\ 0 & \frac{P_{2}(\xi)}{W_{2}(\xi)}
  \end{pmatrix}V^{-1}A^{-1}V
  \quad\text{and}\quad
  K_{2}(\xi) =
  \begin{pmatrix}
    \frac{E_{1}(\xi)}{W_{1}(\xi)} & 0\\ 0 & \frac{E_{2}(\xi)}{W_{2}(\xi)}
  \end{pmatrix}V^{-1}A^{-1}V.
\end{equation}
We want the leading behavior to be determined by \(c_{E,j}\), whereas
the terms with \(c_{P,j}\) should be lower order. This, in particular,
means that \(c_{E,j}\) is allowed to have a non-zero limit at
infinity, whereas \(c_{P,j}\) necessarily has to go to zero. The
choice for \(c_{P,j}\) is then forced. For \(c_{E,j}\), we anchor the
integral at \(\xi = \xi_{1}\), since that is where we want enclosures.
This gives us
\begin{align*}
  \begin{pmatrix} c_{E,1} \\ c_{E,2} \end{pmatrix}
  &= \begin{pmatrix} c_{0,1} \\ c_{0,2} \end{pmatrix}
  + \int_{\xi_{1}}^{\xi} K_{1}(\eta)F(\eta)\,d\eta,\\
  \begin{pmatrix} c_{P,1} \\ c_{P,2} \end{pmatrix}
  &= \int_{\xi}^{\infty} K_{2}(\eta)F(\eta)\,d\eta.
\end{align*}
If we let \(I_{N}(\xi) = V^{-1}J_{N}(\xi)V\), so that
\(F(\xi) = -I_{N}(\xi)Z(\xi)\), and
\begin{align*}
  I_{K_{1}}(\xi)
  &= -\int_{\xi_{1}}^{\xi} K_{1}(\eta)I_{N}(\eta)Z(\eta)\,d\eta,\\
  I_{K_{2}}(\xi)
  &= -\int_{\xi}^{\infty} K_{2}(\eta)I_{N}(\eta)Z(\eta)\,d\eta,
\end{align*}
then we get that fixed points of the operator
\begin{equation}\label{eq:T_12}
  T_{12}(Z)(\xi) = E_{12}(\xi)\begin{pmatrix} c_{0,1} \\ c_{0,2} \end{pmatrix}
  + E_{12}(\xi)I_{K_{1}}(\xi) + P_{12}(\xi)I_{K_{2}}(\xi)
\end{equation}
give us solutions to Equation~\eqref{eq:Z}.

The process is then analogous to the case of the forward solution in
Section~\ref{sec:forward-solution-infinity}:
\begin{enumerate}
\item In Section~\ref{sec:fixed-point-Z}, we give explicit conditions
  for the existence of a fixed point of the operator \(T_{12}\).
\item In Section~\ref{sec:fixed-point-enclosures-Z}, we discuss how to
  use the fixed-point equation to compute tight enclosures of \(Z\)
  and \(Z'\).
\end{enumerate}

In the subsections below, we use \(|\cdot|_{\infty}\) to denote the
supremum norm of a vector and \(\|\cdot\|_{\infty}\) to denote the
operator norm on matrices induced by this vector norm, given by the
maximum absolute row sum of the matrix. In some cases, we also need to
consider the elementwise bounds for vectors or row-wise bounds for
matrices. We use the notation \(|\cdot|_{\infty,j}\) to denote the
absolute value of the \(j\)th element of a vector and
\(\|\cdot\|_{\infty,j}\) to denote the sum of the absolute values in
row \(j\). With this notation, we have
\(|\cdot|_{\infty} = \max_{j} |\cdot|_{\infty,j}\) and
\(\|\cdot\|_{\infty} = \max_{j} \|\cdot\|_{\infty,j}\).

\subsection{Existence of a fixed point}
\label{sec:fixed-point-Z}

As a first step, we need bounds for the functions \(P_{1}\), \(P_{2}\),
\(E_{1}\) and \(E_{2}\). As for the forward and backward
solutions, it will also be convenient to introduce the notation
\begin{equation*}
  J_{P,1}(\xi) = P_{1}(\xi) W_{1}(\xi)^{-1},\quad
  J_{P,2}(\xi) = P_{2}(\xi) W_{2}(\xi)^{-1},\quad
  J_{E,1}(\xi) = E_{1}(\xi) W_{1}(\xi)^{-1} \quad\text{and}\quad
  J_{E,2}(\xi) = E_{2}(\xi) W_{2}(\xi)^{-1}.
\end{equation*}
If we let
\begin{equation*}
  B_{W,1} = \frac{1}{2}e^{\sign(\imag(c)) \pi i (b - a + \lambda)}(-c)^{b - 1} \quad\text{and}\quad
  B_{W,2} = \frac{1}{2}e^{-\sign(\imag(c)) \pi i (b - \conj{a} + \lambda)}(-\conj{c})^{b - 1},
\end{equation*}
then we can write these as
\begin{align*}
  J_{P,1}(\xi) &= B_{W,1}P_{1}(\xi)e^{c\xi^{2}}\xi^{d - 1},\\
  J_{P,2}(\xi) &= B_{W,2}P_{2}(\xi)e^{\conj{c}\xi^{2}}\xi^{d - 1},\\
  J_{E,1}(\xi) &= B_{W,1}E_{1}(\xi)e^{c\xi^{2}}\xi^{d - 1},\\
  J_{E,2}(\xi) &= B_{W,2}E_{2}(\xi)e^{\conj{c}\xi^{2}}\xi^{d - 1}.
\end{align*}

We then have the following lemma.

\begin{lemma}\label{lemma:P_i_E_i-bounds}
  Let \(\xi_{1} \geq 1\). Assuming that the parameters \(\lambda\),
  \(\kappa\), \(\omega\) and \(\epsilon\) are such that
  \(U(a - \lambda, b, -c\xi_{1}^{2})\),
  \(U(\conj{a} - \lambda, b, -\conj{c}\xi_{1}^{2})\),
  \(U(b - a + \lambda, b, c\xi_{1}^{2})\) and
  \(U(b - \conj{a} + \lambda, b, \conj{c}\xi_{1}^{2})\)
  satisfy the conditions of~\cite[Lemma~7.1]{Dahne2024}, we have, for
  \(\xi \geq \xi_{1}\) and \(j \in \{1, 2\}\), the bounds
  \begin{align*}
    \left|P_{j}(\xi)\right| &\leq C_{P_{j}}\xi^{-\frac{1}{\sigma} + 2\real(\lambda)},\\
    \left|E_{j}(\xi)\right| &\leq C_{E_{j}}e^{-\real(c)\xi^{2}}\xi^{\frac{1}{\sigma} - d - 2\real(\lambda)},\\
    \left|J_{P,j}(\xi)\right| &\leq C_{J_{P,j}}e^{\real(c)\xi^{2}}\xi^{-\frac{1}{\sigma} + d + 2\real(\lambda) - 1},\\
    \left|J_{E,j}(\xi)\right| &\leq C_{J_{E,j}}\xi^{\frac{1}{\sigma} - 2\real(\lambda) - 1}.
  \end{align*}
  Here, the constants for \(E_{j}\) and \(P_{j}\) are given by
  \begin{align*}
    C_{P_{1}} &= C_{U}(a - \lambda, b, n, -c\xi_{1}^{2})\left|(-c)^{-a + \lambda}\right|,\\
    C_{P_{2}} &= C_{U}(\conj{a} - \lambda, b, n, -\conj{c}\xi_{1}^{2})\left|(-\conj{c})^{-\conj{a} + \lambda}\right|,\\
    C_{E_{1}} &= C_{U}(b - a + \lambda, b, n, c\xi_{1}^{2})\left|c^{-b + a - \lambda}\right|,\\
    C_{E_{2}} &= C_{U}(b - \conj{a} + \lambda, b, n, \conj{c}\xi_{1}^{2})\left|\conj{c}^{-b + \conj{a} - \lambda}\right|.
  \end{align*}
  Here \(C_{U}\) is as in~\cite[Lemma~7.1]{Dahne2024} and \(n\) is any
  non-negative integer. The constants related to \(J_{P,j}\) and
  \(J_{E,j}\) are given by
  \begin{align*}
    C_{J_{P,j}} &= |B_{W,j}|C_{P_{j}},\\
    C_{J_{E,j}} &= |B_{W,j}|C_{E_{j}}.
  \end{align*}
\end{lemma}

\begin{proof}
  From the definitions of \(P_{j}\) and \(E_{j}\), we get
  \begin{align*}
    \left|P_{1}(\xi)\right| &= |U(a - \lambda, b, -c\xi^{2})|,\\
    \left|P_{2}(\xi)\right| &= |U(\conj{a} - \lambda, b, -\conj{c}\xi^{2})|,\\
    \left|E_{1}(\xi)\right| &= |e^{-c\xi^{2}}U(b - a + \lambda, b, c\xi^{2})|,\\
    \left|E_{2}(\xi)\right| &= |e^{-\conj{c}\xi^{2}}U(b - \conj{a} + \lambda, b, \conj{c}\xi^{2})|.
  \end{align*}
  The bounds then follow from the bounds for \(U\)
  in~\cite[Lemma~7.1]{Dahne2024}. The bounds for \(J_{E,j}\) and
  \(J_{P,j}\) are immediate consequences of their definitions and the
  bounds for \(E_{j}\) and \(P_{j}\), respectively.
\end{proof}

For the rest of this section, whenever the bounds from
Lemma~\ref{lemma:P_i_E_i-bounds} are used, we will implicitly assume
that the parameters \(\lambda\), \(\kappa\), \(\omega\) and
\(\epsilon\) satisfy its assumptions.

With this notation, the matrices \(K_{1}\) and \(K_{2}\)
from~\eqref{eq:K_1-K_2} take the form
\begin{equation*}
  K_{1}(\xi) = \begin{pmatrix} J_{P,1}(\xi) & 0\\ 0 & J_{P,2}(\xi)\end{pmatrix}V^{-1}A^{-1}V
  \quad\text{and}\quad
  K_{2}(\xi) = \begin{pmatrix} J_{E,1}(\xi) & 0\\ 0 & J_{E,2}(\xi)\end{pmatrix}V^{-1}A^{-1}V.
\end{equation*}
Since
\begin{equation*}
  V^{-1}A^{-1}V = \frac{1}{1 + \epsilon^{2}}
  \begin{pmatrix}
    \epsilon - i & 0\\
    0 & \epsilon + i
  \end{pmatrix},
\end{equation*}
both \(K_{1}\) and \(K_{2}\) are diagonal.

The following lemma gives us control of the \(\|\cdot\|_{\infty,j}\)
norms for \(K_{1}\) and \(K_{2}\).

\begin{lemma}\label{lemma:bound-K_1-K_2}
  We have the following bounds:
  \begin{align*}
    \|K_{1}(\xi)\|_{\infty,j} &\leq C_{K_{1},j}e^{\real(c)\xi^{2}}\xi^{-\frac{1}{\sigma} + d + 2\real(\lambda) - 1},\\
    \|K_{2}(\xi)\|_{\infty,j} &\leq C_{K_{2},j}\xi^{\frac{1}{\sigma} - 2\real(\lambda) - 1},
  \end{align*}
  where
  \begin{align*}
    C_{K_{1},j} &= \frac{1}{\sqrt{1 + \epsilon^{2}}}C_{J_{P,j}},\\
    C_{K_{2},j} &= \frac{1}{\sqrt{1 + \epsilon^{2}}}C_{J_{E,j}}.
  \end{align*}
\end{lemma}

\begin{proof}
  This follows immediately from Lemma~\ref{lemma:P_i_E_i-bounds} and
  the above representations, since the matrices are diagonal and
  \begin{equation*}
    \left|\frac{\epsilon \pm i}{1 + \epsilon^{2}}\right| = \frac{1}{\sqrt{1 + \epsilon^{2}}}.
  \end{equation*}
\end{proof}

We expect the fixed points of \(T_{12}\) to have the same decay as
\(E_{1}\) and \(E_{2}\). We therefore perform the fixed-point argument
in the Banach space \(\Space_{L}\) of continuous functions
\(Z \colon [\xi_{1}, \infty) \to \mathbb{C}^{2}\) for which the norm
\begin{equation*}
  \|Z\|_{L} = \sup_{\xi \geq \xi_{1}} e^{\real(c)\xi^{2}}\xi^{-\frac{1}{\sigma} + d + 2\real(\lambda)}|Z(\xi)|_{\infty}
\end{equation*}
is finite.

As in Section~\ref{sec:fixed-point-hat}, we need asymptotic bounds for
\(I_{K_{1}}\) and \(I_{K_{2}}\). As a first step, the following lemma
gives us bounds for \(I_{N}\). In this case, we specialize to
\(\sigma = 1\) and \(\delta = 0\), since that is the case we primarily
care about.

\begin{lemma}\label{lemma:bound-I_N}
  If \(\sigma = 1\) and \(\delta = 0\), then we have the bound
  \begin{equation*}
    \|I_{N}(\xi)\|_{\infty} \leq C_{I_{N}}\xi^{-2},
  \end{equation*}
  with
  \begin{equation*}
    C_{I_{N}} = 3\|\hat{Q}\|_{0}^{2}.
  \end{equation*}
  Note that, for the norm \(\|\hat{Q}\|_{0}\) to be defined, we here
  treat \(\hat{Q}\) as a function into \(\mathbb{C}\) rather than into
  \(\mathbb{R}^{2}\).
\end{lemma}

\begin{proof}
  For \(\sigma = 1\) and \(\delta = 0\), we have
  \begin{equation*}
    J_{N}(\xi) = J((\hat{Q}(\xi) \cdot \hat{Q}(\xi))I + 2\hat{Q}(\xi)\hat{Q}(\xi)^{T}).
  \end{equation*}
  If we let \(\hat{a}\) and \(\hat{b}\) denote the two components of
  \(\hat{Q}\), this gives us
  \begin{equation*}
    I_{N}(\xi) = V^{-1}J_{N}(\xi)V =
    \begin{pmatrix}
      2i(\hat{a}^{2} + \hat{b}^{2}) & -i(\hat{a} + i\hat{b})^{2}\\
      i(\hat{a} - i\hat{b})^{2} & -2i(\hat{a}^{2} + \hat{b}^{2})
    \end{pmatrix}.
  \end{equation*}
  By instead interpreting \(\hat{Q}\) as
  \(\hat{Q} = \hat{a} + i\hat{b}\), we can write this as
  \begin{equation*}
    I_{N}(\xi) =
    \begin{pmatrix}
      2i|\hat{Q}|^{2} & -i\hat{Q}^{2}\\
      \conj{-i\hat{Q}^{2}} & -2i|\hat{Q}|^{2}
    \end{pmatrix},
  \end{equation*}
  from which it follows that \(\|I_{N}\|_{\infty} = 3|\hat{Q}|^{2}\).
  By definition of the norm \(\|\hat{Q}\|_{0}\), we have (with
  \(\sigma = 1\)) \(|\hat{Q}(\xi)| \leq \|\hat{Q}\|_{0}\xi^{-1}\),
  which gives us the bound.
\end{proof}

Combining the above result with the bounds for \(K_{1}\) and \(K_{2}\)
from Lemma~\ref{lemma:bound-K_1-K_2}, we can obtain bounds for
\(I_{K_{1}}\) and \(I_{K_{2}}\). For this, the following lemma will be
useful.

\begin{lemma}\label{lemma:exponential-integral-bound}
  Let \(\xi > 0\), \(p - 1 < 0\) and \(q > 0\). Then
  \begin{equation*}
    \int_{\xi}^{\infty} e^{-q\eta^{2}}\eta^{p}\,d\eta \leq \frac{1}{2q}e^{-q\xi^{2}}\xi^{p - 1}.
  \end{equation*}
\end{lemma}

\begin{proof}
  Integrating by parts, we have
  \begin{equation*}
    \int_{\xi}^{\infty} e^{-q\eta^{2}}\eta^{p}\,d\eta
    = \frac{1}{2q}e^{-q\xi^{2}}\xi^{p - 1}
    + \frac{1}{2q}\int_{\xi}^{\infty} e^{-q\eta^{2}}\frac{d}{d\eta}\left(\eta^{p - 1}\right)\,d\eta.
  \end{equation*}
  Since \(\frac{d}{d\eta}\left(\eta^{p - 1}\right) < 0\), the last
  integral is negative and we get an upper bound by removing it.
\end{proof}

The bounds for \(I_{K_{1}}\) and \(I_{K_{2}}\) are given in the
following lemma.

\begin{lemma}\label{lemma:I_K_1-I_K_2-bounds}
  Assume that the assumptions of Lemma~\ref{lemma:bound-I_N} hold and
  that \(\xi_{1} \geq 1\), \(\real(c) > 0\) and
  \(\frac{2}{\sigma} - d - 4\real(\lambda) - 4 < 0\).
  Then, for \(\xi \geq \xi_{1}\) and \(j \in \{1, 2\}\), we have the
  following bounds:
  \begin{align*}
    |I_{K_{1}}(\xi)|_{\infty,j} &\leq C_{I_{K_{1}},j}\xi_{1}^{-2}\|Z\|_{L},\\
    |I_{K_{2}}(\xi)|_{\infty,j} &\leq C_{I_{K_{2}},j}e^{-\real(c)\xi^{2}}\xi^{\frac{2}{\sigma} - d - 4\real(\lambda) - 4}\|Z\|_{L},
  \end{align*}
  with
  \begin{equation*}
    C_{I_{K_{1}},j} = \frac{C_{K_{1},j}C_{I_{N}}}{2},\quad
    C_{I_{K_{2}},j} = \frac{C_{K_{2},j}C_{I_{N}}}{2\real(c)}.
  \end{equation*}
\end{lemma}

\begin{proof}
  Let us start by noting that
  \begin{equation*}
    K_{1}(\eta)I_{N}(\eta)Z(\eta) \quad\text{and}\quad
    K_{2}(\eta)I_{N}(\eta)Z(\eta)
  \end{equation*}
  satisfy
  \begin{align*}
    |K_{1}(\eta)I_{N}(\eta)Z(\eta)|_{\infty,j} &\leq \|K_{1}(\eta)\|_{\infty,j}\|I_{N}(\eta)\|_{\infty}|Z(\eta)|_{\infty},\\
    |K_{2}(\eta)I_{N}(\eta)Z(\eta)|_{\infty,j} &\leq \|K_{2}(\eta)\|_{\infty,j}\|I_{N}(\eta)\|_{\infty}|Z(\eta)|_{\infty}.
  \end{align*}
  This gives us that
  \begin{align*}
    |I_{K_{1}}(\xi)|_{\infty,j}
    &\leq \int_{\xi_{1}}^{\xi} \|K_{1}(\eta)\|_{\infty,j}\|I_{N}(\eta)\|_{\infty}|Z(\eta)|_{\infty}\,d\eta,\\
    |I_{K_{2}}(\xi)|_{\infty,j}
    &\leq \int_{\xi}^{\infty} \|K_{2}(\eta)\|_{\infty,j}\|I_{N}(\eta)\|_{\infty}|Z(\eta)|_{\infty}\,d\eta.
  \end{align*}
  We have
  \begin{equation*}
    |Z(\eta)|_{\infty} \leq e^{-\real(c)\eta^{2}}\eta^{\frac{1}{\sigma} - d - 2\real(\lambda)}\|Z\|_{L}.
  \end{equation*}
  Combining this with the bounds from
  Lemmas~\ref{lemma:bound-K_1-K_2} and~\ref{lemma:bound-I_N}, we get
  \begin{equation*}
    \begin{split}
      \|K_{1}(\eta)\|_{\infty,j}\|I_{N}(\eta)\|_{\infty}|Z(\eta)|_{\infty}
      &\leq C_{K_{1},j}e^{\real(c)\eta^{2}}\eta^{-\frac{1}{\sigma} + d + 2\real(\lambda) - 1}
      C_{I_{N}}\eta^{-2}
      e^{-\real(c)\eta^{2}}\eta^{\frac{1}{\sigma} - d - 2\real(\lambda)}\|Z\|_{L}\\
      &\leq C_{K_{1},j}C_{I_{N}}\eta^{-3}\|Z\|_{L}
    \end{split}
  \end{equation*}
  and
  \begin{equation*}
    \begin{split}
      \|K_{2}(\eta)\|_{\infty,j}\|I_{N}(\eta)\|_{\infty}|Z(\eta)|_{\infty}
      &\leq C_{K_{2},j}\eta^{\frac{1}{\sigma} - 2\real(\lambda) - 1}
      C_{I_{N}}\eta^{-2}
      e^{-\real(c)\eta^{2}}\eta^{\frac{1}{\sigma} - d - 2\real(\lambda)}\|Z\|_{L}\\
      &\leq C_{K_{2},j}C_{I_{N}}e^{-\real(c)\eta^{2}}\eta^{\frac{2}{\sigma} - d - 4\real(\lambda) - 3}\|Z\|_{L}.
    \end{split}
  \end{equation*}
  Upon inserting this into the integrals, we obtain
  \begin{align*}
    |I_{K_{1}}(\xi)|_{\infty,j}
    &\leq C_{K_{1},j}C_{I_{N}}\|Z\|_{L}\int_{\xi_{1}}^{\xi} \eta^{-3}\,d\eta,\\
    |I_{K_{2}}(\xi)|_{\infty,j}
    &\leq C_{K_{2},j}C_{I_{N}}\|Z\|_{L}\int_{\xi}^{\infty}e^{-\real(c)\eta^{2}}\eta^{\frac{2}{\sigma} - d - 4\real(\lambda) - 3}\,d\eta.
  \end{align*}
  For the first integral, we note that
  \begin{equation*}
    \int_{\xi_{1}}^{\xi} \eta^{-3}\,d\eta \leq \int_{\xi_{1}}^{\infty} \eta^{-3}\,d\eta
    = \frac{\xi_{1}^{-2}}{2},
  \end{equation*}
  which gives us the bound for \(I_{K_{1}}\). For \(I_{K_{2}}\),
  Lemma~\ref{lemma:exponential-integral-bound} directly yields
  \begin{equation*}
    \int_{\xi}^{\infty}e^{-\real(c)\eta^{2}}\eta^{\frac{2}{\sigma} - d - 4\real(\lambda) - 3}\,d\eta
    \leq \frac{1}{2\real(c)}e^{-\real(c)\xi^{2}}\xi^{\frac{2}{\sigma} - d - 4\real(\lambda) - 4}.
  \end{equation*}
\end{proof}

The following lemma collects the corresponding bounds for
\(T_{12}\).

\begin{lemma}\label{lemma:Z-fixed-point-bounds}
  Under the assumptions of Lemma~\ref{lemma:I_K_1-I_K_2-bounds}, the
  operator \(T_{12}\) defines a continuous mapping
  \(T_{12} \colon \Space_{L} \to \Space_{L}\). Moreover,
  \begin{equation}
    \label{eq:T_12-1}
    \|T_{12}(Z)\|_{L} \leq \max(C_{E_{1}}|c_{0,1}|, C_{E_{2}}|c_{0,2}|)
    + C_{T_{12}}\xi_{1}^{-2}\|Z\|_{L}
  \end{equation}
  and
  \begin{equation}
    \label{eq:T_12-2}
    \|T_{12}(Z) - T_{12}(W)\|_{L} \leq C_{T_{12}}\xi_{1}^{-2}\|Z - W\|_{L},
  \end{equation}
  for all \(Z, W \in \Space_{L}\). Here,
  \begin{equation*}
    C_{T_{12}} = \max(C_{E_{1}}C_{I_{K_{1}},1}, C_{E_{2}}C_{I_{K_{1}},2})
    + \max(C_{P_{1}}C_{I_{K_{2}},1}, C_{P_{2}}C_{I_{K_{2}},2})\xi_{1}^{-2}.
  \end{equation*}
\end{lemma}

\begin{proof}
  For~\eqref{eq:T_12-1}, we need to bound
  \(e^{\real(c)\xi^{2}}\xi^{-\frac{1}{\sigma} + d +
    2\real(\lambda)}|T_{12}(Z)(\xi)|_{\infty}\). Bounding it
  termwise, we get
  \begin{align*}
    e^{\real(c)\xi^{2}}\xi^{-\frac{1}{\sigma} + d + 2\real(\lambda)}|T_{12}(Z)(\xi)|_{\infty}
    &\leq e^{\real(c)\xi^{2}}\xi^{-\frac{1}{\sigma} + d + 2\real(\lambda)}\left|E_{12}(\xi)\begin{pmatrix} c_{0,1} \\ c_{0,2} \end{pmatrix}\right|_{\infty}\\
    &\qquad+ e^{\real(c)\xi^{2}}\xi^{-\frac{1}{\sigma} + d + 2\real(\lambda)}|E_{12}(\xi)I_{K_{1}}(\xi)|_{\infty}\\
    &\qquad+ e^{\real(c)\xi^{2}}\xi^{-\frac{1}{\sigma} + d + 2\real(\lambda)}|P_{12}(\xi)I_{K_{2}}(\xi)|_{\infty}.
  \end{align*}
  For the first term, we get from Lemma~\ref{lemma:P_i_E_i-bounds} that
  \begin{equation*}
    e^{\real(c)\xi^{2}}\xi^{-\frac{1}{\sigma} + d + 2\real(\lambda)}
    \left|E_{12}(\xi)\begin{pmatrix} c_{0,1} \\ c_{0,2} \end{pmatrix}\right|_{\infty}
    \leq \max(C_{E_{1}}|c_{0,1}|, C_{E_{2}}|c_{0,2}|).
  \end{equation*}
  For the second term, we get from Lemmas~\ref{lemma:P_i_E_i-bounds}
  and~\ref{lemma:I_K_1-I_K_2-bounds} that
  \begin{equation*}
    e^{\real(c)\xi^{2}}\xi^{-\frac{1}{\sigma} + d + 2\real(\lambda)}|E_{12}(\xi)I_{K_{1}}(\xi)|_{\infty}
    \leq \max(C_{E_{1}}C_{I_{K_{1}},1},C_{E_{2}}C_{I_{K_{1}},2})\xi_{1}^{-2}\|Z\|_{L}
  \end{equation*}
  and for the third term, we get
  \begin{equation*}
    e^{\real(c)\xi^{2}}\xi^{-\frac{1}{\sigma} + d + 2\real(\lambda)}|P_{12}(\xi)I_{K_{2}}(\xi)|_{\infty}
    \leq \max(C_{P_{1}}C_{I_{K_{2}},1}, C_{P_{2}}C_{I_{K_{2}},2})\xi_{1}^{-4}\|Z\|_{L}.
  \end{equation*}
  Combining these bounds gives us~\eqref{eq:T_12-1}.

  For~\eqref{eq:T_12-2}, we similarly get
  \begin{multline*}
    e^{\real(c)\xi^{2}}\xi^{-\frac{1}{\sigma} + d + 2\real(\lambda)}|T_{12}(Z)(\xi) - T_{12}(W)(\xi)|_{\infty}\\
    \leq e^{\real(c)\xi^{2}}\xi^{-\frac{1}{\sigma} + d + 2\real(\lambda)}
    \left|E_{12}(\xi)\int_{\xi_{1}}^{\xi} K_{1}(\eta)I_{N}(\eta)(Z(\eta) - W(\eta))\,d\eta\right|_{\infty}\\
    + e^{\real(c)\xi^{2}}\xi^{-\frac{1}{\sigma} + d + 2\real(\lambda)}
    \left|P_{12}(\xi)\int_{\xi}^{\infty} K_{2}(\eta)I_{N}(\eta)(Z(\eta) - W(\eta))\,d\eta\right|_{\infty}.
  \end{multline*}
  By using that
  \begin{equation*}
    |Z(\eta) - W(\eta)|_{\infty} \leq e^{-\real(c)\eta^{2}}\eta^{\frac{1}{\sigma} - d - 2\real(\lambda)}\|Z - W\|_{L},
  \end{equation*}
  the two integrals can be bounded using the same approach as in
  Lemma~\ref{lemma:I_K_1-I_K_2-bounds}, which, combined with
  Lemma~\ref{lemma:P_i_E_i-bounds}, gives us the result.

  That \(T_{12}\) maps \(\Space_{L}\) to \(\Space_{L}\) follows from
  the continuity of \(T_{12}(Z)(\xi)\) in \(\xi\) and the
  bound~\eqref{eq:T_12-1}. That the mapping is continuous is immediate
  from the Lipschitz bound~\eqref{eq:T_12-2}.
\end{proof}

The estimates~\eqref{eq:T_12-1} and~\eqref{eq:T_12-2} show that
\(T_{12}\) is a contraction of the ball
\(B_{\rho} = \{u \in \Space_{L} \colon \|u\|_{L} \leq \rho\}\) into
itself if \(\rho\) is such that
\begin{equation*}
  \max(C_{E_{1}}|c_{0,1}|, C_{E_{2}}|c_{0,2}|) + C_{T_{12}}\xi_{1}^{-2}\rho \leq \rho
\end{equation*}
and
\begin{equation}
  \label{eq:T_12-ineq-2}
  C_{T_{12}}\xi_{1}^{-2} < 1.
\end{equation}
Note that the condition \(C_{T_{12}}\xi_{1}^{-2} < 1\) means that the
first inequality is equivalent to
\begin{equation}
  \label{eq:T_12-ineq-1}
  \rho \geq (1 - C_{T_{12}}\xi_{1}^{-2})^{-1}\max(C_{E_{1}}|c_{0,1}|, C_{E_{2}}|c_{0,2}|).
\end{equation}
The following proposition establishes the existence of a fixed point
using the Banach fixed-point theorem. To use the winding argument
when proving the existence of an eigenvalue in
Theorem~\ref{thm:unstable-eigenvalue}, we also need to verify that the
fixed point depends analytically on \(\lambda\).

\begin{proposition}\label{prop:Z-fixed-point}
  Under the assumptions of Lemma~\ref{lemma:Z-fixed-point-bounds}, if
  \(\rho\) is such that the two inequalities~\eqref{eq:T_12-ineq-1}
  and~\eqref{eq:T_12-ineq-2} are satisfied, then the map \(T_{12}\)
  has a unique fixed point \(Z\) in
  \(B_{\rho} = \{u \in \Space_{L} \colon \|u\|_{L} \leq \rho\}\). In
  particular, \(Z\) satisfies \(\|Z\|_{L} \leq \rho\). Moreover, let
  \(R\) be a closed rectangle in the complex plane such that the above
  holds for every \(\lambda \in R\), with one and the same \(\rho\).
  Then, for every \(\xi \geq \xi_{1}\), the values
  \(Z(\xi) = Z(\xi, \lambda)\) and \(Z'(\xi) = Z'(\xi, \lambda)\)
  depend analytically on \(\lambda\) in the interior of \(R\).
\end{proposition}

\begin{proof}
  The existence of a unique fixed point and the bound
  \(\|Z\|_{L} \leq \rho\) are immediate consequences of the two
  inequalities~\eqref{eq:T_12-ineq-1} and~\eqref{eq:T_12-ineq-2} and
  the Banach fixed-point theorem.

  To prove analyticity in \(\lambda\), we begin by noting that
  \(T_{12}(u)(\xi) = T_{12}(u, \lambda)(\xi)\) depends analytically on
  both \(u\) and \(\lambda\). Analyticity in \(u\) is immediate since
  it is affine in \(u\). The dependence on \(\lambda\) comes from
  the functions \(E_{12}\), \(P_{12}\), \(K_{1}\) and \(K_{2}\)
  appearing in \(T_{12}\). The dependence of these functions on
  \(\lambda\) is given through \(B_{W,1}\) and \(B_{W,2}\), both of
  which are analytic in \(\lambda\), and
  \begin{equation*}
    U\left(a - \lambda, b, -c\xi^{2}\right),\quad
    U\left(b - a + \lambda, b, c\xi^{2}\right),\quad
    U\left(\conj{a} - \lambda, b, -\conj{c}\xi^{2}\right) \quad\text{and}\quad
    U\left(b - \conj{a} + \lambda, b, \conj{c}\xi^{2}\right).
  \end{equation*}
  Since \(U\) is an entire function in its first argument, these
  functions all depend analytically on \(\lambda\). Consequently,
  \(T_{12}\) is analytic in \(\lambda\).

  Now consider the Picard iterates
  \(Z_{n + 1}(\cdot, \lambda) = T_{12}(Z_{n}(\cdot, \lambda), \lambda)\),
  starting from \(Z_{0} = 0\). By induction, each value
  \(Z_{n}(\xi, \lambda)\) is analytic in \(\lambda\); for
  \(I_{K_{2}}\) we use here that the bounds in
  Lemma~\ref{lemma:I_K_1-I_K_2-bounds} hold uniformly for
  \(\lambda \in R\), so that the improper integral converges locally
  uniformly in \(\lambda\). Since \(R\) is compact and \(C_{T_{12}}\)
  depends continuously on \(\lambda\), the fact
  that~\eqref{eq:T_12-ineq-2} holds for every \(\lambda \in R\) gives
  \begin{equation*}
    q = \max_{\lambda \in R}C_{T_{12}}\xi_{1}^{-2} < 1.
  \end{equation*}
  Since \(Z_{0} = 0\) and \(T_{12}\) maps \(B_{\rho}\) into itself, we
  have \(\|Z_{1}\|_{L} \leq \rho\), and the Banach fixed-point theorem
  gives
  \begin{equation*}
    \|Z_{n} - Z\|_{L} \leq \frac{q^{n}}{1 - q}\|Z_{1}\|_{L}
    \leq \frac{q^{n}}{1 - q}\rho,
  \end{equation*}
  with \(q\) and \(\rho\) independent of \(\lambda \in R\). Hence
  \begin{equation*}
    |Z_{n}(\xi, \lambda) - Z(\xi, \lambda)|_{\infty}
    \leq e^{-\real(c)\xi^{2}}\xi^{\frac{1}{\sigma} - d - 2\real(\lambda)}\frac{q^{n}}{1 - q}\rho,
  \end{equation*}
  and, for fixed \(\xi\), the prefactor is bounded for \(\lambda\) in
  the compact set \(R\). The convergence
  \(Z_{n}(\xi, \lambda) \to Z(\xi, \lambda)\) is therefore uniform for
  \(\lambda \in R\). As a uniform limit of analytic functions,
  \(Z(\xi) = Z(\xi, \lambda)\) therefore depends analytically on
  \(\lambda\) in the interior of \(R\).

  Finally, differentiating the fixed-point equation and using that
  \(E_{12}I_{K_{1}}' + P_{12}I_{K_{2}}' = 0\), we get
  \begin{equation*}
    Z'(\xi)
    = E_{12}'(\xi)\begin{pmatrix} c_{0,1} \\ c_{0,2} \end{pmatrix}
    + E_{12}'(\xi)I_{K_{1}}(\xi) + P_{12}'(\xi)I_{K_{2}}(\xi).
  \end{equation*}
  Every term on the right-hand side is analytic in \(\lambda\), and
  hence so is \(Z'(\xi, \lambda)\).
\end{proof}

\subsection{Enclosures for fixed point}
\label{sec:fixed-point-enclosures-Z}

To compute refined enclosures of \(Z\) and \(Z'\) at
\(\xi = \xi_{1}\), we follow a similar approach to that used for the
forward solution in Section~\ref{sec:fixed-point-enclosures-hat}.
However, in this case, we do not need enclosures that are as tight,
and the estimates established in the preceding section are therefore
sufficient.

Since \(I_{K_{1}}(\xi_{1}) = 0\), we have
\begin{equation*}
  Z(\xi_{1}) = E_{12}(\xi_{1})\begin{pmatrix} c_{0,1} \\ c_{0,2} \end{pmatrix}
  + P_{12}(\xi_{1})I_{K_{2}}(\xi_{1}).
\end{equation*}
For the derivative, we note that \(E_{12}I_{K_{1}}'\) and
\(P_{12}I_{K_{2}}'\) cancel, leaving us with
\begin{equation*}
  Z'(\xi_{1})
  = E_{12}'(\xi_{1})\begin{pmatrix} c_{0,1} \\ c_{0,2} \end{pmatrix}
  + P_{12}'(\xi_{1})I_{K_{2}}(\xi_{1}).
\end{equation*}
The functions \(E_{12}\) and \(P_{12}\), along with their derivatives,
can easily be enclosed. To enclose \(I_{K_{2}}(\xi_{1})\), we use the
bound from Lemma~\ref{lemma:I_K_1-I_K_2-bounds} together with
\(\|Z\|_{L} \leq \rho\) from Proposition~\ref{prop:Z-fixed-point},
which gives an enclosure centered at zero. Since we do not need tight
enclosures here, this is sufficient.

Enclosures of \(Y_{\infty}\) and \(Y_{\infty}'\) at \(\xi_{1}\) are
finally recovered from \(Y = VZ\), which gives
\(Y_{\infty}(\xi_{1}) = VZ(\xi_{1})\) and
\(Y_{\infty}'(\xi_{1}) = VZ'(\xi_{1})\). In particular, the functions
\(Y_{\infty,1}\) and \(Y_{\infty,2}\) appearing in~\eqref{eq:H}
correspond to taking \(c_{0} = (1, 0)^{T}\) and
\(c_{0} = (0, 1)^{T}\), respectively.

\section{Linearized solution at zero}
\label{sec:linearized-solution-zero}

In this section, we study solutions to the initial value problem
associated with Equation~\eqref{eq:Y}, given by
\begin{align}
  \label{eq:Y-zero}
  AY'' + (B_{1}\xi + B_{2}\xi^{-1})Y' + (C + J_{N}(\xi) - 2\kappa\lambda I)Y &= 0,\\
  Y(0) &= y_{0},\nonumber\\
  Y'(0) &= 0.\nonumber
\end{align}
We are interested in enclosing \(Y(\xi_{1})\) and \(Y'(\xi_{1})\). The
general approach is similar to the backward and forward cases. We
split the interval \([0, \xi_{1}]\) into \([0, \xi_{0}]\) and
\([\xi_{0}, \xi_{1}]\) and use a Taylor expansion at zero for the
first interval, and a rigorous numerical integrator implemented in the
CAPD library~\cite{Kapela2021} for the second interval. As in the
previous section, we keep the strength of the nonlinearity and the
spatial dimension symbolic throughout this section. The matrices
\(A\), \(B_{1}\), \(B_{2}\), \(C\) and \(J_{N}\) are thus as
in~\eqref{eq:ABC-sigma-d} and~\eqref{eq:J_N-sigma-d}.

The CAPD integrator only handles systems of real first-order equations
and therefore requires us to rewrite the equation in that form. To
handle the term \(J_{N}\), which depends on the solution to the
forward equation, we also add the forward equation to this system. We
split \(Y\) into real and imaginary parts as \(Y = Y_{r} + iY_{i}\)
and \(\hat{Q}\) as \(\hat{Q} = \hat{a} + i\hat{b}\), giving us
\begin{align*}
  AY_{r}'' + (B_{1}\xi + B_{2}\xi^{-1})Y_{r}' + (C + J_{N}(\xi) - 2\kappa\real(\lambda) I)Y_{r} + 2\kappa\imag(\lambda)IY_{i} &= 0,\\
  AY_{i}'' + (B_{1}\xi + B_{2}\xi^{-1})Y_{i}' - 2\kappa\imag(\lambda)IY_{r} + (C + J_{N}(\xi) - 2\kappa\real(\lambda) I)Y_{i} &= 0,\\
  \hat{a}'' + \epsilon \hat{b}'' + \frac{d - 1}{\xi}(\hat{a}' + \epsilon \hat{b}') + \kappa \xi \hat{b}' + \frac{\kappa}{\sigma}\hat{b} + \omega \hat{a} + (\hat{a}^{2} + \hat{b}^{2})^{\sigma}\hat{a} - \delta(\hat{a}^{2} + \hat{b}^{2})^{\sigma}\hat{b} &= 0,\\
  \hat{b}'' - \epsilon \hat{a}'' + \frac{d - 1}{\xi}(\hat{b}' - \epsilon \hat{a}') - \kappa \xi \hat{a}' - \frac{\kappa}{\sigma}\hat{a} + \omega \hat{b} + (\hat{a}^{2} + \hat{b}^{2})^{\sigma}\hat{b} + \delta(\hat{a}^{2} + \hat{b}^{2})^{\sigma}\hat{a} &= 0,
\end{align*}
with initial conditions coming from the Taylor expansion on
\([0, \xi_{0}]\). This is then written as a 12-dimensional scalar
system of first-order equations before being implemented in CAPD.

To handle the removable singularity at \(\xi = 0\), we expand
\(Y\) in a Taylor series as
\begin{equation*}
  Y = \sum_{n = 0}^{\infty} Y_{n}\xi^{n}.
\end{equation*}
We also recall that, as discussed in
Section~\ref{sec:forward-solution-zero}, we can expand the real and
imaginary parts, \(\hat{a}\) and \(\hat{b}\), of \(\hat{Q}\) as
\begin{equation*}
  \hat{a} = \sum_{n = 0}^{\infty} \hat{a}_{n}\xi^{n}\quad\text{and}\quad
  \hat{b} = \sum_{n = 0}^{\infty} \hat{b}_{n}\xi^{n}.
\end{equation*}

Inserting the expansion of \(Y\) into Equation~\eqref{eq:Y-zero} gives
us the recurrence relation
\begin{equation*}
  (n + 2)((n + 1)A + B_{2})Y_{n + 2} = -(nB_{1} + C - 2\kappa\lambda I)Y_{n} - v_{n},
\end{equation*}
for even \(n \geq 0\), and \(Y_{n} = 0\) for odd \(n\). Here,
\(v_{n} = (J_{N}Y)_{n}\) and we have
\begin{equation}
  \label{eq:Y-zero-recurrence-1}
  ((n + 2)((n + 1)A + B_{2}))^{-1}
  = \frac{1}{(n + 2)(n + d)(1 + \epsilon^{2})}
  \begin{pmatrix}
    \epsilon & 1 \\ -1 & \epsilon
  \end{pmatrix}
\end{equation}
and
\begin{multline}
  \label{eq:Y-zero-recurrence-2}
  ((n + 2)((n + 1)A + B_{2}))^{-1}(nB_{1} + C - 2\kappa\lambda I)\\
  = \frac{1}{(n + 2)(n + d)(1 + \epsilon^{2})}
  \begin{pmatrix}
    \epsilon\kappa\left(n + \frac{1}{\sigma} - 2\lambda\right) + \omega
    & \kappa\left(n + \frac{1}{\sigma} - 2\lambda\right) - \epsilon \omega\\
    -\kappa\left(n + \frac{1}{\sigma} - 2\lambda\right) + \epsilon \omega
    & \epsilon\kappa\left(n + \frac{1}{\sigma} - 2\lambda\right) + \omega
  \end{pmatrix}.
\end{multline}

To bound the remainder term, we will show that, for \(n > N\), we have
\(|Y_{n}|_{\infty} \leq r^{n}\) for some (large) integer \(N\) and
some \(r > 0\). Here, we use \(|\cdot|_{\infty}\) to denote the
supremum norm of a vector.

This allows us to bound the remainder term in \(Y\). For
\(0 < \xi < \frac{1}{r}\), we have
\begin{equation*}
  \left|\sum_{n = N + 1}^{\infty} Y_{n}\xi^{n}\right|_{\infty}
  \leq \sum_{n = N + 1}^{\infty}(r\xi)^{n}
  = \frac{(r\xi)^{N + 1}}{1 - r\xi},
\end{equation*}
whereas the derivative satisfies
\begin{equation*}
  \left|\frac{d}{d\xi}\left(\sum_{n = N + 1}^{\infty} Y_{n}\xi^{n}\right)\right|_{\infty}
  = \left|\sum_{n = N + 1}^{\infty} nY_{n}\xi^{n - 1}\right|_{\infty}
  \leq \frac{1}{\xi}\sum_{n = N + 1}^{\infty}n(r\xi)^{n}
  = \frac{r(r\xi)^{N}(N + 1 - Nr\xi)}{(1 - r\xi)^{2}}.
\end{equation*}

To find \(r\) and \(N\) for \(Y\), we have the following lemma, which
is valid for \(\sigma = 1\). It is closely related
to~\cite[Lemma~8.1]{Dahne2024}.

\begin{lemma}\label{lemma:tail-bound-Y}
  Let \(\sigma = 1\). Let \(M\), \(N\), \(\tilde{C}\) and \(r\) be such
  that \(N\) is even, \(3M < N\),
  \begin{equation*}
    |\hat{a}_{n}|, |\hat{b}_{n}|, |Y_{n}|_{\infty} \leq \tilde{C} r^{n} \quad\text{for } n < M
  \end{equation*}
  and
  \begin{equation*}
    |\hat{a}_{n}|, |\hat{b}_{n}|, |Y_{n}|_{\infty} \leq r^{n} \quad\text{for } M \leq n \leq N.
  \end{equation*}
  If
  \begin{multline}
    \label{eq:tail-bound-inequality-Y}
    \frac{1 + |\epsilon|}{1 + \epsilon^{2}}\Bigg(
      \frac{|\kappa|}{N + d}
      + \frac{|\omega| + 2|\kappa||\lambda|}{(N + 2)(N + d)}\\
      + 6(1 + |\delta|)\left(
        \frac{1}{8} + \frac{1}{2N}
        + \frac{3m\tilde{C}(1 + 3/N)}{2(N + d)}
        + \frac{3m^{2}\tilde{C}^{2}}{(N + 2)(N + d)}
      \right)
    \Bigg) \leq r^{2},
  \end{multline}
  where \(m = \lceil M / 2 \rceil\), then
  \begin{equation*}
    |Y_{n}|_{\infty} \leq r^{n} \quad\text{for } n > N.
  \end{equation*}
\end{lemma}

\begin{proof}
  By~\cite[Lemma~8.1]{Dahne2024}, we have, under these assumptions,
  that
  \begin{equation*}
    |\hat{a}_{n}|, |\hat{b}_{n}| \leq r^{n} \quad\text{for } n > N.
  \end{equation*}
  Since \(N\) is even, we have \(|Y_{N + 1}|_{\infty} = 0\). It
  therefore suffices to show that
  \begin{equation*}
    |Y_{N + 2}|_{\infty} \leq r^{N + 2}.
  \end{equation*}
  The result then follows for \(n > N + 2\) by induction since the
  left-hand side of~\eqref{eq:tail-bound-inequality-Y} is decreasing
  in \(N\).

  We start by computing a bound for \(v_{N}\). For \(\sigma = 1\), we
  get that
  \begin{equation*}
    J_{N}(\xi) =
    \begin{pmatrix}
      -(3\delta \hat{a}^{2} + 2\hat{a}\hat{b} + \delta \hat{b}^{2}) & -(\hat{a}^{2} + 2\delta \hat{a}\hat{b} + 3\hat{b}^{2})\\
      3\hat{a}^{2} - 2\delta \hat{a}\hat{b} + \hat{b}^{2} & -(\delta \hat{a}^{2} - 2\hat{a}\hat{b} + 3\delta \hat{b}^{2})
    \end{pmatrix}.
  \end{equation*}
  Let \(Y^{(1)}\) and \(Y^{(2)}\) denote the two components of \(Y\).
  Then
  \begin{equation*}
    J_{N}Y =
    \begin{pmatrix}
      -(3\delta \hat{a}^{2} + 2\hat{a}\hat{b} + \delta \hat{b}^{2})Y^{(1)} - (\hat{a}^{2} + 2\delta \hat{a}\hat{b} + 3\hat{b}^{2})Y^{(2)}\\
      (3\hat{a}^{2} - 2\delta \hat{a}\hat{b} + \hat{b}^{2})Y^{(1)} - (\delta \hat{a}^{2} - 2\hat{a}\hat{b} + 3\delta \hat{b}^{2})Y^{(2)}
    \end{pmatrix}.
  \end{equation*}
  The bounds for the coefficients of \(\hat{a}\), \(\hat{b}\),
  \(Y^{(1)}\) and \(Y^{(2)}\) are all identical, so all the terms in
  \(J_{N}Y\) can be bounded in the same way. We make the computations
  for the term \(\hat{a}\hat{b}Y^{(1)}\), but the other terms follow
  analogously (the approach is identical to the one for bounding
  \((a^{3})_{N}\) in~\cite[Lemma~8.1]{Dahne2024}).

  We have
  \begin{equation*}
    (\hat{a}\hat{b}Y^{(1)})_{N} = \sum_{i + j + k = N}\hat{a}_{i}\hat{b}_{j}Y^{(1)}_{k}.
  \end{equation*}
  However, since \(\hat{a}_{n}\), \(\hat{b}_{n}\) and \(Y_{n}\) all
  vanish for odd \(n\) and \(N\) is even, we can limit the sum to
  \begin{equation*}
    (\hat{a}\hat{b}Y^{(1)})_{N} = \sum_{i + j + k = N/2}\hat{a}_{2i}\hat{b}_{2j}Y^{(1)}_{2k}.
  \end{equation*}
  Since the bounds depend on whether \(n\) is less than \(M\), let
  us group the terms in the sum by how many of the indices \(i\),
  \(j\) and \(k\) are less than \(m = \lceil M / 2 \rceil\):
  \begin{enumerate}
  \item The number of terms with no indices less than \(m\) is
    \begin{equation*}
      T_{0} = \sum_{i = 0}^{N/2 - 3m}\sum_{j = 0}^{N/2 - 3m - i} 1 = \frac{(N/2 - 3m + 1)(N/2 - 3m + 2)}{2}.
    \end{equation*}
    These are bounded by \(r^{N}\).
  \item For the terms with exactly one index less than \(m\), there are
    \(3\) choices for which index is the small one. If we let \(i\) be
    the small index, then there are \(N/2 - 2m - i + 1\) choices for
    \(j\) and \(k\). Summing over all choices for \(i\) and
    multiplying by \(3\), we get
    \begin{equation*}
      T_{1} = 3\sum_{i = 0}^{m - 1} \left(N/2 - 2m - i + 1\right) = 3\frac{m}{2}\left(N - 5m + 3\right).
    \end{equation*}
    These are bounded by \(\tilde{C}r^{N}\).
  \item The number of terms with exactly two indices less than \(m\)
    is \(T_{2} = 3m^{2}\). These are bounded by \(\tilde{C}^{2}r^{N}\).
  \end{enumerate}
  Note that since \(3M < N\), there are no terms with three indices
  less than \(m\). In total, this yields the bound
  \begin{equation*}
    |(\hat{a}\hat{b}Y^{(1)})_{N}| \leq (T_{0} + T_{1}\tilde{C} + T_{2}\tilde{C}^{2})r^{N}.
  \end{equation*}
  Since the bounds for all the other terms in \(v\) are the same, we
  get
  \begin{equation*}
    |v_{N}|_{\infty}
    \leq 6(1 + |\delta|)(T_{0} + T_{1}\tilde{C} + T_{2}\tilde{C}^{2})r^{N}.
  \end{equation*}

  Recall that
  \begin{equation*}
    Y_{N + 2} = -((N + 2)((N + 1)A + B_{2}))^{-1}(NB_{1} + C - 2\kappa\lambda I)Y_{N} - ((N + 2)((N + 1)A + B_{2}))^{-1}v_{N}.
  \end{equation*}
  This gives
  \begin{multline*}
    |Y_{N + 2}|_{\infty} \leq \|((N + 2)((N + 1)A + B_{2}))^{-1}(NB_{1} + C - 2\kappa\lambda I)\|_{\infty}|Y_{N}|_{\infty}\\
    + \|((N + 2)((N + 1)A + B_{2}))^{-1}\|_{\infty}|v_{N}|_{\infty},
  \end{multline*}
  where \(\|\cdot\|_{\infty}\) denotes the operator norm on matrices
  induced by the supremum norm on vectors (given by the maximum
  absolute row sum). From Equations~\eqref{eq:Y-zero-recurrence-1}
  and~\eqref{eq:Y-zero-recurrence-2}, a straightforward computation
  gives
  \begin{equation*}
    \|((N + 2)((N + 1)A + B_{2}))^{-1}\|_{\infty}
    \leq \frac{1 + |\epsilon|}{(N + 2)(N + d)(1 + \epsilon^{2})},
  \end{equation*}
  and
  \begin{equation*}
    \|((N + 2)((N + 1)A + B_{2}))^{-1}(NB_{1} + C - 2\kappa\lambda I)\|_{\infty}
    \leq \frac{(1 + |\epsilon|)(|\kappa|(N + 1 + 2|\lambda|) + |\omega|)}{(N + 2)(N + d)(1 + \epsilon^{2})}.
  \end{equation*}
  Combining this with the bounds for \(|Y_{N}|_{\infty}\) and
  \(|v_{N}|_{\infty}\), we get
  \begin{equation*}
    |Y_{N + 2}|_{\infty} \leq
    \frac{1 + |\epsilon|}{(N + 2)(N + d)(1 + \epsilon^{2})}\left(
    |\kappa|(N + 1 + 2|\lambda|) + |\omega|
    + 6(1 + |\delta|)(T_{0} + T_{1}\tilde{C} + T_{2}\tilde{C}^{2})
    \right)r^{N}.
  \end{equation*}
  We have
  \begin{equation*}
    \begin{split}
      \frac{T_{0}}{(N + 2)(N + d)}
      &= \frac{(N/2 - 3m + 1)(N/2 - 3m + 2)}{2(N + 2)(N + d)}\\
      &= \frac{1}{8}\frac{(N - 6m + 2)(N - 6m + 4)}{(N + 2)(N + d)}\\
      &\leq \frac{1}{8}\frac{N + 4}{N + d}
        \leq \frac{1}{8}\left(1 + \frac{4}{N}\right)
        = \frac{1}{8} + \frac{1}{2N},
    \end{split}
  \end{equation*}
  as well as
  \begin{equation*}
    \frac{T_{1}}{(N + 2)(N + d)}
    = 3\frac{m}{2}\frac{N - 5m + 3}{(N + 2)(N + d)}
    \leq 3\frac{m}{2}\frac{N + 3}{(N + 2)(N + d)}
    \leq \frac{3m(1 + 3/N)}{2(N + d)}.
  \end{equation*}
  This gives us
  \begin{multline*}
    |Y_{N + 2}|_{\infty} \leq
    \frac{1 + |\epsilon|}{1 + \epsilon^{2}}\Bigg(
    \frac{|\kappa|}{N + d}
    + \frac{|\omega| + 2|\kappa||\lambda|}{(N + 2)(N + d)}\\
    + 6(1 + |\delta|)\left(
        \frac{1}{8} + \frac{1}{2N}
        + \frac{3m\tilde{C}(1 + 3/N)}{2(N + d)}
        + \frac{3m^{2}\tilde{C}^{2}}{(N + 2)(N + d)}
      \right)
    \Bigg)r^{N}.
  \end{multline*}

  By~\eqref{eq:tail-bound-inequality-Y}, the factor in front of
  \(r^{N}\) is bounded by \(r^{2}\), giving us the required bound
  \(|Y_{N + 2}|_{\infty} \leq r^{N + 2}\).
\end{proof}

\section{Implementation details}
\label{sec:implementation-details}
The full code on which the computer-assisted parts of the proofs are
based, as well as notebooks presenting the results and figures, is
available in the repository~\cite{CGL2.jl}. The code builds on top of
the code from~\cite{Dahne2024}, found in the repository~\cite{CGL.jl}.
Below we briefly describe what has been kept in the new code, what has
been removed and what has been added. For more details we refer to the
repository itself as well as~\cite{Dahne2024}.

More details about the structure of the code and exactly how the
results are generated can be found in that repository. The majority of
the code is implemented in Julia~\cite{Julia-2017}, except for the
parts related to CAPD~\cite{Kapela2021}, which are implemented in C++.
The rigorous parts, on which the computer-assisted results are based,
make use of the CAPD~\cite{Kapela2021} library and the
Arb~\cite{Johansson2017arb}/FLINT~\cite{Flint} library\footnote{In
  2023 Arb was merged with the FLINT library.}. The CAPD library is
used for the rigorous numerical integration for enclosing \(Q_{0}\),
\(\hat{Q}_{0}\) and \(Y_{0}\), and the Arb library is used for
everything else. The Arb library is used through the Julia wrapper
\texttt{Arblib.jl}~\cite{Arblib.jl}, with many of the basic interval
arithmetic algorithms, such as isolating roots or enclosing maximum
values of single-variable functions, implemented in
\texttt{ArbExtras.jl}~\cite{ArbExtras.jl}.

Compared to the code in~\cite{CGL.jl}, the whole implementation of
\(G\) stays the same, as does the implementation of the Krawczyk
interval Newton method and the handling of confluent hypergeometric
functions. Parts that have been removed include, in particular, the
entire logic for the branch verification and the code for counting
critical points. All HPC related functionality has also been removed,
as the computations required for this paper are well within the
capabilities of a regular desktop.

The additions to the code are the implementation of the functions
\(\hat{G}\) and \(H\), which in turn relies on evaluation of
\(\hat{Q}_{0}\), \(\hat{Q}_{\infty}\), \(Y_{0}\) and \(Y_{\infty}\),
as well as the finite-difference scheme described in
Appendix~\ref{sec:finite-differences-approximations}.

\pdfbookmark[1]{Acknowledgments}{sec:acknowledgments}
\section*{Acknowledgments}
The authors are grateful for discussions with Hao Jia regarding the
mechanism for non-uniqueness from an unstable eigenvalue.

\begin{description}
\item[Funding] Vladimír Šverák was supported in part by NSF
  DMS-2247027 and NSF DMS-2553691.
\item[Data availability] The code on which the computer-assisted parts
  of the proofs are based and notebooks generating the results and
  figures are available at the following URL:
  \url{https://github.com/Joel-Dahne/CGL2.jl}.
\item[Declaration of AI usage] AI, primarily Claude Opus 5, has been
  used as part of proofreading the manuscript, generating some of the
  initial estimates in
  Section~\ref{sec:non-uniqueness-from-unstable-eigenvalue} and
  reviewing the code. The authors take full responsibility for the
  content of both the manuscript and the code.
\end{description}

\appendix

\section{Finite-difference approximations of eigenvalues}
\label{sec:finite-differences-approximations}

To compute an initial approximation of an unstable eigenvalue of
\(L_{\hat{Q}}\), we make use of a finite-difference scheme. For this,
we use the representation
\begin{equation*}
  L_{\hat{Q}}Y = \frac{1}{2\kappa}\left[AY'' + (B_{1}\xi + B_{2}\xi^{-1})Y' + (C + J_{N}(\xi))Y\right]
\end{equation*}
with \(A\), \(B_{1}\), \(B_{2}\), \(C\) and \(J_{N}\) as
in~\eqref{eq:ABC-sigma-d} and~\eqref{eq:J_N-sigma-d}.

To discretize the spatial coordinate, we use a uniform grid given by
\begin{equation*}
  0 = x_{0}, x_{1}, \dots, x_{n}, x_{n + 1} = \xi_{1},
\end{equation*}
and denote the step size by \(h\). We let \(Y_{i} = Y(x_{i})\). We
have a Neumann condition at zero and a Dirichlet condition at the
truncation point \(\xi_{1}\), giving us \(Y_{0} = Y_{1}\) and
\(Y_{n + 1} = 0\). On the grid points, the discretized derivatives are
given by
\begin{equation*}
  Y_{i}' = \frac{Y_{i + 1} - Y_{i - 1}}{2h},\quad
  Y_{i}'' = \frac{Y_{i + 1} - 2Y_{i} + Y_{i - 1}}{h^{2}}.
\end{equation*}
If we let
\begin{equation*}
  \bm{Y} =
  \begin{pmatrix}
    Y_{1} \\ \vdots \\ Y_{n}
  \end{pmatrix},
\end{equation*}
then the derivatives can be written using the block matrices
\begin{equation*}
  D^{(1)} =
  \frac{1}{2h}\begin{pmatrix}
    -I & I & 0 & 0 & 0 & \cdots & 0\\
    -I & 0 & I & 0 & 0 & \cdots & 0\\
    0 & -I & 0 & I & 0 & \cdots & 0\\
    \vdots & & & & & & \vdots \\
    0 & 0 & \cdots & 0 & -I & 0 & I\\
    0 & 0 & \cdots & 0 & 0 & -I & 0
  \end{pmatrix},\quad
  D^{(2)} =
  \frac{1}{h^{2}}\begin{pmatrix}
    -I & I & 0 & 0 & 0 & \cdots & 0\\
    I & -2I & I & 0 & 0 & \cdots & 0\\
    0 & I & -2I & I & 0 & \cdots & 0\\
    \vdots & & & & & & \vdots \\
    0 & 0 & \cdots & 0 & I & -2I & I\\
    0 & 0 & \cdots & 0 & 0 & I & -2I
  \end{pmatrix}.
\end{equation*}
Note the adjustments to the top-left entries, which come from the
Neumann condition \(Y_{0} = Y_{1}\). The discretized operator
\(L_{n}\) can then be written as the sum of block matrices
\begin{equation*}
  L_{n} = \frac{1}{2\kappa}\left[A_{n}D^{(2)} + (B_{1,n}X_{n} + B_{2,n}X_{n}^{-1})D^{(1)} + C_{n} + J_{N,n}\right].
\end{equation*}
Here \(A_{n}\), \(B_{1,n}\), \(B_{2,n}\) and \(C_{n}\) are the
\(2n \times 2n\) block-diagonal matrices formed by taking \(n\) copies
of \(A\), \(B_{1}\), \(B_{2}\) and \(C\), respectively, on the
diagonal. Similarly, \(X_{n}\) and \(J_{N,n}\) are the block-diagonal
matrices given by
\begin{equation*}
  X_{n} =
  \begin{pmatrix}
    \left(\begin{smallmatrix} x_{1} & 0 \\ 0 & x_{1} \end{smallmatrix}\right) & & 0 \\
    & \ddots & \\
    0 & & \left(\begin{smallmatrix} x_{n} & 0 \\ 0 & x_{n} \end{smallmatrix}\right)
  \end{pmatrix},\quad
  J_{N,n} =
  \begin{pmatrix}
    J_{N}(x_{1}) & & 0 \\
    & \ddots & \\
    0 & & J_{N}(x_{n})
  \end{pmatrix}.
\end{equation*}
Eigenvalues of \(L_{n}\) are then computed using the ARPACK sparse
eigenvalue solver~\cite{Lehoucq1998}.

\section{Asymptotic behavior of the backward self-similar solution}
\label{sec:improvements-backward}

Compared to~\cite{Dahne2024}, we need better control of the leading
asymptotic behavior of the backward self-similar solution \(Q\) at
infinity. Recall that \(Q\) is a solution to the equation
\begin{equation*}
  Q(\xi) = \gamma P(\xi) + P(\xi)I_{E}(Q, \xi) + E(\xi)I_{P}(Q, \xi),
\end{equation*}
where
\begin{equation*}
  I_{P}(Q, \xi) = \int_{\xi}^{\infty} J_{P}(\eta)|Q(\eta)|^{2\sigma}Q(\eta)\,d\eta,\quad
  I_{E}(Q, \xi) = \int_{\xi_{1}}^{\xi} J_{E}(\eta)|Q(\eta)|^{2\sigma}Q(\eta)\,d\eta
\end{equation*}
with
\begin{equation*}
  J_{P}(\xi) = B_{W}P(\xi)e^{-c\xi^{2}}\xi^{d - 1}
  \quad\text{and}\quad
  J_{E}(\xi) = B_{W}E(\xi)e^{-c\xi^{2}}\xi^{d - 1},
\end{equation*}
where
\begin{equation*}
  B_{W} = -\frac{1 + i\delta}{i\kappa}e^{-\sign(\imag(c))\pi i(b - a)}c^{b}.
\end{equation*}
By~\cite[Lemma~7.3]{Dahne2024}, we have the bounds
\begin{equation*}
  |P(\xi)| \leq C_{P}\xi^{-\frac{1}{\sigma}}
  \quad\text{and}\quad
  |E(\xi)| \leq C_{E}e^{\real(c)\xi^{2}}\xi^{\frac{1}{\sigma} - d},
\end{equation*}
as well as
\begin{equation*}
  |J_{P}(\xi)| \leq C_{J_{P}}e^{-\real(c)\xi^{2}}\xi^{-\frac{1}{\sigma} + d - 1}
  \quad\text{and}\quad
  |J_{E}(\xi)| \leq C_{J_{E}}\xi^{\frac{1}{\sigma} - 1}.
\end{equation*}
From~\cite[Lemma~7.4]{Dahne2024} with \(\normv = 0\), we further get
\begin{align*}
  |I_{E}(\xi)| &\leq C_{I_{E}}\|Q\|_{0}^{2\sigma + 1}\xi_{1}^{-2},\\
  |I_{P}(\xi)| &\leq C_{I_{P}}\|Q\|_{0}^{2\sigma + 1}e^{-\real(c)\xi^{2}}\xi^{-\frac{2}{\sigma} + d - 2}.
\end{align*}

Since \(I_{E}\) does not vanish in the limit \(\xi \to \infty\), the
leading asymptotic behavior of \(Q\) is not given by the term
\(\gamma P\) alone. Instead, we have the following lemma.

\begin{lemma}\label{lemma:p_Q_0}
  We have
  \begin{equation*}
    Q(\xi) = c^{-a}p_{Q,0}\xi^{-2a} + \mathcal{O}(\xi^{-\frac{1}{\sigma} - 2}),
  \end{equation*}
  with
  \begin{equation*}
    p_{Q,0} = \gamma + I_{E,\infty},
  \end{equation*}
  and
  \begin{equation*}
    I_{E,\infty} = \int_{\xi_{1}}^{\infty} J_{E}(\eta)|Q(\eta)|^{2\sigma}Q(\eta)\,d\eta.
  \end{equation*}
\end{lemma}

\begin{proof}
  We have
  \begin{equation*}
    Q(\xi) = P(\xi)(\gamma + I_{E}(\xi)) + E(\xi)I_{P}(\xi).
  \end{equation*}
  Let us start with the second term. Combining the above bounds for
  \(E(\xi)\) and \(I_{P}\), we get
  \begin{equation*}
    |E(\xi)I_{P}(\xi)| \leq C_{E}C_{I_{P}}\|Q\|_{0}^{2\sigma + 1}
    \xi^{-\frac{1}{\sigma} - 2},
  \end{equation*}
  and hence
  \(E(\xi)I_{P}(\xi) = \mathcal{O}(\xi^{-\frac{1}{\sigma} - 2})\).

  For the first term, we note that
  \(P(\xi) = c^{-a}\xi^{-2a} + \mathcal{O}(\xi^{-\frac{1}{\sigma} -
    2})\). Since \(\real(2a) = \frac{1}{\sigma}\), it therefore suffices
  to prove that \(I_{E}(\xi) = I_{E,\infty} + \mathcal{O}(\xi^{-2})\).
  We have
  \begin{equation*}
    |I_{E}(\xi) - I_{E,\infty}|
    \leq \int_{\xi}^{\infty}|J_{E}(\eta)||Q(\eta)|^{2\sigma + 1}\,d\eta
    \leq C_{J_{E}}\|Q\|_{0}^{2\sigma + 1}\int_{\xi}^{\infty}\eta^{-3}\,d\eta
    = \frac{C_{J_{E}}}{2}\|Q\|_{0}^{2\sigma + 1}\xi^{-2},
  \end{equation*}
  which is indeed \(\mathcal{O}(\xi^{-2})\).
\end{proof}

To compute an enclosure of \(I_{E,\infty}\), we have the following
lemma. We specialize to \(\sigma = 1\) to simplify the treatment
of the remainder terms.

\begin{lemma}\label{lemma:I_E_infty}
  Assume that \(\sigma = 1\). Then
  \begin{equation*}
    I_{E,\infty}
    = |\gamma|^{2}\gamma\int_{\xi_{1}}^{\infty} J_{E}(\eta)|P(\eta)|^{2}P(\eta)\,d\eta
    + R_{I_{E,\infty}},
  \end{equation*}
  with
  \begin{equation*}
    |R_{I_{E,\infty}}|
    \leq \frac{C_{J_{E}}(C_{R,1} + C_{R,2})}{4}
    \left(3|\gamma|^{2}C_{P}^{2} + 3|\gamma|C_{P}A_{\max} + A_{\max}^{2}\right)
    \xi_{1}^{-4},
  \end{equation*}
  and
  \begin{equation*}
    C_{R,1} = \frac{C_{P}C_{J_{E}}}{2}\|Q\|_{0}^{3},\quad
    C_{R,2} = C_{E}C_{I_{P}}\|Q\|_{0}^{3}\quad\text{and}\quad
    A_{\max} = \max(C_{R,1}, C_{R,2})\xi_{1}^{-2}.
  \end{equation*}
\end{lemma}

\begin{proof}
  Define
  \(R_{Q}(\xi) = P(\xi)I_{E}(\xi) + E(\xi)I_{P}(\xi)\), so that
  \(Q(\xi) = \gamma P(\xi) + R_{Q}(\xi)\). Rather than using the bound
  \(|I_{E}(\xi)| \leq C_{I_{E}}\|Q\|_{0}^{2\sigma + 1}\xi_{1}^{-2}\),
  which only captures the limiting value of \(I_{E}\) as
  \(\xi \to \infty\), we keep track of the growth of \(I_{E}\) on
  \([\xi_{1}, \infty)\). Proceeding as in the proof of
  Lemma~\ref{lemma:p_Q_0}, but integrating over \([\xi_{1}, \xi]\)
  instead of \([\xi, \infty)\), we get
  \begin{equation*}
    |I_{E}(\xi)|
    \leq C_{J_{E}}\|Q\|_{0}^{3}\int_{\xi_{1}}^{\xi}\eta^{-3}\,d\eta
    = \frac{C_{J_{E}}}{2}\|Q\|_{0}^{3}\left(\xi_{1}^{-2} - \xi^{-2}\right).
  \end{equation*}
  Together with the above bounds for \(P(\xi)\), \(E(\xi)\) and
  \(I_{P}(\xi)\), this gives
  \begin{equation*}
    \begin{split}
      |R_{Q}(\xi)|
      &\leq
        \frac{C_{P}C_{J_{E}}}{2}\|Q\|_{0}^{3}\left(\xi_{1}^{-2} - \xi^{-2}\right)\xi^{-1}
        + C_{E}C_{I_{P}}\|Q\|_{0}^{3}\xi^{-3}\\
      &= A(\xi)\xi^{-1},
    \end{split}
  \end{equation*}
  with
  \begin{equation*}
    A(\xi) = C_{R,1}\left(\xi_{1}^{-2} - \xi^{-2}\right) + C_{R,2}\xi^{-2}
    = C_{R,1}\xi_{1}^{-2} + (C_{R,2} - C_{R,1})\xi^{-2}.
  \end{equation*}
  From the second form, we see that \(A\) is monotone on
  \([\xi_{1}, \infty)\), taking values between
  \(A(\xi_{1}) = C_{R,2}\xi_{1}^{-2}\) and
  \(\lim_{\xi \to \infty} A(\xi) = C_{R,1}\xi_{1}^{-2}\), so that
  \(A(\xi) \leq A_{\max}\).

  With \(\sigma = 1\), we have
  \begin{multline*}
    |Q(\xi)|^{2\sigma}Q(\xi)
    = |\gamma|^{2}\gamma|P(\xi)|^{2}P(\xi)
    + 2|\gamma|^{2}|P(\xi)|^{2} R_{Q}(\xi)
    + \gamma^{2}P(\xi)^{2}\conj{R_{Q}(\xi)}\\
    + 2\gamma P(\xi) |R_{Q}(\xi)|^{2}
    + \conj{\gamma}\conj{P(\xi)} R_{Q}(\xi)^{2}
    + |R_{Q}(\xi)|^{2} R_{Q}(\xi).
  \end{multline*}
  Let \(R_{Q,N}\) denote the sum of the remainder terms, so that
  \begin{equation*}
    |Q(\xi)|^{2\sigma}Q(\xi) = |\gamma|^{2}\gamma|P(\xi)|^{2}P(\xi)
    + R_{Q,N}(\xi).
  \end{equation*}
  Substituting this into \(I_{E,\infty}\) yields
  \begin{equation*}
    I_{E,\infty}
    = |\gamma|^{2}\gamma\int_{\xi_{1}}^{\infty} J_{E}(\eta)|P(\eta)|^{2}P(\eta)\,d\eta
    + \int_{\xi_{1}}^{\infty} J_{E}(\eta)R_{Q,N}(\eta)\,d\eta.
  \end{equation*}
  To bound the remainder integral, we note that the bounds for \(P\)
  and \(R_{Q}\) yield
  \begin{equation*}
    \begin{split}
      |R_{Q,N}(\xi)|
      &\leq \left(3|\gamma|^{2}C_{P}^{2} + 3|\gamma|C_{P}A(\xi) + A(\xi)^{2}\right)A(\xi)\xi^{-3}\\
      &\leq \left(3|\gamma|^{2}C_{P}^{2} + 3|\gamma|C_{P}A_{\max} + A_{\max}^{2}\right)A(\xi)\xi^{-3}.
    \end{split}
  \end{equation*}
  Inserting this back into the integral and combining it with the bound
  \(|J_{E}(\xi)| \leq C_{J_{E}}\xi^{\frac{1}{\sigma} - 1} = C_{J_{E}}\)
  from~\cite[Lemma~7.3]{Dahne2024}, we get
  \begin{equation*}
    \left|\int_{\xi_{1}}^{\infty} J_{E}(\eta)R_{Q,N}(\eta)\,d\eta\right|
    \leq C_{J_{E}}\left(3|\gamma|^{2}C_{P}^{2} + 3|\gamma|C_{P}A_{\max} + A_{\max}^{2}\right)
    \int_{\xi_{1}}^{\infty} A(\eta)\eta^{-3}\,d\eta.
  \end{equation*}
  The remaining integral is given by
  \begin{equation*}
    \begin{split}
      \int_{\xi_{1}}^{\infty} A(\eta)\eta^{-3}\,d\eta
      &= C_{R,1}\int_{\xi_{1}}^{\infty}\left(\xi_{1}^{-2}\eta^{-3} - \eta^{-5}\right)\,d\eta
        + C_{R,2}\int_{\xi_{1}}^{\infty}\eta^{-5}\,d\eta\\
      &= C_{R,1}\left(\frac{\xi_{1}^{-4}}{2} - \frac{\xi_{1}^{-4}}{4}\right)
        + C_{R,2}\frac{\xi_{1}^{-4}}{4}
      = \frac{C_{R,1} + C_{R,2}}{4}\xi_{1}^{-4},
    \end{split}
  \end{equation*}
  which gives the stated bound.
\end{proof}

To enclose the integral
\begin{equation*}
  \int_{\xi_{1}}^{\infty} J_{E}(\eta)|P(\eta)|^{2}P(\eta)\,d\eta,
\end{equation*}
we follow the same approach as for
Lemma~\ref{lemma:integral-J_E-hat-P-hat}. In this case, we have the
following version.

\begin{lemma}\label{lemma:I_E_infty-integral}
  We have
  \begin{equation*}
    \int_{\xi_{1}}^{\infty} J_{E}(\eta)|P(\eta)|^{2}P(\eta)\,d\eta
    = B_{W}(-c)^{a - b}|c^{-a}|^{2}c^{-a}I_{U},
  \end{equation*}
  with
  \begin{equation*}
    \begin{split}
      I_{U}
      = \int_{\xi_{1}}^{\infty}
      &\left(
        \sum_{k = 0}^{n - 1} \frac{(b - a)_{k}(-a + 1)_{k}}{k!c^{k}}\eta^{-2k}
        + R_{U}(b - a, b, n, -c\eta^{2})(-c)^{-n}\eta^{-2n}
        \right)\\
      &\quad\left(
        \sum_{k = 0}^{n - 1} \frac{(a)_{k}(a - b + 1)_{k}}{k!(-c)^{k}}\eta^{-2k}
        + R_{U}(a, b, n, c\eta^{2})c^{-n}\eta^{-2n}
        \right)^{2}\\
      &\quad\left(
        \sum_{k = 0}^{n - 1} \conj{\left(\frac{(a)_{k}(a - b + 1)_{k}}{k!(-c)^{k}}\right)}\eta^{-2k}
        + \conj{R_{U}(a, b, n, c\eta^{2})c^{-n}}\eta^{-2n}
        \right)\eta^{-\frac{2}{\sigma} - 1}\,d\eta.
    \end{split}
  \end{equation*}
  Here \(R_{U}\) is the remainder term in the asymptotic expansion of
  \(U\) from~\cite[Lemma~7.1]{Dahne2024} and \(n\) is any
  non-negative integer.
\end{lemma}

\begin{proof}
  Since
  \begin{equation*}
    J_{E}(\xi) = B_{W}E(\xi)e^{-c\xi^{2}}\xi^{d - 1}
    = B_{W}U(b - a, b, -c\xi^{2})\xi^{d - 1},
  \end{equation*}
  we can write the integral as
  \begin{equation*}
    \int_{\xi_{1}}^{\infty} J_{E}(\eta)|P(\eta)|^{2}P(\eta)\,d\eta
    = B_{W}\int_{\xi_{1}}^{\infty} U(b - a, b, -c\eta^{2})P(\eta)^{2}\conj{P(\eta)}\eta^{d - 1}\,d\eta.
  \end{equation*}
  Using the expansions
  \begin{equation*}
    \begin{split}
      P(\xi)
      = U(a, b, c\xi^{2})
      &= \left(
        \sum_{k = 0}^{n - 1} \frac{(a)_{k}(a - b + 1)_{k}}{k!(-c\xi^{2})^{k}}
        + R_{U}(a, b, n, c\xi^{2})\left(c\xi^{2}\right)^{-n}
        \right)\left(c\xi^{2}\right)^{-a}\\
      &= \left(
        \sum_{k = 0}^{n - 1} \frac{(a)_{k}(a - b + 1)_{k}}{k!(-c)^{k}}\xi^{-2k}
        + R_{U}(a, b, n, c\xi^{2})c^{-n}\xi^{-2n}
        \right)c^{-a}\xi^{-2a},
    \end{split}
  \end{equation*}
  and
  \begin{equation*}
    \begin{split}
      U(b - a, b, -c\xi^{2})
      &= \left(
        \sum_{k = 0}^{n - 1} \frac{(b - a)_{k}(-a + 1)_{k}}{k!(c\xi^{2})^{k}}
        + R_{U}(b - a, b, n, -c\xi^{2})\left(-c\xi^{2}\right)^{-n}
        \right)\left(-c\xi^{2}\right)^{a - b}\\
      &= \left(
        \sum_{k = 0}^{n - 1} \frac{(b - a)_{k}(-a + 1)_{k}}{k!c^{k}}\xi^{-2k}
        + R_{U}(b - a, b, n, -c\xi^{2})(-c)^{-n}\xi^{-2n}
        \right)(-c)^{a - b}\xi^{2(a - b)},
    \end{split}
  \end{equation*}
  we can write the integral as
  \begin{equation*}
    \begin{split}
      B_{W}(-c)^{a - b}|c^{-a}|^{2}c^{-a}\int_{\xi_{1}}^{\infty}
      &\left(
        \sum_{k = 0}^{n - 1} \frac{(b - a)_{k}(-a + 1)_{k}}{k!c^{k}}\eta^{-2k}
        + R_{U}(b - a, b, n, -c\eta^{2})(-c)^{-n}\eta^{-2n}
        \right)\\
      &\quad\left(
        \sum_{k = 0}^{n - 1} \frac{(a)_{k}(a - b + 1)_{k}}{k!(-c)^{k}}\eta^{-2k}
        + R_{U}(a, b, n, c\eta^{2})c^{-n}\eta^{-2n}
        \right)^{2}\\
      &\quad\left(
        \sum_{k = 0}^{n - 1} \conj{\left(\frac{(a)_{k}(a - b + 1)_{k}}{k!(-c)^{k}}\right)}\eta^{-2k}
        + \conj{R_{U}(a, b, n, c\eta^{2})c^{-n}}\eta^{-2n}
        \right)\eta^{-\frac{2}{\sigma} - 1}\,d\eta,
    \end{split}
  \end{equation*}
  where we have used that \(2a + 2\conj{a} = \frac{2}{\sigma}\)
  and \(2b = d\).
\end{proof}

\printbibliography

\begin{tabular}{c@{\hspace{1cm}}c}
  \begin{tabular}[t]{l}
    \textbf{Joel Dahne} \\
    School of Mathematics \\
    University of Minnesota \\
    127 Vincent Hall, 206 Church St. SE \\
    Minneapolis, MN 55455, USA \\
    Email: jdahne@umn.edu
  \end{tabular}
  &
  \begin{tabular}[t]{l}
    \textbf{Vladimír Šverák} \\
    School of Mathematics \\
    University of Minnesota \\
    127 Vincent Hall, 206 Church St. SE \\
    Minneapolis, MN 55455, USA \\
    Email: sverak@umn.edu
  \end{tabular}
\end{tabular}

\end{document}